\documentclass[12pt, reqno,oneside]{amsart}

\usepackage{amssymb, amsmath, amsthm, wasysym, mathrsfs}
\usepackage[backref]{hyperref}
\usepackage[alphabetic,backrefs,lite]{amsrefs}
\usepackage{fullpage}
\usepackage{setspace}
\usepackage{mathtools}
\usepackage{enumitem}

\newtheorem{lemma}{Lemma}[section]
\newtheorem{lemma*}{Lemma}

\newtheorem{proposition}[lemma]{Proposition}
\newtheorem{prop}[lemma]{Proposition}
\newtheorem{cor}[lemma]{Corollary}
\newtheorem{conj}[lemma]{Conjecture}

\newtheorem{question}[lemma]{Question}
\newtheorem{claim*}{Claim}
\newtheorem{thm}[lemma]{Theorem}
\newtheorem{defn}[lemma]{Definition}

\theoremstyle{definition}
\newtheorem{remark}[lemma]{Remark}
\newtheorem{remarks}[lemma]{Remarks}
\newtheorem{rmk}[lemma]{Remark}

\newcommand{\G}{{\mathbb G}}

\newcommand{\PP}{{\mathbb P}}

\newcommand{\Q}{{\mathbb Q}}
\newcommand{\R}{{\mathbb R}}
\newcommand{\Z}{{\mathbb Z}}

\newcommand{\Qbar}{{\overline{\Q}}}

\newcommand{\calA}{{\mathcal A}}

\newcommand{\calD}{{\mathcal D}}

\newcommand{\calO}{{\mathcal O}}
\newcommand{\calP}{{\mathcal P}}

\newcommand{\calX}{{\mathcal X}}

\newcommand{\OO}{{\mathcal O}}

\DeclareMathOperator{\HH}{H}

\DeclareMathOperator{\Gal}{Gal}

\DeclareMathOperator{\Pic}{Pic}

\DeclareMathOperator{\Spec}{Spec}

\DeclareMathOperator{\Proj}{Proj}

\DeclareMathOperator{\Vol}{Vol}

\newcommand{\isom}{\simeq}

\newcommand{\coeffset}{\mathfrak W}

\numberwithin{equation}{section}
\numberwithin{table}{section}
\newcommand{\defi}[1]{\textsf{#1}} 

\title[Campana points]{Campana points of bounded height on vector group compactifications}
\author{Marta Pieropan}
\author{Arne Smeets}
\author{Sho Tanimoto}
\author{Anthony V\'arilly-Alvarado}

\address{Marta Pieropan, Utrecht University, Mathematical Institute, Budapestlaan 6, 3584 CD Utrecht, the Netherlands \emph{and} EPFL SB MATH CAG, B\^at. MA, Station 8, 1015 Lausanne, Switzerland}
\email{m.pieropan@uu.nl}
\urladdr{https://webspace.science.uu.nl/~piero001/}
\address{Arne Smeets, KU Leuven, Departement Wiskunde, Celestijnenlaan 200B, 3001 Heverlee, Belgium \emph{and} Radboud Universiteit Nijmegen, Heyendaalseweg 135, 6525 AJ Nijmegen, the Netherlands}
\email{arnesmeets@gmail.com}
\urladdr{https://sites.google.com/site/arnesmeets/}
\address{Sho Tanimoto, Department of Mathematics, Faculty of Science, Kumamoto University, Kurokami 2-39-1 Kumamoto 860-8555 Japan}
\address{Priority Organization for Innovation and Excellence, Kumamoto University}
\email{stanimoto@kumamoto-u.ac.jp}
\urladdr{http://shotanimoto.wordpress.com}
\address{Anthony V\'arilly-Alvarado, Department of Mathematics MS 136, Rice University, 6100 S.\ Main St., Houston, TX 77005, USA}
\email{av15@rice.edu}
\urladdr{http://math.rice.edu/\~{}av15}

\date{August 26, 2020}


\begin{document}


	\begin{abstract}
        We initiate a systematic quantitative study of subsets 
        of rational points that are integral with respect to a weighted boundary divisor 
        on Fano orbifolds. 
        We call the points in these sets Campana points.
        Earlier work of Campana and subsequently Abramovich shows that there 
        are several reasonable competing definitions  for Campana points.
        We use a version that delineates well different types of behaviour of points as 
        the weights on the boundary divisor vary. This prompts a Manin-type conjecture on 
        Fano orbifolds for sets of Campana points that satisfy a klt (Kawamata log terminal) 
        condition. By importing work of Chambert-Loir and Tschinkel to our set-up, we prove 
        a log version of Manin's conjecture for klt Campana points on equivariant 
        compactifications of vector groups.        
	\end{abstract}
	

	\maketitle
	
	\tableofcontents
	

\newpage

\section{Introduction}

	Manin's conjecture for rational points, extensively studied now for more than three 
	decades, predicts an asymptotic formula for the counting function of rational points 
	of bounded height on rationally 
	connected algebraic varieties over number fields. 
	The class of equivariant compactifications of homogeneous spaces has 
	proved to be a particularly fertile testing ground for the conjecture
	\cites{FMT89, BT96a, BT98, CLT02, STBT07,GMO08, GO11, TT12,  GTBT15, ST16}.
	The related problem of counting integral points on homogeneous spaces has received 
	much attention as well, both classically (see, for example,  \cite{DRS93, EM93}), and  
	recently, as attested by \cites{CLT10b, BO12, CLT12, TBT13, TT15, Chow}. 
	By choosing a suitable compactification, one can identify the set of integral points on 
	the original variety with the set of rational points on the compactification that are 
	integral with respect to the boundary divisor. Hence,
	this latter body 
	of work represents progress towards a ``logarithmic version'' of Manin's conjecture 
	for integral points. Regrettably, subtleties of a mostly geometric nature have so far 
	prevented a general formulation of a Manin-type conjecture for integral points.
	
	In this paper we focus on an intermediate notion: sets of rational points that are 
	\emph{integral with respect to a weighted boundary divisor} \cite{Campana15},
	which we call \emph{Campana points}. Such sets depend on the choice of weights 
	and ``interpolate''
	between the set of integral points and the set of rational points, 
	which can both be recovered as sets of Campana points for suitable choices of weights.
	If the weighted boundary divisor is \emph{Kawamata log terminal} (klt for short), 
	we say that the Campana points are klt.
	The set of rational points is a set of klt Campana points, while the set of integral points is not.
	However, the set of 
	integral points can be written 
	as an infinite intersection of sets of klt Campana points.
	
	To date, Manin-type problems for sets of Campana points have not 
	been well-studied. The only results we are aware of are to be found in 
	\cite{BVV12}, \cite{VanV12} and \cite{BY19} 
	and we believe that this research direction is relatively new. 
	
	The purpose of this paper is to  propose a Manin-type conjecture for 
	the distribution of klt Campana points on Fano orbifolds. 
	We show that the conjecture holds for all smooth vector group compactifications 
	with a strict normal crossings boundary divisor for the 
	weighted log-anticanonical height and for many more choices of heights.
	We investigate also the case of non-klt Campana points, and we observe that 
	all the difficulties that one encounters when dealing with integral points 
	 appear also in this setting.
	 	 
\subsection{Campana points}	
	
	There are several ways to ``interpolate'' between the classical notions of rational and 
	integral points. Keeping Manin's conjecture in mind, this article  
	argues in favor of a compelling option that arises from Campana's theory of pairs, 
	which he baptised \defi{orbifoldes g\'eom\'etriques}\footnote{Unlike the name suggests, 
	such objects are not stacks, but simply pairs consisting of a variety equipped with a 
	$\mathbb{Q}$-divisor of a specific type.}. 	
	There are various competing notions of Campana points in the 
	literature \cites{Abramovich, AVA16}, 
	and they all agree with the original definition of Campana 
	\cites{MR2163487, Campana15} on curves.
	On higher dimensional varieties, the various notions can lead to significantly 
	different sets of points,
	manifestly affecting the counting problems addressed in this paper, 
	as we explain in \S \ref{subsec:Example}. 
	We choose to work with Campana's original definition \cites{Campana15}
	because it best allows us to formulate a Manin-type conjecture which shares many characteristics with the now classical conjectures for rational points \cite{BM90} and \cite{Pey95}. Our study of local height integrals and Euler products 
	for vector group compactifications 	
	shows that the notion considered
	 in this paper interacts well with the tools from harmonic 
	 analysis: the regularization of the Euler product of local height integrals 
	 looks similar to the one used for the study of Manin's conjecture for rational 
	 points~(see Proposition~\ref{prop:EulerOneFactor} and Corollary~\ref{cor:eulerproduct}). 
	 
	 The notion of Campana points appearing in \cite{AVA16} is different from the one 
	 considered here. That notion enjoys good functoriality properties, but it seems ill-suited 
	 to the study of points of bounded height: for example, if one were to use the height zeta 
	 function method to count points of bounded height on vector group compactifications, 
	 then the regularization of the 
	 Euler product of local height integrals for the main term 
	 would require a newfound set of ideas. 
	 We consider this clarification an important contribution of this paper.

\subsection{A log Manin conjecture}
\label{subsec:conjecture}

	Let $(X,D_\epsilon)$ be a Campana orbifold (see \S \ref{subsec:orbifolds}) 
	over a number field $F$. Assume 
	moreover that $X$ is projective and that $-(K_X + D_\epsilon)$ is ample; a pair $(X,D_\epsilon)$ with this 
	additional property is called a \defi{Fano orbifold}. 
	Recall that the effective cone $\mathrm{Eff}^1(X)$ is finitely 
	generated by~\cite{BCHM}. 
	Fix a finite set $S$ of places of $F$ containing all archimedean places, as well 
	as a good integral model $(\mathcal X, \mathcal D_{\epsilon})$ of $(X, D_\epsilon)$ over 
	the ring of $S$-integers $\mathcal O_{F, S}$ of $F$ (see \S\ref{subsec:orbifolds}). 
	Write $(\mathcal{X},\mathcal{D}_\epsilon)(\mathcal{O}_{F,S})$ for the set of 
	$\mathcal{O}_{F,S}$-Campana points of $(\mathcal{X},\mathcal{D}_\epsilon)$ 
	(see Definition \ref{def:Campanapoints}), and  assume that $\lfloor D_\epsilon\rfloor = 0$,
	i.e., every weight $\epsilon_\alpha$ is strictly smaller than $1$.
	This condition is equivalent to saying that $(X, D_\epsilon)$ is klt in the sense of birational 
	geometry (see \cite[Definition 2.34]{KM98} for a definition of klt singularities, 
	and \cite[Lemma 2.30]{KM98} for a characterization).
	Let 
	\[
	\mathsf H_{\mathcal L}: X(F) \rightarrow \mathbb R_{>0}.
	\]
	be the height function determined by an adelically metrized 
	big line bundle $\mathcal L = (L,\|\cdot\|)$ on $X$ as in \cite[\S1.3]{Pey95}.	
	For any subset $U\subset X(F)$ 
	and positive real  number $T$, we consider the counting function
	\[
	\mathsf N(U, \mathcal L, T) = \# \{P \in U \mid \mathsf H_{\mathcal L} (P) \leq T\}.
	\]

	\begin{conj}[Manin-type conjecture for Fano orbifolds]
	\label{conj: ManinforCampana}
	Suppose that in addition to being big, the divisor $L$ is nef, and that the set of klt 
	Campana points $(\mathcal{X},\mathcal{D}_\epsilon)(\mathcal{O}_{F,S})$ is not thin. 
	Then there exists a thin set 
	$Z \subset (\mathcal{X},\mathcal{D}_\epsilon)(\mathcal{O}_{F,S})$ as in \S \ref{sec:thin} 
	such that
	
	\begin{equation}
	\label{eq:asymptotic_conjecture}
	\mathsf N((\mathcal{X},\mathcal{D}_\epsilon)(\mathcal{O}_{F,S}) \setminus Z, 
	\mathcal L, T) 
	\sim c(F, S, (\mathcal X, \mathcal D_\epsilon), \mathcal L, Z)
	T^{a((X, D_\epsilon), L)}(\log T)^{b(F, (X,D_\epsilon), L) -1}
	\end{equation}
	as $T\to\infty$,
	where 	
	$$ a((X, D_\epsilon), L) 
	= \inf \{t \in \mathbb R \mid tL + K_X + D_\epsilon \in \mathrm{Eff}^1(X)\}$$
	is the Fujita invariant of $(X, D_\epsilon)$ with respect to $L$, 
	$b(F, (X,D_\epsilon), L)$ is the codimension of the minimal supported face of 
	$\mathrm{Eff}^1(X)$ that contains the class 
	$a((X, D_\epsilon), L)[L] + [K_X + D_\epsilon]$ (cf.~\cite[Definition 2.1]{HTT15}),
	and the leading constant
	$c(F, S, (\mathcal X, \mathcal D_\epsilon), \mathcal L, Z)$ is a 
	positive
	Tamagawa constant, described in \S \ref{subsec:leading_constant}.
	\end{conj}
	
	The definition of the exponents $a((X, D_\epsilon), L)$ and $b(F, (X,D_\epsilon), L)$ 
	in the conjecture above is analogous to the case of rational points \cite{BM90}.
	This is the main reason for our choice among various possible definitions of Campana points.
	
	Although $a((X, D_\epsilon), L)$ and $b(F, (X,D_\epsilon), L)$ 
	do not depend on the choice of an integral model 
  	for $(X,D_\epsilon)$, the leading constant does depend on such a choice.		
	The description of the leading constant is analogous to 
  	Peyre's constant in \cite{Pey95} and \cite{TamagawaBT}. 

	The removal of a thin subset of rational points in order 
	to get a count that is not dominated by accumulating subvarieties  is a natural assumption, 
	which is already present in the case of Manin's conjecture for rational points
	(see for example~\cites{Pey17,LST18}). 
	In \S \ref{subsec:BY19} we explain 
	why a recent example of Browning 
	and Yamagishi~\cite{BY19} 
	whose exceptional set cannot be a proper closed subset
	is still compatible with
	Conjecture \ref{conj: ManinforCampana}.
		
	While the geometric properties of klt singularities are not used in this paper, we believe 
	that they will play a prominent role in the analysis of the exceptional sets for 
	Conjecture \ref{conj: ManinforCampana}.
	Indeed, in the classical case of rational points one of the key ingredients in the proof of thinness 
	of the conjectural exceptional set in \cite{LST18} is the BAB conjecture, which holds for klt 
	log Fano varieties (more precisely in the $\epsilon$-klt setting), proved in \cite{Bir19} 
	and \cite{Bir16}, but fails in the dlt case. This is one of the main reasons for expecting 
	that klt Campana points are easier to deal with compared to integral points.
		
	In attempting to formulate a conjecture for sets of Campana points that are not klt
	we encounter the same difficulties that have prevented the formulation of a conjecture in 
	the much more extensively studied case of integral points.
	For example, 
	the exponents appearing in the asymptotics of the counting functions in these results 
	depend heavily on the \emph{divisor} chosen for the counting function, and not only 
	on its numerical class 
	(see, e.g., \cite{CLT12} for integral points and \S \ref{sec:proofdlt} for Campana points).  
	It seems sensible to study explicit examples of sets of Campana points that are ``barely'' 
	non klt, e.g., when exactly one of the weights $\epsilon_\alpha$ is equal to $1$, as a 
	step towards a better understanding of the distribution of integral points 
	on Fano varieties. 
		
\subsection{Evidence} 
\label{subsec:results}

	We prove Conjecture \ref{conj: ManinforCampana} for equivariant 
	compactifications of vector groups. 
	This important class of varieties satisfies Manin's conjecture for rational 
	points \cite{CLT02} and  analogous asymptotics for integral 
	points \cite{CLT12}. It has also been studied for the motivic version of Manin's conjecture 
	in \cite{CLL16}, \cite{Bilu}. Hence, it provides an ideal testing ground for 
	Conjecture \ref{conj: ManinforCampana}.
	
	Let $F$ be a number field and let $G = \mathbb G_a^n$ be the $n$-dimensional 
	vector group. Let $X$ be a smooth, projective, equivariant compactification of $G$ 
	defined over $F$, such that the boundary divisor $D = X \setminus G$ is a strict 
	normal crossings divisor on $X$, with irreducible components $(D_\alpha)_{\alpha \in A}$. 
	Let $S $ be a finite set of places of $F$, containing all archimedean places, such that 
	there is a good integral model $(\mathcal X, \mathcal D)$
	for $(X,D)$ over the 
	ring of $S$-integers $\mathcal O_{F, S}$ of $F$ in the sense of \S\ref{subsec:campana}. 
	We choose a \defi{weight vector} $\epsilon = (\epsilon_\alpha)_{\alpha \in A},$ where 
	$$\epsilon_\alpha \in 
	\left\{\left. 1 - \frac1m\,\right|\, m \in \mathbb{Z}_{\geq 1}\right\} \cup \{1\}$$ 
	for all $\alpha$, and we set 
	$$D_\epsilon = \sum_{\alpha \in \mathcal A} \epsilon_\alpha D_\alpha,\ \ \ 
	\mathcal{D}_\epsilon = \sum_{\alpha \in \mathcal A} \epsilon_\alpha \mathcal{D}_\alpha,$$
	where $\mathcal D_\alpha$ denotes the closure of $D_\alpha$ in $\mathcal X$.
	Let $L$ be a big line bundle on $X$, and let $\mathcal{L}$ denote $L$ equipped 
	with a smooth adelic metrization. 

	\bigskip

	Our first main result addresses the situation where all $\epsilon_\alpha$ are strictly 
	smaller than $1$; we refer to this case as the \emph{klt case}. In this situation, we 
	get a precise result for ``many'' $L$. 
	We recall that a divisor is said to be \defi{rigid} if it has Iitaka dimension zero;
	see \cite[Section 2.1]{Laz04} for a definition of Iitaka dimension.
		
	\begin{thm} 
	\label{thm:maintheorem1} 
	With the notation above, assume that $(X,D_\epsilon)$ is klt. 
	Let $a = a((X, D_\epsilon), L)$ 
	be defined as in Conjecture \ref{conj: ManinforCampana}.
	If $aL + K_X + D_\epsilon$ is rigid, then the asymptotic formula in
	Conjecture \ref{conj: ManinforCampana} holds for 
	$(\mathcal X, \mathcal D_\epsilon, \mathcal L)$ 
	with exceptional set 
	$$Z = (X\setminus G) \cap (\mathcal{X},\mathcal{D}_\epsilon)(\mathcal{O}_{F,S}).$$ 
	\end{thm}

	\begin{remarks}
	\begin{enumerate}[leftmargin=*]
	\item[] 
	\item The asymptotic \eqref{eq:asymptotic_conjecture} holds for a pair $(X,D_\epsilon)$ 
	in Theorem~\ref{thm:maintheorem1} even if the pair is not a Fano orbifold. 
	See Theorem \ref{theo:main_kltrigid}.
	\medskip 
  	\item If $L=-(K_X+D_\epsilon)$, the rigidity condition in the statement is trivially satisfied, 
  	since in that case $a=1$. In this case, $b$ is the Picard rank of $X$. 
  	\medskip
  	\item We prove the conclusion of Theorem~\ref{thm:maintheorem1} also when the 
  	adjoint divisor is not rigid, under additional technical assumptions.
  	See Theorem~\ref{theo:main_kltnonrigid}. 
	\end{enumerate}
	\end{remarks}

	The more general case where some of the weights $\epsilon_\alpha$ are allowed 
	to be equal to $1$ -- to which we refer as the \emph{dlt case} -- is more subtle. In 
	this case, we have to restrict our attention to the case where $L$ is the ``orbifold 
	anticanonical line bundle'', due to subtleties arising in the formulation of the main 
	term.

	\begin{thm} 
	\label{thm:maintheorem2} 
	With notation as above, let $L$ 
	be the line bundle $-(K_X + D_\epsilon)$, and let $\mathcal{L}$ denote $L$ equipped 
	with a smooth adelic metrization as above. There exists a geometric invariant 
	$b = b(F, S, (X, D_\epsilon), L) > 0$, defined in \S \ref{sec:proofdlt}, such that 
	\[
	 \mathsf N((\mathcal{X},\mathcal{D}_\epsilon)(\mathcal{O}_{F,S}) \cap G(F), \mathcal{L}, T)  
	\sim \frac{c}{(b-1)!}T(\log T)^{b-1} 
	\text{ as $T \rightarrow \infty$},
	\] 
	for some positive constant $c$ that depends on $F,S,(\mathcal{X},\mathcal{D}_\epsilon)$ 	
	and  $\mathcal{L}$.
	\end{thm}

	It is important to observe that the logarithmic exponent $b$ in Theorem 
	\ref{thm:maintheorem2} for dlt  points depends on the choice of $S$; this was not 
	the case in Theorem \ref{thm:maintheorem1} for klt Campana points. In essence, when 
	$\epsilon_\alpha = 1$ for at least one index $\alpha$, the local zeta functions 
	associated to places in $S$ can contribute positively to $b$. 
	This is a typical feature observed in the literature about integral points of bounded height.
	Moreover, if $\epsilon_\alpha\in\{0,1\}$ for all $\alpha$, 
	our result recovers \cite{CLT12}.

	We note that the pair $(X,D_\epsilon)$ in the statement of 
	Theorem \ref{thm:maintheorem2} is not 
	required
	to be a Fano orbifold. In particular, Theorem \ref{thm:maintheorem2} 
	holds for all smooth compactifications 
	of vector groups with strict normal crossing boundary, and
	there are numerous such compactifications:
	indeed,
	blowing-up invariant points always produces new examples. 
	See \S \ref{subsec:geometry} for more details. 

\subsection{Methods} 
\label{subsec:methods}

	To prove Theorems~\ref{thm:maintheorem1} and~\ref{thm:maintheorem2}, we use 
	the height zeta function method, as in the foundational papers \cite{CLT02} and 	
	\cite{CLT12}. 
	Let 
	$$G(F)_\epsilon = G(F) \cap (\mathcal{X},\mathcal{D}_\epsilon)(\mathcal{O}_{F,S})$$ 
	be the set of rational points in $G$ which extend to Campana $\mathcal{O}_{F,S}$-points 
	on $(\mathcal{X},\mathcal{D}_\varepsilon)$ in the sense of \S\ref{subsec:campana}. 
	Even though the notation may suggest otherwise, the set $G(F)_\epsilon$ \emph{does} 
	depend on the choice of $S$ and the $\mathcal{O}_{F,S}$-model 
	$(\mathcal{X},\mathcal{D})$, which we have fixed once and for all.
	Then the \defi{height zeta function} is given by
	\[
	\mathsf Z_\epsilon(\bold s) 
	= \sum_{\bold x \in G(F)_\epsilon} \mathsf H(\bold x, \bold s)^{-1}
	=  \sum_{\bold x \in G(F)} \mathsf H(\bold x, \bold s)^{-1}\delta_\epsilon(\bold x),
	\] 
	where $\delta_\epsilon(\bold{x})$ is the indicator function detecting whether a given 
	point in $G(F)$ belongs to $G(F)_\epsilon$. Our goal is to obtain a meromorphic continuation 
	of this analytic function, and to apply a Tauberian theorem. To this end, we consider the 
	Fourier transform over the ad\`eles: 
	\[
	\widehat{\mathsf H}_\epsilon(\bold a, \bold s) 
	= \int_{G(\mathbb A_F)} \mathsf H(\bold x, \bold s)^{-1}
	\delta_\epsilon (\bold x)\psi_{\bold a}(\bold x) \, \mathrm d \bold x,
	\]
	and we use the Poisson summation formula
	\[
	\sum_{\bold x \in G(F)} \mathsf H(\bold x, \bold s)^{-1}\delta_\epsilon(\bold x)
	= \sum_{\bold a \in G(F)} \widehat{\mathsf H}_\epsilon(\bold a, \bold s)
	\]
	to obtain a meromorphic continuation of $\mathsf Z_\epsilon (\bold s)$. To prove 
	the absolute convergence of the right hand side, we estimate 
	$\widehat{\mathsf H}_\epsilon(\bold a, \bold s)$ by combining work from 
	\cites{CLT02, CLT10, CLT12} on height integrals with oscillating phase. 
	
\subsection{Structure of the paper} 
\label{subsec:structure_paper}

	After setting up the notation in \S\ref{sec:notation}, we start \S \ref{subsec:orbifolds} 
	by recalling the notion of Campana orbifold.
	We discuss different notions of Campana points that appear 
	in the literature in \S \ref{subsec:campana} -- this is crucial,
	since only one of these works well for our purposes.
	We include an example in \S\ref{subsec:Example} that shows how different notions 
	lead to different asymptotics for point counts on a single orbifold.
	In \S \ref{subsec:campana} we discuss a Peyre-type description of the 
	leading constant in Conjecture~\ref{conj: ManinforCampana}, then
	we introduce a notion of thin 
	set in the context of Campana points in \S \ref{sec:thin};
	in \S\ref{subsec:BY19} we discuss the compatibility of Browning and Yamagishi's 
	example~\cite{BY19} with Conjecture~\ref{conj: ManinforCampana}. Finally, 
	in \S\ref{subsec:functoriality} we discuss the functoriality properties of Campana points 
	under birational transformations. 
		
	In \S\ref{subsec:Clemens} we review a type of simplicial complex, called the 
	\emph{Clemens complex}, that helps to keep track, in the presence of integrality 
	conditions, of the contribution of local height integrals to the pole of the height 
	zeta function. We then use these complexes to give birational invariance results 
	(Lemmas~\ref{lem:bbirational} and~\ref{lemm: birationalinvariance_f}) for the $a$ 
	and $b$-invariants that appear in the asymptotic formula of the counting function 
	for Campana points.

	In \S\ref{subsec:geometry}, we specialize to Campana orbifolds that are equivariant 
	compactifications of vector groups. We recall basic facts about their geometry such 
	as their Picard groups and effective cones of divisors, as well as results from harmonic 
	analysis. After a  discussion on local and global heights in \S\ref{sec:heightzeta}, 
	we define the height zeta function of an equivariant compactification of a vector group, 
	and explain how to reduce the Poisson summation formula to the convergence of a sum 
	of Fourier transforms of local height functions (local height integrals). 
	Sections~\ref{sec: height integral} and~\ref{sec: height integral II} contain the 
	necessary estimates of local height integrals; before carrying on these technical estimates, we have included an interlude with a detailed explanation of the calculations in dimension $1$, for the benefit of readers new to this type of analysis.

	Theorems~\ref{thm:maintheorem1} and~\ref{thm:maintheorem2} are established, 
	respectively, in \S\ref{sec:proofklt} and \S\ref{sec:proofdlt}.

\subsection{Acknowledgements} 
	
	The authors would like to thank Tim Browning, Fr\'ed\'eric Campana, Ulrich Derenthal, Yoshishige Haraoka, 
	and Brian {Leh\-mann} for useful discussions and for their feedback. 
	We thank Dan Loughran for his valuable comments and for pointing out 
	a mistake in an early version of this paper. 
	We also thank the referee for very careful and thoughtful comments 
	which significantly improved the exposition of the paper and generalized our main theorems.
	
	We thank for their hospitality the organizers of the trimester program 
	``Reinventing Rational Points'' at the Institut Henri Poincar\'e, 
	Daniel Huybrechts at the Universit\"at Bonn, and 
	Michael Stoll, organizer of the workshop ``Rational Points 2019'' at Schney, where parts 
	of this paper were completed.
	
	Arne Smeets was supported by a Veni grant from NWO. Sho Tanimoto was partially 
	supported by Lars Hesselholt's Niels Bohr professorship, by MEXT Japan, Leading Initiative 
	for Excellent Young Researchers (LEADER), by Inamori Foundation, and by JSPS KAKENHI 
	Early-Career Scientists Grant numbers 19K14512. Anthony V\'arilly-Alvarado was partially 
	supported by NSF grants DMS-1352291 and DMS-1902274. 

\section{Notation} 
\label{sec:notation}

\subsection{Number fields, completions, and zeta functions}

	Let $F$ be an arbitrary number field. Denote by $\mathcal O_F$ its ring of integers, 
	by $\Omega_F$ its set of places, by $\Omega_F^{<\infty}$ the set of all finite 
	(non-archimedean) places, and by $\Omega_F^\infty$ the set of all infinite (archimedean) 
	places. For any finite set $S \subset \Omega_F$ containing $\Omega_F^\infty$, we 
	denote by $\mathcal O_{F, S}$ the ring of $S$-integers of $F$. For each $v\in \Omega_F$, 
	we denote by $F_v$ the completion of $F$ with respect to $v$. If $v$ is non-archimedean, 
	we denote by $\mathcal O_v$ the corresponding ring of integers, with maximal ideal 
	$\mathfrak m_v$ and residue field $k_v$ of size $q_v$. We write $\mathbb A_F$ for 
	the ring of ad\`eles of $F$.

	For each $v\in \Omega_F$, the additive group $F_v$ is locally compact, 
	and carries a self-dual Haar measure $\mathrm dx_v = \mu_v$ that we normalize  as follows:
	\begin{itemize}[leftmargin=*]
	\item $\mathrm dx_v$ is the ordinary Lebesgue measure on the real line if $v$ is real,
	\item $\mathrm dx_v$ is twice the ordinary Lebesgue measure on the plane if $v$ is complex,
	\item $\mathrm dx_v$ is the measure for which $\mathcal O_v$ has volume 	
	$N(\mathfrak{D})^{-1/2}$ if $v$ is a nonarchimedean place, where $\mathfrak{D}$ 
	denotes the absolute different of $F_v$, with norm $N(\mathfrak{D})$.
	\end{itemize}
	These Haar measures satisfy $\mu_v(\mathcal O_v) = 1$ for all but finitely many 
	non-archimedean places $v$; they induce a self-dual measure $\mathrm dx = \mu$ 
	on $\mathbb A_F$.
	We denote by $\mathrm d \bold x_v$ the induced Haar measure on $F_v^n$.
	We also denote the product measure on $\mathbb A_F^n$ by $\mathrm d \bold x$.

	We define the absolute value $|\cdot|_v$ by requiring that
	\[
	\mu_v (xB) = |x|_v\cdot \mu_v (B)
	\]
	for any Borel set $B \subset F_v$. When $v$ is real, $|\cdot|_v$ is the usual absolute value. 
	When $v$ is complex, $|\cdot|_v$ is the square of the usual norm on the complex numbers. 
	For any prime number $p$, we have $|p|_p = 1/p$. For any finite extension 
	$F_v/\mathbb Q_p$, we have
	\[
	|x|_v = |N_{F_v/\mathbb Q_p}(x)|_p.
	\]

	We define the \defi{local zeta function} by
	\[
	\zeta_{F_v}(s) = 
	\begin{cases}
	s^{-1} & \text{ if $F_v = \mathbb R$ or $\mathbb C$, }\\ 
	\left(1-q_v^{-s}\right)^{-1} & \text{ if $v$ is non-archimedean.}
	\end{cases}
	\]
	For non-archimedean places, the local zeta functions fit together to give the 
	\defi{Dedekind zeta function}
	\[
	\zeta_F(s) = \prod_{v \in \Omega_F^{< \infty}} \zeta_{F_v}(s).
	\] 

\subsection{Varieties and divisors}

	Let $F$ be a field with fixed algebraic closure $\bar F$. An $F$-variety $X$ is a 
	geometrically integral separated $F$-scheme of finite type. We denote by $\bar{X}$ 
	the base change of $X$ to $\bar{F}$. If $F$ is a number field and $v\in \Omega_F$, 
	we write $X_v$ for the base change of $X$ to $F_v$. Given a Weil $\mathbb{R}$-divisor 
	$D = \sum_i a_i D_i$ on $X$, we denote by $\lfloor D \rfloor = \sum_i \lfloor a_i \rfloor D_i$ 
	its ``integral part''. We denote the reduced divisor $\sum_{a_i \neq 0} D_i$ 
	by $D_{\mathrm{red}}$. 
	Given a scheme $\mathcal X$ defined over a ring $A$, we denote by $\mathcal X\otimes_AB$ 
	the base change of $\mathcal X$ under a ring extension $A\to B$.

\subsection{Conventions for complex numbers}

	We denote the real part of a complex number $s$ by $\Re(s)$, and the absolute 
	value by $|s|$. Given $\bold s=(s_1,\dots,s_n)\in\mathbb C^n$ and $c\in \mathbb R$, 
	by the expression $\Re(\bold s)>c$ we mean that $\Re(s_i)>c$ for all $i\in\{1,\dots,n\}$. 
	We also write $|\bold s|:=\max_{i=1}^n|s_i|$.

\section{Campana orbifolds, Campana points and the conjecture}
\label{sec:campana_orbifolds_points_conjecture}

	In this section we recall two notions of Campana points, 
	we discuss the leading constant and the exceptional sets 
	in Conjecture \ref{conj: ManinforCampana}, 
	and we investigate the functoriality properties of the sets of Campana points.

\subsection{Orbifolds} 
\label{subsec:orbifolds}

	We recall Campana's notion of \defi{orbifolds} (``orbifoldes g\'eom\'etriques''), as 
	introduced in his foundational papers \cite{Campana04, MR2831280}. In this article, 
	we only consider those orbifolds which Campana calls ``smooth''; in this section, we 
	allow $F$ to be any field.

	\begin{defn} 
	\label{defn:CampanaOrbifold}
	A \defi{Campana orbifold} over $F$ is a pair $(X, D)$ consisting of a smooth variety 
	$X$ and an effective Weil $\mathbb Q$-divisor $D$ on $X$, both defined over $F$, 
	such that 
\smallskip
	\begin{enumerate}[leftmargin=*]
	\item we have 
	$$D = \sum_{\alpha \in \mathcal A} \epsilon_\alpha D_\alpha,$$ where the $D_\alpha$ 
	are prime divisors on $X$, and $\epsilon_\alpha$ belongs to the set of weights
	$$ \coeffset :=
	\left\{\left. 1 - \frac1m\,\right|\, m \in \mathbb{Z}_{\geq 1}\right\} \cup \{1\} $$ 
	for all $\alpha \in \mathcal A$;
\medskip
	\item the support $D_{\mathrm{red}} = \sum_{\alpha \in \mathcal A} D_\alpha$ 
	is a divisor with strict normal crossings on $X$. 
	\end{enumerate}
	\end{defn}

	Condition (2) in this definition implies that the irreducible components $D_\alpha$ 
	of $D_{\mathrm{red}}$ are smooth; it is important to note, however, that they may 
	well be geometrically reducible. 
	We refer to \cite[\S41.21]{Stacks} for the definition of strict normal crossing.
	The definition also implies that any  Campana orbifold 
	$(X,D)$ is a \defi{dlt} (divisorial log terminal) pair, in the sense of birational geometry (see \cite[Definition 2.37]{KM98} for this notion). 
	We say that $(X,D)$ is \defi{klt} (Kawamata log terminal) if moreover 
	$\epsilon_\alpha \neq 1$ for all $\alpha \in \mathcal  A$, i.e., if all weights are 
	strictly smaller than $1$.

	Conversely, given a smooth $F$-variety $X$, a reduced divisor 
	$D = \sum_{\alpha \in \mathcal A} D_\alpha$ on $X$ with strict normal crossings 
	and a weight vector $\epsilon = (\epsilon_\alpha)_{\alpha \in \mathcal A}$, 
	where $\epsilon_\alpha \in \coeffset$ for all $\alpha$, we obtain a Campana orbifold 
	$(X,D_\epsilon)$ over $F$ by setting  
	$D_\epsilon = \sum_{\alpha \in \mathcal A} \epsilon_\alpha D_\alpha$. 
	
In this paper we consider only Campana orbifolds $(X,D)$ with $X$ proper. 

\subsection{Two types of Campana points} 
\label{subsec:campana}

	The notion of ``orbifold rational point'' is explored in Campana's papers 
	\cite[\S 9]{Campana04}, \cite[\S 4]{MR2163487}, \cite[\S 12]{MR2831280}, \cite[\S7.6]{Campana15} 
	and in Abramovich's survey 
	\cite[Lecture 2]{Abramovich}. The adjective ``rational'' may create confusion, so we use 
	the name Campana points here, to acknowledge that they are an intermediate 
	notion between rational and integral points. In fact, \cite{Abramovich} defines 
	two different notions of Campana points, one more restrictive than the other. It is 
	essential for us to separate the two notions, since the orbifold analogue of Manin's 
	conjecture seems to work well only for the more restrictive version; this is the one to 
	which we will refer to simply as \defi{Campana points} (Definition~\ref{def:Campanapoints}). 
	The notion featuring in the recent 
	paper \cite{AVA16} is (a slight variant of) the less restrictive version, and we will refer 
	to it as \defi{weak Campana points} (Definition~\ref{def:weakCampanapoints}); 
	it seems to be ill-behaved for the problem studied 
	in this paper (see \S \ref{subsec:Example}). 

	\begin{remark} 
	So far few results on the arithmetic of (weak) Campana points are available. 
	Work on points of bounded height goes back to \cite{VanV12}, followed immediately 
	by \cite{BVV12} and more recently by \cite{BY19}. Work  of  
	Schindler and the first author \cite{PS20} investigates the distribution of Campana points on 
	toric varieties. 
Recent work of Xiao \cite{Xiao} extends our results to biequivariant compactifications of the Heisenberg group.
	
	In dimension $1$, where both notions of Campana points coincide, the analogue of 
	Mordell's conjecture for Campana points has been proved over function fields, first in 
	characteristic $0$ by Campana himself \cite{MR2163487}, and only recently in 
	arbitrary characteristic \cite{KPS19}. Over number fields, the only known result says 
	that the $abc$ conjecture implies Mordell's conjecture for Campana points; see 
	\cite[Appendix]{Smeets} for a detailed argument.
	\end{remark}

	Let $(X,D_\epsilon)$ be a Campana orbifold with $X$ proper over $F$, 
	where 
	$D_\epsilon = \sum_{\alpha \in \mathcal A} \epsilon_\alpha D_\alpha$ and the 
	$\epsilon_\alpha$ belong to the usual set $\coeffset$. Let $S \subseteq \Omega_F$ 
	be a finite set containing $\Omega_F^\infty$. 
	We say that $(X,D_\epsilon)$ has a \defi{good integral model away from $S$} 
	if there exists a flat, proper model $\mathcal X$ over $\mathcal{O}_{F,S}$
	such that $\mathcal X$ is regular. Given such a model, we denote by $\mathcal{D}_\alpha$ 
	the Zariski closure of $D_\alpha$ in $\mathcal{X}$, and we write $(\mathcal{X},\mathcal{D}_\epsilon)$ 
	for the model, where 
	$\mathcal{D}_\epsilon := \sum_{\alpha \in \mathcal A} \epsilon_\alpha \mathcal{D}_\alpha$.

	Campana points can only be defined once a suitable model has been fixed, so let us 
	fix a good integral model $(\mathcal{X},\mathcal{D}_\epsilon)$ for $(X,D_\epsilon)$ 
	over $\mathcal{O}_{F,S}$. Any rational point $P \in X(F)$ extends uniquely to an integral 
	point $\mathcal{P} \in \mathcal{X}(\mathcal{O}_{F,S})$ by the valuative criterion 
	for properness.

	Let $\mathcal A_\epsilon=\{\alpha\in \mathcal A: \epsilon_\alpha\neq0\}$.
	Let 
	$X^\circ=X\setminus (\bigcup_{\alpha\in\mathcal A_\epsilon}D_\alpha)$. 
	If $P \in X^\circ(F)$ and if $v \not\in S$ is a place of $F$, then we get an induced point 
	$\mathcal{P}_v \in \mathcal{X}(\mathcal{O}_{v})$. 
	For each $\alpha \in \mathcal A$ 
	such that $\mathcal{P}_v \not\subseteq \mathcal D_\alpha$, 
	the pullback of $\mathcal D_\alpha$ via $\mathcal{P}_v$ 
	defines a non-zero ideal in $\mathcal O_v$.
	We denote its colength by $n_v(\mathcal D_\alpha, P)$; this is 
	the \defi{intersection multiplicity} of $P$ and $\mathcal D_\alpha$ at $v$. 
	When $P \in D_\alpha$ for some $\alpha \in \mathcal A_\epsilon$, we define 
	$n_v(\mathcal D_\alpha, P)$ to be $+\infty$. 
	
	The \defi{total intersection number} of $P$ with $\mathcal{D}$ is then 
	$$n_v(\mathcal{D}_\epsilon, P) 
	= \sum_{\alpha \in \mathcal A_\epsilon} \epsilon_\alpha n_v(\mathcal{D}_\alpha, P).$$

	The following definition goes back to \cite[\S 2.1.7]{Abramovich} and features in  
	\cite{AVA16} as well. 

	\begin{defn} 
	\label{def:weakCampanapoints}
	With the notation introduced above, we say that $P \in X(F)$ is a 
	\defi{weak Campana $\mathcal{O}_{F,S}$-point} on $(\mathcal{X},\mathcal{D}_\epsilon)$ 
	if the following holds: 
	\begin{enumerate}[leftmargin=*]
	\item for all $\alpha$ with $\epsilon_\alpha = 1$ 
	and $v \notin S$, $n_v(\mathcal D_\alpha, P) = 0$, i.e., 
	$P \in \left(X\setminus \bigcup_{\epsilon_\alpha = 1}D_\alpha\right)(\calO_{F,S})$
	and 
	\medskip
	\item for $v \not\in S$,  if $n_v(\mathcal{D}_\epsilon, P) > 0$ then 
	$$n_v(\mathcal{D}_\epsilon, P) 
	\leq \left(\sum_{\alpha \in \mathcal A_\epsilon} n_v(\mathcal{D}_\alpha, P)\right) - 1.$$ 
	In particular, if $n_v(\mathcal{D}_\alpha, P)=+\infty$ for some $\alpha\in\mathcal A_\epsilon$, 
	the inequality is trivially satisfied.
	\end{enumerate}
	\end{defn}

	We denote the set of weak Campana $\mathcal{O}_{F,S}$-points on 
	$(\mathcal{X},\mathcal{D}_\epsilon)$ by 
	$(\mathcal{X},\mathcal{D}_\epsilon)_{\textsf{w}}(\mathcal{O}_{F,S})$.

	We obtain a more restrictive notion by imposing conditions for individual irreducible 
	components of the support of $D$, in the spirit of \cite[Definition 2.4.17]{Abramovich}:

	\begin{defn} 
	\label{def:Campanapoints} 
	With the notation introduced above, we say that $P \in X(F)$ is a 
	\defi{Campana $\mathcal{O}_{F,S}$-point} on $(\mathcal{X},\mathcal{D}_\epsilon)$ 
	if the following hold:
	\smallskip
	\begin{enumerate}[leftmargin=*]
	\item  for all $\alpha$ with $\epsilon_\alpha = 1$ 
	and $v \notin S$, $n_v(\mathcal D_\alpha, P) = 0$, i.e., 
	$P \in \left(X\setminus \bigcup_{\epsilon_\alpha = 1}D_\alpha\right)(\calO_{F,S})$
	and 
	\medskip
	\item for $v\notin S$, and all $\alpha \in \mathcal A_\epsilon$ with both $\epsilon_\alpha < 1$ and  
	$n_v(\mathcal{D}_\alpha, P) > 0$, we have
	$$n_v(\mathcal{D}_\alpha, P) \geq \frac{1}{1 - \epsilon_\alpha}.$$  
	In other words, writing $\epsilon_\alpha = 1 - \frac{1}{m_\alpha}$, we require 
	$n_v(\mathcal{D}_\alpha, P) \geq m_\alpha$ whenever $n_v(\mathcal{D}_\alpha, P) > 0$. 
	\end{enumerate}
	\end{defn}
	\begin{rmk}
	\label{rmk:clarification_definition}
	Definition \ref{def:Campanapoints} implies that a point $P \in X(F)$ that lies in $D_\alpha(F)$ for 
	some $\alpha\in\mathcal A_\epsilon$ is a Campana $\mathcal{O}_{F,S}$-point if it lies 
	in the $v$-adic closure of $X^\circ(F_v)\cap ((\mathcal{X},\mathcal{D}_\epsilon)(\mathcal O_{F,S}))$ 
	for all places $v\notin S$.
	\end{rmk}

	We denote the set of  Campana $\mathcal{O}_{F,S}$-points on 
	$(\mathcal{X},\mathcal{D}_\epsilon)$ by 
	$(\mathcal{X},\mathcal{D}_\epsilon)(\mathcal{O}_{F,S})$. We have 
	$$X(F) 
	\supseteq (\mathcal{X},\mathcal{D}_\epsilon)_{\textsf{w}}(\mathcal{O}_{F,S}) 
	\supseteq (\mathcal{X},\mathcal{D}_\epsilon)(\mathcal{O}_{F,S}) 
	\supseteq \mathcal{X}^\circ(\mathcal{O}_{F,S}),$$ 
	where 
	$\mathcal{X}^\circ
	= \mathcal{X} \setminus \left(\sum_{\alpha \in \mathcal A_\epsilon} \mathcal{D}_\alpha\right)$. 
	The leftmost two inclusions are equalities if $\epsilon_\alpha = 0$ for all 
	$\alpha \in \mathcal A$, and the rightmost inclusion is an equality if 
	$\epsilon_\alpha = 1$ 
	for all $\alpha \in \mathcal A_\epsilon$.

	For $v \notin S$, we denote by $(\mathcal{X},\mathcal{D}_\epsilon)(\mathcal O_v)$ the set 
	of points $P_v\in X(F_v)$ such that $n_v(\mathcal D_\epsilon, P_v)$ satisfies 
	the condition in Definition \ref{def:Campanapoints}.
	We also define the set of adelic Campana points by
	\[
	(\mathcal{X},\mathcal{D}_\epsilon)(\mathbb A_F) 
	= \prod_{v \notin S} (\mathcal{X},\mathcal{D}_\epsilon)(\mathcal O_v) \times \prod_{v \in S} X(F_v).
	\]	
	By Remark~\ref{rmk:clarification_definition} the space 
	$(\mathcal{X},\mathcal{D}_\epsilon)(\mathcal O_v)$ is 
	a closed subspace of the topological space $X(F_v)$; in particular, it is compact.
	
\subsubsection{An instructive example}
\label{subsec:Example}

	The following example illustrates the difference between the two notions of Campana 
	points introduced above. We show that these notions yield different asymptotics for 
	counts of points of bounded height. Moreover, the difference is encoded not only in the 
	leading constant, but also in the exponent of the logarithm. In \S\ref{subsec:functoriality} 
	we use this example to discuss functoriality of Campana points under birational 
	transformations. 
	
	Let $X=\mathbb{P}^2_\mathbb{Q}$ with coordinates $(x_0:x_1:x_2)$, and let 
	$D_i=\{x_i=0\}$ for $i\in\{0,1,2\}$. Taking $\mathcal X=\mathbb{P}^2_{\mathbb Z}$ 
	and $\epsilon_0,\epsilon_1,\epsilon_2\in\coeffset$, the Campana orbifold 
	$(X, \sum_{i=0}^2 \epsilon_i D_i)$ has the obvious good 
	integral model 
	$(\mathcal X, \sum_{i=0}^2 \epsilon_i \mathcal D_i)$ over $\mathbb Z$ in the sense 
	of \S \ref{subsec:campana}. For $0\leq i\leq 2$, we write $\epsilon_i=1-\frac1{m_i}$ 
	with the convention that $\frac 1{m_i}=0$ if $\epsilon_i=1$. A point in 
	$\mathcal X(\mathbb Z)$, 
	represented by coprime integer coordinates $(x_0:x_1:x_2)$, is 
	\begin{itemize}
	\item a \defi{weak Campana $\mathbb Z$-point} if 
	$x_i\in\{\pm1\}$ for all $i\in\{0,1,2\}$ such that $\epsilon_i=1$, and
	\[
	p\mid \prod_{\substack{0\leq i\leq 2\\ \epsilon_i\neq 0}} x_i \quad \Rightarrow \quad \sum_{\substack{0\leq i\leq 2\\ \epsilon_i\neq 0}}\frac 1{m_i}v_p(x_i)\geq1
	\]
	for every prime $p$, or equivalently,  if $x_0^{m_1m_2}x_1^{m_0m_2}x_2^{m_0m_1}$ 
	is $m_0m_1m_2$-full (in the case $0<\epsilon_0,\epsilon_1,\epsilon_2<1$);
	\medskip
	\item a \defi{Campana $\mathbb Z$-point} if 
	$x_i\in\{\pm1\}$ for all $i\in\{0,1,2\}$ such that $\epsilon_i=1$, and
	\[
	p \mid x_i \quad \Rightarrow \quad \frac 1{m_i} v_p(x_i)\geq 1
	\]
	for every prime $p$ and every $i\in\{0,1,2\}$ such that $\epsilon_i\neq 1$, or equivalently, if $x_i$ is $m_i$-full for 
	all $i\in\{0,1,2\}$, assuming $\epsilon_0,\epsilon_1,\epsilon_2<1$.
	\end{itemize}
	Note how a point on the boundary divisor can be a Campana point: 
	for example, if $\epsilon_0,\epsilon_1,\epsilon_2<1$ and $P = (0:x_1:x_2)$ with $x_1$, $x_2$ coprime integers, 
	then $P$ is a weak $\Z$-Campana point, although it is a $\Z$-Campana point only if for $i = 1, 2$, 
	we have $p\mid x_i \implies v_p(x_i) \geq m_i$.
	
	Let us specialize to the case where $m_0 = m_1=m_2=2$. 
	We set $X^\circ = X \setminus (\bigcup_{i=0}^2D_i)$. 

	To count (weak) Campana 
	points of bounded height we use the exponential Weil height 
	\begin{align*}
	H\colon \mathbb P^2(\mathbb{Q}) &\to \mathbb R \\
	(x_0:x_1:x_2) &\mapsto \max \{ |x_0|,|x_1|,|x_2|\} 
	\textrm{ whenever $x_0,x_1,x_2$ are coprime integers}.
	\end{align*}
	
	\begin{proposition}
	Let $\mathcal X, \mathcal D_0,\mathcal D_1,\mathcal D_2$ be as above 	
	and let $\mathcal D_\epsilon=\sum_{i=0}^2\frac 12\mathcal D_i$.
	Then for sufficiently large $T>0$,
	\begin{align}
	\# \{x\in (\mathcal X,\mathcal D_\epsilon)(\mathbb Z)\cap X^\circ(\mathbb Q): H(x)\leq T\} 
	& \ll T^{3/2}, 
	\label{eq:example3.2.1campana}
	\\
	\# \{x\in (\mathcal X,\mathcal D_\epsilon)_{\mathrm w}(\mathbb Z)\cap X^\circ(\mathbb Q) : H(x)\leq T\} 
	& \gg T^{3/2}\log T.
	\label{eq:example3.2.1weakcampana}
	\end{align}
	\end{proposition}	
	\begin{proof}
	In this setting, the set of Campana $\mathbb Z$-points on $X^\circ$ is in bijection with the set of triples 
	$(x_0,x_1,x_2)\in \mathbb Z_{\neq0}^3$ such that $\gcd(x_0,x_1,x_2)=1$ and $x_0$, $x_1$ 
	and $x_2$ are all squareful. The counting function of Campana $\mathbb Z$-points of 
	Weil height bounded by $T$ has an upper bound given by the cardinality of the set 
	obtained by removing the coprimality condition, which grows asymptotically like 
	$T^{\frac 32}$, up to multiplication by a positive constant, by \cite{erdos} (see also \cite{BG58}).
	
	The set of weak Campana $\mathbb Z$-points on $X^\circ$ is in bijection with 
	the set of triples $(x_0,x_1,x_2)\in \mathbb Z_{\neq0}^3$ such that $\gcd(x_0,x_1,x_2)=1$ 
	and $x_0x_1x_2$ is  squareful.  
	To prove the lower bound in \eqref{eq:example3.2.1weakcampana}; we count points of 
	bounded height in the subset $A$ of coprime triples $(x_0,x_1,x_2)\in\mathbb Z^3_{>0}$ 
	such that  $x_0$ is a square and $x_1x_2$ is a square.  The size of this subset is 
	estimated by
	\[
	\sum_{d\leq T}\mu(d) 
	\cdot \#\left\{1\leq x_0\leq T :\, x_0\text{ square}, d\mid x_0 \right\}
	\cdot \#\left\{1\leq x_1,x_2\leq T:\,  x_1x_2 \text{ square}, 
	d\mid x_1, d\mid x_2\right\},
	\]
	where $\mu$ denotes the M\"obius function. The number of squares up to $T$ that 
	are divisible by a given squarefree integer $d$ is $T^{1/2}/d + O(1)$. To estimate 
	the cardinality of the set $B$ of pairs $(x_1,x_2)\in (d\mathbb Z_{>0})^2$ such that 
	$x_1,x_2\leq T$ and $x_1x_2$ is a square, we write $u=\gcd(x_1/d,x_2/d)$ and 
	$y_i=x_i/(du)$ for $i\in\{1,2\}$. Then $x_1x_2$ is a square if and only if both $y_1$ 
	and $y_2$ are squares. Writing $y_i= z_i^2$ for $i\in\{1,2\}$, we get
	\[
	\#B=\sum_{u\leq T/d}
	\sum_{\substack{z_1,z_2\leq (T/(du))^{1/2}\\ \gcd(z_1,z_2)=1}} 1= 
	\frac{T/d\log (T/d)}{\zeta_{\mathbb Q} (2)} +O(T/d).
	\]
	Therefore, $\#A=(\zeta_{\mathbb Q} (2))^{-2}\, T^{3/2}\log T +
	O_\delta (T^{3/2}(\log T)^\delta)$ for all $\delta>0$. 	
	\end{proof}	
	
	The upper bound \eqref{eq:example3.2.1campana} is 
	in agreement with  
	Conjecture~\ref{conj: ManinforCampana}. 
	Indeed, for the line bundle $L = \calO(1)$, we have $a((X, D_\epsilon), L) = 3/2$ 
	and $b = b(F, (X,D_\epsilon), L) = 1$, so Conjecture~\ref{conj: ManinforCampana} 
	predicts a counting formula for Campana points of bounded height that grows like 
	$cT^{3/2}$ as $T\to \infty$, which is correct. 
	The upper bound is in fact sharp; see \cite[Theorem 1.2]{PS20}.
	The lower bound \eqref{eq:example3.2.1weakcampana} shows that counting Campana points 
	and weak Campana points of bounded height in the same setting can lead to different asymptotics. 
	However, since the lower bound is based on counting points in a thin set 
	(denoted by $A$ in the proof), it does not show that Conjecture 1.1 fails 
	when counting weak Campana points. We are unaware of any successful attempt to produce 
	an asymptotic formula for the count of weak Campana points of bounded height in an example 
	where the sets of Campana points and weak Campana points do not coincide. 
	
\subsection{The leading constant}
\label{subsec:leading_constant}
	We keep the notation introduced in \S \ref{subsec:conjecture}. 
	In this section, we define the leading constant that appears in 
	Conjecture~\ref{conj: ManinforCampana}, in the case when the divisor 
	$a((X, D_\epsilon), L)L + K_X + D_\epsilon$ is $\mathbb Q$-linearly equivalent 
	to a rigid effective divisor $E$.	
	The construction 
	here is analogous to \cite{Pey95} and \cite{TamagawaBT}. For simplicity, we assume 
	that the boundary divisor $D$ contains all components of $E$; we denote by 
	$\mathcal A(L)$ the set of irreducible components of $D$ that are not contained 
	in the support of $E$.

	Write $U = X \setminus \mathrm{Supp}(E)$, and let $\Lambda$ be the image of 
	$\mathrm{Eff}^1(X)$ under the projection map 
	$\rho\colon \mathrm{Pic}(X) \rightarrow \Pic(U)$; this is a finitely generated, 
	polyhedral cone since $X$ is a Fano orbifold.
	Let
	\[
	\chi_{\Lambda}(\rho([L])) 
	= \int_{\Lambda^*} e^{-\langle \rho([L]),\bold x\rangle}\,\mathrm d \bold x,
	\]
	where $\Lambda^* \subset \Pic(U)_\R^*$ is the dual cone to $\Lambda$ and 
	$\mathrm d \bold x$ is the Lebesgue measure on $\Pic(U)_\R^*$, normalized by 
	the dual lattice $\Pic(U)^* \subset \Pic(U)_\R^*$ 
	(see~\cite[Definition 2.3.14]{TamagawaBT}). The \defi{$\alpha$-constant} of the pair 
	$(X,D_\epsilon)$ with respect to $L$ is
	\[
	\alpha((X, D_\epsilon), L) 
	:= \chi_{\Lambda} (\rho([L]))\prod_{\alpha \in \mathcal A(L)}(1 - \epsilon_\alpha),
	\]
	and the \defi{$\beta$-constant} of the pair $(X,D_\epsilon)$ with respect to $L$ is 
	\[
	\beta((X, D_\epsilon), L) = \#\HH^1(\Gamma, \mathrm{Pic}(\overline{U})).
	\]
	The group $\HH^1(\Gamma, \mathrm{Pic}(\overline{U}))$ is finite. 
	Indeed, since $X$ is a Fano orbifold, it follows from \cite{HM07} that $X$ is rationally connected. 
	Hence $\mathrm{Pic}(\overline{X})$ is a free $\mathbb Z$-module of finite rank. 
	Furthermore since $E$ is rigid, its geometric components generate a primitive lattice 
	in $\mathrm{Pic}(\overline{X})$. Thus its cokernel $\mathrm{Pic}(\overline{U})$ is torsion free. 
	Hence we conclude that $\HH^1(\Gamma, \mathrm{Pic}(\overline{U}))$ is finite.

	The open set $U$ can be endowed with a Tamagawa measure 
	$\tau_{U}$~\cite[Definition 2.8]{CLT10}; fixing an adelic metrization on each component 
	of $D$ and on $K_X$, we let $\tau_{U, D_\epsilon} = \mathsf H_{D_\epsilon}\tau_{U}$, 
	where $\mathsf H_{D_\epsilon}$ is the height function associated to the divisor 
	$D_\epsilon$. We define the \defi{Tamagawa constant} by
	\[
	\tau(F, S, (\mathcal X, \mathcal D_\epsilon), \mathcal L) 
	:= \int_{\overline{U(F)}_\epsilon} 
	\mathsf H(x, a((X, D_\epsilon), L)L + K_X + D_\epsilon)^{-1} \, 
	\mathrm d \tau_{U, D_\epsilon},
	\]
	where $\overline{U(F)}_\epsilon$ denotes either 
	\begin{enumerate}
	\item the topological closure of ${(\mathcal{X},\mathcal{D}_\epsilon)(\mathcal{O}_{F,S})}\cap U(F)$ 
	in $U(\mathbb A_F)$, or 
	\medskip
	\item the \defi{Brauer set} $U(\mathbb A_F)_\epsilon^{\mathrm{\mathrm{Br}(U)}}$ defined as follows: 
	for any subset $B \subset U(F_v)$, let $B_\epsilon$ denote the support of $\delta_{\epsilon, v}$ on $B$. 
	The \defi{adelic Campana set} is the restricted product
	\[
	U(\mathbb A_F)_\epsilon = \prod_{v}{'} U(F_v)_\epsilon
	\]
	with respect to $U(\mathcal O_v)_\epsilon$.
	The set $U(\mathbb A_F)_\epsilon^{\mathrm{\mathrm{Br}(U)}}$ is the zero locus of the Brauer-Manin pairing. 
	See \cite[Chapter 8]{Poonen} for the definition of the Brauer-Manin pairing.
	 \end{enumerate}
	 In Theorem \ref{thm:maintheorem1} we use the latter definition of $\overline{U(F)}_\epsilon$; 
	 see Lemma \ref{lem:rigid_klt_constant}. It is not known whether the two sets coincide; 
	 see Question \ref{question:weak_approximation} below.
	We recall that already in the classical case of rational points, it is not clear what domain should appear in the integral 
	that defines the Tamagawa constant; see \cite[Remarks 6.13 and 7.8]{Sal98}.
	This integral converges in the general setting of a Fano orbifold, by an analog of 
	Denef's formula (\ref{eq:localFTtrivialchar}) in this setting. 
	Finally the leading constant 
	for Conjecture~\ref{conj: ManinforCampana} is
	\[
	c(F, S, (\mathcal X, \mathcal D_\epsilon), \mathcal L) 
	= \frac{\alpha((X, D_\epsilon), L)\beta((X, D_\epsilon), L)
	\tau(F, S, (\mathcal X, \mathcal D_\epsilon), \mathcal L)}
	{a((X, D_\epsilon), L)(b(F, (X, D_\epsilon), L)-1)!}.
	\]
	Our Theorem~\ref{thm:maintheorem1} 
	agrees with Conjecture~\ref{conj: ManinforCampana}, including the prediction for 
	the constant, as we show in \S\ref{subsec:rigid}.

\subsection{Thin exceptional sets}\label{sec:thin}

	In the formulation of Conjecture \ref{conj: ManinforCampana} 
	we expect that it is necessary to remove a thin 
	set of Campana points from the count in order to obtain a formula that reflects the 
	global geometry of the  Campana orbifold; indeed, already for rational points it has 
	been understood for quite some time that a version of Manin's conjecture with only a 
	closed -- rather than thin -- exceptional set admits counterexamples, see 
	\cites{BT96b,LeRu13, BL17}. 
	Meanwhile, several authors have recently built up evidence towards a version of Manin's 
	conjecture with a thin exceptional set, see~\cites{LT17,Pey17,Sen17,LST18}. While we 
	do believe that the set of klt Campana points is itself \emph{not} thin, we are unable at 
	present to show this; however, we propose a problem we hope will ameliorate this 
	circumstance.  

	Let $(X,D_\epsilon)$ be a Fano orbifold over a number field $F$, 
	i.e., a Campana orbifold such that $-(K_X + D_\epsilon)$ is ample.	
	Fix a finite set $S \subset \Omega_F$ containing all archimedean places of $F$, as well 
	as a good integral model $(\mathcal X, \mathcal D)$ of $(X, D)$ over 
	$\Spec \, \mathcal O_{F, S}$, as in \S\ref{subsec:campana}. 	
	Write $(\mathcal{X},\mathcal{D}_\epsilon)(\mathcal{O}_{F,S})$ for the set of 
	$\mathcal{O}_{F,S}$-Campana points of $(\mathcal{X},\mathcal{D}_\epsilon)$. 
	
	\begin{defn}
	A \defi{thin subset} of $(\mathcal{X},\mathcal{D}_\epsilon)(\mathcal{O}_{F,S})$ is 
	a subset of a finite union of
	\begin{enumerate} 
  	\item type I sets: those of the form 
  	$Z\cap (\mathcal{X},\mathcal{D}_\epsilon)(\mathcal{O}_{F,S})$ 
  	for a proper Zariski closed subset $Z \subset X$;
\medskip
  	\item type II sets: those of the form 
  	$f(Y(F))\cap (\mathcal{X},\mathcal{D}_\epsilon)(\mathcal{O}_{F,S}),$ where 
  	$f\colon Y \to X$ is a generically finite cover of degree at least $2$, with 
  	$Y$ a projective, integral $F$-variety.
	\end{enumerate} 
	\end{defn}

	It is natural to ask whether $(\mathcal{X},\mathcal{D}_\epsilon)(\mathcal{O}_{F,S})$ 
	is itself \emph{not} thin, possibly after a finite extension of the ground field. After all, 
	if a version of Manin's conjecture with a thin exceptional set is to hold for Campana points 
	on Fano orbifolds, we would like to have something left to count \emph{after} the removal 
	of a thin subset.  We are thus forced to make what we hope is a superfluous hypothesis 
	in Conjecture~\ref{conj: ManinforCampana}, namely, that 
	$(\mathcal{X},\mathcal{D}_\epsilon)(\mathcal{O}_{F,S})$ itself is 
	\emph{not} thin in our setting.
	
	This shortcoming is already present in the traditional case of rational points on smooth 
	Fano varieties, where we expect the set of rational points to be not thin if it is non-empty. 
	This is known conditionally on Colliot-Th\'el\`ene's conjecture predicting that the 
	Brauer-Manin obstruction controls all failures of weak approximation on rationally connected 
	varieties~\cite{CT03}. Indeed, this conjecture implies that smooth Fano varieties satisfy 
	``weak weak approximation'', which in turn implies that the set of rational points is not 
	thin~\cite[Theorem~3.5.7]{Ser92}.

	On a positive note, Serre has shown that $\PP^n(F)$ is not thin \cite[\S3.4]{Ser92}. 
	This prompts us to ask:
	\begin{question}
	Let $F$ be a number field and let $D = \sum_{\alpha \in \mathcal{A}} D_\alpha$ be 
	a divisor on $\mathbb P_F^n$ with strict normal crossings. For each 
	$\alpha \in \mathcal A$, pick $\epsilon_\alpha \in \coeffset$ with $\epsilon_\alpha < 1$ 
	and set $D_\epsilon = \sum_{\alpha \in \mathcal A} \epsilon_\alpha D_\alpha$, so that 
	the Campana orbifold $(\mathbb P^n, D_\epsilon)$ is klt. Assume moreover that 
	$-(K_{\mathbb P^n} + D_\epsilon)$ is ample. Fix a good integral model 
	$(\calP^n,\calD_\epsilon)$ of $(\mathbb P^n, D_\epsilon)$, and a finite set $S$ of 
	places of $F$ that includes all the archimedean places. Is the set 
	$(\calP^n,\calD_\epsilon)(\mathcal{O}_{F,S})$ of klt Campana points non-thin?
	\end{question}

	For some partial results, we refer to the recent paper of 
	Browning--Yamagishi~\cite[\S 4]{BY19}. 
	A version of this question for integral points on a log K3 surface is addressed in \cite{Coc19}.
\medskip

	In a different direction, if the set of Campana points 
	$(\mathcal{X},\mathcal{D}_\epsilon)(\mathcal{O}_{F,S})$ were thin, then there 
	would exist a set of places $T$ such that the image of this set in $\prod_{v \in T} X(F_v)$ 
	is not dense, by~\cite[Theorem~3.5.3]{Ser92}. Since we expect 
	$(\mathcal{X},\mathcal{D}_\epsilon)(\mathcal{O}_{F,S})$ to be \emph{not} thin, we ask:

	\begin{question}
	\label{question:weak_approximation}
	Is there a finite set $S_0 \subset \Omega_F$ containing $S$ such that 
	for any $T \subseteq \Omega_F$ a finite set of places such that $S_0\cap T = \emptyset$, 
	$(\mathcal{X},\mathcal{D}_\epsilon)(\mathcal{O}_{F,S})$ is dense in 
	$\prod_{v \in T}((\mathcal X, \mathcal D_\epsilon)(\mathcal O_v))$?
	In other words, does the set of Campana points satisfy 
	\defi{weak weak approximation}?
	\end{question}

\subsection{Browning-Yamagishi's example}
\label{subsec:BY19}

	In \cite[Theorem 1.2]{BY19}, Browning and Yamagishi presented an illuminating example, 
	which illustrates in particular that in the formulation of 
	Conjecture~\ref{conj: ManinforCampana}, it is important to exclude a thin set 
	to obtain the expected growth rate. We briefly recall the construction. 
	We define divisors on $\mathbb P^2_\Q = \Proj \Q[x_0,x_1,x_2]$ by
	\[
	D_i = \{x_i = 0\} \textrm{ for }i = 0, 1, 2,\quad\textrm{and} \quad 
	D_3 = \{x_0 + x_1 + x_2=0 \}.
	\]
	We denote by $H$ the hyperplane class, and we set $D=\bigcup_{i=0}^3D_i$. 
	Consider the Campana orbifold 
	$(\mathbb P^2_\Q, D_\epsilon = \sum_{i = 0}^3 \frac{1}{2} D_i)$; and extend it to 
	the obvious good integral model $(\PP^2_\Z,\calD_\epsilon)$ over $\Spec(\Z)$. 

	A computation shows that
	\[
	a((\mathbb P^2, D_\epsilon), H) = 1, \quad 
	b(\mathbb Q, (\mathbb P^2, D_\epsilon), H) = 1.
	\]
	On the other hand, Browning and Yamagishi show that
	\[
	\mathsf N((\PP^2_\Z,\mathcal{D}_\epsilon)(\Z)\cap(\PP^2\setminus D)(\mathbb Q), H, T) \gg T\log T,
	\]
	a computation at odds with a closed-set version of Conjecture~\ref{conj: ManinforCampana}.
	As we explain below, the unexpected rapid growth of the 
	counting function is explained by a type II thin set.

	Let $Q \subset \PP^3 = \Proj \Q[w_0,w_1,w_2,w_3]$ be the smooth quadric defined by
	\[
	w_0^2 - w_1^2 + w_2^2 = w_3^2 
	\]
	and consider the finite morphism of degree $8$ given by
	\begin{align*}
	f \colon Q &\rightarrow \mathbb P^2_\Q \\
	(w_0:w_1:w_2:w_3) &\mapsto (w_0^2: -w_1^2: w_2^2)
	\end{align*}
	Note that
	\[
	f(Q(\mathbb Q)) \subset (\mathbb P^2_{\mathbb Z}, \mathcal D_\epsilon)(\mathbb Z),
	\]
	and that, by the ramification formula
	we have
	\[
	K_Q = f^*(K_{\mathbb P^2} + D_\epsilon).
	\]
	From this, it follows that 
	$$a(Q, f^*H) = 1, \quad b(\mathbb Q, Q, f^*H) = 2.$$ 
	Therefore the number of rational points on $Q$ grows more quickly than the expected 
	growth rate on $(\mathbb P^2_\Z, D_\epsilon)$. 

	There are in fact infinitely many twists $Q^\sigma/\mathbb P^2_\Q$ such that 
	$$a(Q^\sigma, H) = 1, \quad b(\mathbb Q, Q^\sigma, H) = 2,$$ 
	so it is a priori unclear whether the combined images of their rational points on 
	$\PP^2_\Q$ form a thin set. This type of problem is already addressed in~\cite{LST18}, 
	using Hilbert's irreducibility theorem. We obtain the following auxiliary result:
	\begin{lemma}
	The set
	\[
	Z = \bigcup_\sigma f^\sigma(Q^\sigma(\mathbb Q)),
	\]
	where the union is taken over all $\sigma \in \HH^1(\mathrm{Gal}(\Qbar/\Q), 
	\mathrm{Aut}(\overline{Q}/\mathbb P^2_{\overline{\mathbb Q}}))$ with the property that 
	$$b(\mathbb Q, Q^\sigma, f^{\sigma *}H) = 2,$$ 
	is thin.
	\end{lemma}
	
	The following proof is due to the referee. 

	\begin{proof}
	The twists $Q^\sigma$ are given by 
	$Q_{a_0,a_1,a_2}=\{a_0w_0^2-a_1w_1^2 +a_2w_2^2 = w_3^2\}\subseteq\mathbb P^3$ 
	for $a_0,a_1,a_2\in\mathbb Q^\times$, and  
	$Q_{a_0,a_1,a_2}$ has Picard rank $2$ if and only if $a_0a_1a_2$ is a square. 
	The corresponding twists of $f$ are 
	$$
	f_{a_0,a_1,a_2} : Q_{a_0,a_1,a_2} \to \mathbb P^2_\mathbb Q, 
	\quad (w_0:\dots:w_3) \mapsto (a_0 w_0^2 : -a_1 w_1^2 : a_2 w_2^2).
	$$	
	We observe that for all $a_0,a_1,a_2\in\mathbb Q^\times$ such that $a_0a_1a_2$ is a square,  
	the images of the $\mathbb Q$-points on $Q_{a_0,a_1,a_2}$ under $f_{a_0,a_1,a_2}$ are contained 
	in the set of points $(x_0:x_1:x_2)$ in $\mathbb P^2(\mathbb Q)$ such that $-x_0x_1x_2$ is a square, 
	which is a thin set.
	\end{proof}

\subsection{Birational invariance and functoriality}
\label{subsec:functoriality}

	We conclude this section 
	by exploring the functoriality properties 
	of sets of Campana points under birational morphisms. 

\subsubsection{An instructive example (continued)}

	To motivate our discussion, we appeal to the example of \S\ref{subsec:Example}: 
	recall that $X=\mathbb{P}^2_\mathbb{Q}$ with coordinates 
	$(x_0:x_1:x_2)$, $D_i=\{x_i=0\}$ for $i\in\{0,1,2\}$, and consider the Campana 
	orbifold $(X, \sum_{i=0}^2(1-\frac 1{m_i})D_i)$ with $\mathbb Z$-model 
	$\mathcal X=\mathbb P^2_{\mathbb Z}$. 

	Let $\varphi: Y\to X$ be the blow-up with center the intersection point of $D_1$ and  $D_2$. 
	Then $\varphi$ is an isomorphism over $X^\circ = X \setminus (\bigcup_{i=0}^2 D_i)$. 
	Let $Y^\circ=\varphi^{-1}(X^\circ)$.
	Denote by $E$ the exceptional divisor and by $\widetilde D_i$ the strict transform of 
	$D_i$ for $i\in\{0,1,2\}$. Then $Y^\circ=Y\setminus (E\cup(\bigcup_{i=0}^2\widetilde D_i))$. 
	The blow-up $\mathcal {Y}$ of $\mathcal X$ at the 
	subvariety defined by $\{x_1=x_2=0\}$ yields a smooth projective $\mathbb Z$-model of $Y$. 
	We observe that  given a point $P\in Y^\circ(\mathbb{Q})$, the point $\varphi(P)$ is
	\begin{itemize}
	\item
	a weak Campana $\mathbb Z$-point on 
	$(\mathcal X, \sum_{i=0}^2(1-\frac 1{m_i})\mathcal D_i)$ if for every prime $p$, the sum
	\[
	\frac 1{m_0}n_p(\widetilde D_0, P) + \frac 1{m_1}n_p(\widetilde D_1,P)+
	\frac 1{m_2}n_p(\widetilde D_2,P)+\left (\frac 1{m_1}+\frac 1{m_2} \right )n_p(E,P)
	\]
	is either $0$ or at least $1$; 
\medskip
	\item
	a Campana $\mathbb Z$-point on 
	$(\mathcal X, \sum_{i=0}^2(1-\frac 1{m_i})\mathcal D_i)$ if for every prime $p$, 
	the numbers
	\[
	\frac 1{m_0}n_p(\widetilde D_0, P), \quad 
	\frac 1{m_1}(n_p(\widetilde D_1,P)+n_p(E,P)), \quad 
	\frac 1{m_2}(n_p(\widetilde D_2,P)+n_p(E,P))
	\]
	are either $0$ or at least $1$.
	\end{itemize}
	This description clearly shows that the set of (weak) Campana points is not invariant 
	under birational morphisms, i.e., for general $m_0,m_1,m_2$, there is no 
	choice of positive integers $\widetilde m_0,\widetilde m_1,\widetilde m_2,\widetilde m_E$
	such that the restriction of the blow-up $\varphi$ to $Y^\circ$ would
	induce a bijection between the set of (weak) Campana points for 
	$(\mathcal Y, (1-\frac 1{\widetilde m_E})\mathcal E + \sum_{i=0}^2(1-\frac 1{\widetilde m_i})\widetilde{\mathcal D}_i )$ on the open subset $Y^\circ$ and the set 
	of (weak) Campana points for $(\mathcal X, \sum_{i=0}^2(1-\frac 1{m_i})\mathcal D_i)$ on the isomorphic open subset $X^\circ$, where $\mathcal E, \widetilde{\mathcal D}_0, \widetilde{\mathcal D}_1, \widetilde{\mathcal D}_2$ denote the closures in $\mathcal Y$ of $E, \widetilde D_0, \widetilde D_1, \widetilde D_2$, respectively.

	Not all is lost, however: if we define $\widetilde m_i=m_i$ 	
	for $i\in\{0,1,2\}$ and $\widetilde m_E=\max\{m_1,m_2\}$, then the set of (weak) Campana points 
	on the resulting orbifold
	$(\mathcal Y, (1-\frac 1{\widetilde m_E})\mathcal E + \sum_{i=0}^2(1-\frac 1{\widetilde m_i})\widetilde{\mathcal D}_i )$ is mapped by $\varphi$ into a \emph{subset} of the 
	set of (weak) Campana points on 
	$(\mathcal X, \sum_{i=0}^2(1-\frac 1{m_i})\mathcal D_i)$. 
	
\subsubsection{The general picture}

	Let $X$ be a rationally connected smooth projective variety defined over a number field 
	$F$ and let $D = \sum_{\alpha\in\mathcal A} D_\alpha$ be a strict normal crossings divisor 
	on $X$. Fix a weight vector $\epsilon = (\epsilon_\alpha)_{\alpha\in\mathcal A}$ where 
	$\epsilon_\alpha \in \coeffset$ with $\epsilon_\alpha = 1-1/m_\alpha < 1$. Set 
	$D_\epsilon = \sum_{\alpha \in \mathcal A} \epsilon_\alpha D_\alpha$ and consider 
	the Campana orbifold $(X, D_\epsilon)$, which is a klt pair. 

	Let
	\[
	\varphi \colon \widetilde{X} \rightarrow X,
	\]
	be a birational morphism from a smooth projective variety $\widetilde{X}$, such that 
	$\widetilde{D} =( \varphi^*D)_{\mathrm{red}}$ is a strict normal crossing divisor. 
	We assume for simplicity that $\varphi$ is an isomorphism outside of $D$ and that both 
	$(\widetilde{X}, \widetilde D)$ and $(X, D)$ admit good integral models 
	$(\widetilde{\mathcal X}, \widetilde {\mathcal D})$ and $(\mathcal X, \mathcal D)$ that 
	are compatible. We assign a weight vector $\tilde{\epsilon}$ to $\widetilde{D}$ as follows.
	For the strict transform of a component $D_\alpha$ of $D$, we set 
	$\tilde{\epsilon}_\alpha = \epsilon_\alpha$. If $E_\beta$ is an exceptional divisor and 
	if $e_{\beta, \alpha}$ denotes the coefficient of $E_\beta$ in $\varphi^*D_\alpha$, 
	then we define 
	$$\tilde{m}_\beta 
	= \max \{  \lceil m_\alpha /e_{\beta, \alpha} \rceil\mid e_{\beta, \alpha} >0\}
	\quad \textrm{ and }\quad \tilde{\epsilon}_\beta = 1-1/\tilde{m}_\beta.$$
	Then $\varphi:(\widetilde X,\widetilde D_{\tilde \epsilon})\to (X,D)$ is 
	a ``morphisme orbifolde'' in the sense of \cite[D\'efinition 2.3]{MR2831280}.
	
	By construction, we have
	\[
	\varphi((\widetilde{\mathcal{X}},\widetilde{\mathcal{D}}_{\tilde{\epsilon}})
	(\mathcal{O}_{F,S}) )
	\subset (\mathcal{X},\mathcal{D}_\epsilon)(\mathcal{O}_{F,S}),
	\]
	but this inclusion need not be an equality. On the other hand, the $a$-  and $b$-invariants 
	are well-behaved for our choice of $\tilde\epsilon$, as we now explain. 
	We observe that 	
	$$
	K_{\widetilde{X}} + D_{\tilde{\epsilon}} \geq \varphi^*(K_X + D_\epsilon)
	$$ by \cite[Corollaire 2.12]{MR2831280}. 
	Then the arguments of \cite[\S 2]{HTT15} show that
	\[
	a((\widetilde{X}, \widetilde{D}_{\tilde{\epsilon}}), \varphi^*L) = a((X, D_{\epsilon}), L), 	
	\quad 
	b(F, (\widetilde{X}, \widetilde{D}_{\tilde{\epsilon}}), \varphi^*L) = b(F, (X, D_{\epsilon}), L).
	\]
	
	We end by remarking that 
	$\tau(F, S, (\widetilde{\mathcal X}, \widetilde{\mathcal D}_\epsilon), \mathcal L)$ and 
	$\tau(F, S, (\mathcal X, \mathcal D_\epsilon), \mathcal L)$ will be different in general 
	because $(\calX,\calD_\epsilon)(\calO_{F,S})$ and 
	$(\widetilde{\calX},\widetilde{\calD}_{\tilde\epsilon})(\calO_{F,S})$ are different. 
	Our overall conclusion is that 
	our Manin-type conjecture for klt Campana points
	is quite sensitive 
	to birational modifications. 
	In particular, proving the asymptotic formula for the counting 
	function after a birational modification need not easily yield an asymptotic formula for 
	the original variety.

\section{Analytic Clemens complexes}
\label{subsec:Clemens}

	Clemens complexes are simplicial sets that keep track of containment relations 
	between the intersections of components of a divisor in a variety. 
	As in \cite{CLT12}, Clemens complexes will be used in \S \ref{sec:proofdlt} to keep 
	track of the contribution  of the local height integrals to the pole of the height zeta 
	function when some integrality conditions appear, that is, when some component of 
	the boundary has weight $1$. For a more detailed treatment, we refer the reader 
	to~\cite[\S 3.1]{CLT10}.

	In this section $X$ is a smooth, proper variety over a number field $F$, and 
	$D = \sum_{\alpha\in\calA} D_\alpha$ is a reduced divisor on $X$ with strict 
	normal crossings. Let $v \in \Omega_F$, and fix an embedding $\bar F \subseteq \bar F_v$, 
	so that $\Gamma_v := \Gal(\bar F_v/F_v)$ acts on $\bar X$ and $\bar D$. Write 
	$\bar \calA$ for the indexing set of $\bar D$, and $\calA_v$ for the set of orbits of 
	$\bar \calA$ under the action of $\Gamma_v$. Recall that $X_v$ denotes the base 
	change of $X$ to $F_v$; write 
	$D_v := D\otimes_F F_v = \bigcup_{\beta \in \calA_v} D_{v,\beta}$, where the 
	$D_{v,\beta}$ are irreducible components. 

	Given a divisor $D'$ on $X$ such that 
	$\bar D'=\bigcup_{\alpha\in\mathscr A}\bar D_\alpha$ for some 
	$\mathscr{A}\subseteq\bar{\mathcal A}$, we denote by $\mathscr A_v$ the set of 
	orbits of $\mathscr A$ under the action of $\Gamma_v$. As a \emph{set}, the 
	\defi{$F_v$-analytic Clemens complex} associated to $D'$ consists of irreducible 
	components $Z$ of intersections $\bigcap_{\beta\in B}D_{v,\beta}$ for 
	$B\subseteq\mathscr A_v$ such that $Z(F_v)\neq\emptyset$. The complex enjoys 
	additional structure, e.g., as a poset; see \cite[\S 3.1]{CLT10} for details.
	The \defi{dimension} of the Clemens complex of $D'$ is 
	\[
	\max\left\{\# B: B\subseteq \mathscr {A}_v,  
	\bigcap_{\beta\in B} D_{v,\beta} (F_v)\neq\emptyset\right\} -1.
	\]

	We may now define the $a$- and $b$-invariants of the pair $(X,D)$ at $v$ with respect to 
	a linear combination of boundary components with positive coefficients. These invariants 
	will come up in the calculation of the position and order of the rightmost pole of a local 
	height integral of $X$ at $v$, in the case where $X$ is an equivariant compactification 
	of $G = \G_a^n$. 

	Keeping the notation introduced above, we assume further that 
	${-K}_{X_v} \sim \sum_{\beta \in \calA_v} \rho_\beta D_{v,\beta}$, with 
	$\rho_\beta \in \Z$ 
	for all $\beta$, and we set 
	$L = \sum_{\beta\in\calA_v} \lambda_\beta D_{v,\beta}$ 
	with
	$\lambda_\beta > 0$ for all $\beta$. 
	We define the \defi{$\tilde{a}$-invariant} of the pair $(X,D)$ at $v$ with respect to $L$ by
	\[
	\tilde{a}((X, D), L) = \max_{\beta\in \calA_v} \left\{\frac{\rho_\beta-1}{\lambda_\beta} \right\}.
	\]

	Let us denote the the sum of the boundary components that do \emph{not} appear 
	in the support of $\tilde{a}((X, D), L)L + K_X + D$ by $D'$; in other words, we set
	\[
	D' = D - (\tilde{a}((X, D), L)L + K_X + D)_{\mathrm{red}}.
	\]
	Writing $\mathcal C_{F_v}^{\mathrm{an}}(D, L)$ for the $F_v$-analytic Clemens 
	complex associated to $D'$, we define the \defi{$b$-invariant} of $(X,D)$ at $v$ 
	with respect to $L$ as follows:
	\[
	b(F_v, (X, D), L) = 1 + \dim \mathcal C_{F_v}^{\mathrm{an}}(D, L).
	\] 
	We will now prove that the $\tilde{a}$- and $b$-invariants are birational invariants in a suitable 
	sense. While this result is certainly of independent interest, we will use it to prove the 
	meromorphic continuation of certain local height integrals in \S\ref{sec: height integral}. 

	\begin{lemma}
	\label{lem:bbirational}
	Let $X$, $D$ and $L$ be as above. Let $(\widetilde{X},\widetilde{D})$ be another pair 
	satisfying the same hypotheses as $(X,D)$, namely: 
	(i) $\widetilde{D}$ is a reduced 
	divisor  with strict normal crossings on a smooth proper variety $\widetilde{X}$ over $F$, 
	(ii) ${-K}_{\widetilde{X}_v}$ is a 
	linear combination of irreducible components of $\widetilde{D}_v$. Assume that there is 
	a birational morphism $\varphi \colon \widetilde{X} \to X$ with 
	$\varphi^{-1}(D) = \widetilde{D}$ that is an isomorphism outside $D$. Then
	\[
	\tilde{a}((X, D), L) = \tilde{a}((\widetilde{X}, \widetilde{D}), \varphi^*L)
	\quad\textrm{and}\quad
	b(F_v, (X, D), L) = b(F_v, (\widetilde{X}, \widetilde{D}), \varphi^*L).
	\]
	\end{lemma}

	\begin{proof}
	First, we observe that the birational invariance of the $\tilde{a}$-invariant follows from the fact 
	that the pair $(X, D)$ is log canonical, i.e.,  we can write
	\[
	\tilde{a}((X, D), L)\varphi^*L + K_{\widetilde{X}} + \widetilde{D} 
	= \varphi^*(\tilde{a}((X, D), L)L + K_X + D) + E
	\] 
	where $E \geq 0$ is an effective divisor supported on the exceptional locus of $\varphi$. 

	From now on, we denote $\tilde{a}((X,D),L)$ simply by $a$ and we work over $F_v$, for a fixed 
	place $v$. To prove birational invariance of the $b$-invariant, we first use 
	\cite[Theorem 0.3.1]{AKMW} to reduce to the case where the morphism $\varphi$ is 
	a blow-up of a smooth center having normal crossings with $D$. Let $E$ be an exceptional 
	divisor of $\varphi$.  
	
	First suppose that the image of $E$ is not a component of the intersection of some 
	of the boundary components. Then  \cite[(3.11.1)]{Kollar} shows that the log discrepancy 
	of the exceptional divisor $E$ is greater than $-1$, hence that $E$ appears in the support 
	of $a \varphi^*L + K_{\widetilde{X}} + \widetilde{D}$. Let $Z$ be a maximal element in 
	$\mathcal C_{F_v}^{\mathrm{an}}(D, L)$ such that $b(F_v, (X, D), L) = \mathrm{codim}\, Z$ .
	Let $Z$ be a component of $\cap_{i = 1}^rD_{v,\beta_i}$ thus $\mathrm{codim}\, Z =r$. 
	If the image $T$ of $E$ does \emph{not} contain $Z$, 
	then $b(F_v, (\widetilde{X}, \widetilde{D}), \varphi^*L) = \mathrm{codim}\, Z$. 
	Thus our assertion follows in this case.
	If $T$ contains $Z$, then by rearranging indices, we may assume 
	that $T \subset D_{v,\beta_i}$ for $i \leq k$ and $T \not\subset D_{v,\beta_i}$ for 
	$i > k$. Denoting the codimension of $T$ by $t$, we have $k < t$; hence the strict 
	transforms of $D_{v,\beta_i}$ for $i \leq k$ meet in $\varphi^{-1}(Z)$. On the other 
	hand, the strict transforms of $D_{v,\beta_i}$ for $i > k$ all contain $\varphi^{-1}(Z)$. 
	Thus $b(F_v, (\widetilde{X}, \widetilde{D}), \varphi^*L) = r = b(F_v, (X, D), L)$. 
	Thus our assertion follows in this case too.

	Next suppose that 
	$T$ is a component of the intersection of some of the boundary components.
	Then $E$ does not appear in the support of the difference of 
	$a \varphi^*L + K_{\widetilde{X}} + \widetilde{D}$ and $\varphi^*(aL + K_X + D)$. 
	We further distinguish two cases. First, if  $E$ does not appear in the support of 
	$\varphi^*(aL + K_X + D)$, we denote by $Z$ a maximal element of 
	$\mathcal C_{F_v}^{\mathrm{an}}(D, L)$ so that $b(F_v, (X, D), L) = \mathrm{codim}\, Z$ 
	and we assume that $Z$ is a component of 
	$\cap_{i = 1}^rD_{v,\beta_i}$. Either $T$ and $Z$ do not meet, or $T$ contains $Z$; 
	in the former case, we have 
	$b(F_v, (\widetilde{X}, \widetilde{D}), \varphi^*L) = \mathrm{codim}\, Z$.
	In the latter case, we may assume that 
	$T$ is a component of $\cap_{i = 1}^kD_{v,\beta_i}$ with $k \leq r$.
	Then the strict transforms of the $D_{v,\beta_i}$'s do not meet in 
	$\varphi^{-1}(Z)$, but $E$ and 
	$r-1$ strict transforms of $D_{v,\beta_2}, \cdots, D_{v,\beta_{r}}$ intersect. 
	Thus we conclude that $b(F_v, (\widetilde{X}, \widetilde{D}), \varphi^*L)  = r$.
	Second, if $E$ does appear in the support of $\varphi^*(aL + K_X + D)$, then $T$ does 
	not contain $Z$, and therefore $T$ and $Z$ do not meet. 
	This implies that $b(F_v, (\widetilde{X}, \widetilde{D}), \varphi^*L)  = b(F_v, (X, D), L)$.
	\end{proof}

	We will now introduce a version of the $b$-invariant for rational functions. If $f$ is an 
	arbitrary rational function on $X$, then for every $\alpha\in\mathcal A$, we denote by 
	$d_\alpha(f)$ the coefficient of $D_\alpha$ in the principal divisor $\mathrm{div}(f)$.
   	Let $D''$ be the sum of boundary components $D_\alpha$ such that $D_\alpha$ does not
   	appear in the support of $aL + K_X + D$ and $d_\alpha(f) \leq 0$. We denote by 
   	$\mathcal C_{F_v}^{\mathrm{an}}(D, L, f)$ the $F_v$-analytic Clemens complex 
   	associated to $D''$, and we define the $b$-invariant by
	\[
	b(F_v, (X, D), L, f) = 1 + \dim \mathcal C_{F_v}^{\mathrm{an}}(D, L, f).
	\]
	Using the same methods, we obtain the following analogue of Lemma~\ref{lem:bbirational}:

	\begin{lemma}
	\label{lemm: birationalinvariance_f}
	Let $X$, $D$, $L$ and $f$ be as above. Let $(\widetilde{X},\widetilde{D})$ be another 
	pair satisfying the same hypotheses as $(X,D)$, namely: 
	(i) $\widetilde{D}$ is a reduced 
	divisor  with strict normal crossings on a smooth proper variety $\widetilde{X}$ over $F$, 
	(ii) ${-K}_{\widetilde{X}_v}$ is a 
	linear combination of irreducible components of $\widetilde{D}_v$. Assume that there is 
	a birational morphism $\varphi \colon \widetilde{X} \to X$ with 
	$\varphi^{-1}(D) = \widetilde{D}$ that is an isomorphism outside $D$. Then
	\begin{equation*}
	b(F_v, (X, D), L, f) = b(F_v, (\widetilde{X}, \widetilde{D}), \varphi^*L, f\circ\varphi).
	\tag*{\qed}
	\end{equation*}
	\end{lemma}

\section{Geometry of equivariant compactifications of vector groups} 
\label{subsec:geometry}

	The geometry of vector group compactifications is
	worked out in \cite{HT99}, where equivariant compactifications of a vector group on 
	$\mathbb P^n$ are classified. Surprisingly, there is more than one such compactification. 	
	There are classification results of equivariant 
	compactifications that are del Pezzo surfaces and Fano $3$-folds \cites{DL10, 
	DL15, HM18}, but equivariant compactifications of vector groups need 
	not be Mori dream spaces. Indeed, blow-ups of the standard equivariant compactification on 
	$\mathbb P^n$ along a smooth center on the boundary hyperplane inherit the 
	group compactification structure, so examples with a Cox ring that is not finitely 
	generated can be constructed by  
	blowing up suitable centers 
	(see \cite[Example 2.17]{HTT15}). 	
	This feature makes equivariant compactifications of vector groups 
	difficult to study via universal torsors, showing once more the power of 
	the height zeta function method.
	In addition, equivariant compactifications of vector groups admit deformations, 
	whereas equivariant compactifications involving reductive groups typically do not; 
	this feature also makes the former class of compactifications 
	interesting objects from a geometric point of view.

	We now recall some basic facts on the geometry of equivariant compactifications of 
	vector groups from \cite{HT99} and \cite{CLT02}. Let $X$ be a smooth equivariant 
	compactification of $G = \mathbb G_a^n$ defined over a field $F$ of characteristic 
	$0$. By definition, $X$ contains $G$ as a dense Zariski open, and its complement 
	$D = X \setminus G$ is divisorial, i.e., it is a union of prime divisors:
	\[
	D = \bigcup_{\alpha \in \mathcal A} D_\alpha.
	\]
	The irreducible divisors $D_\alpha$ need not be geometrically irreducible, so we also 
	consider the decomposition of $\bar{D}$ into irreducible components:
	\[
	\bar{D} = \bigcup_{\alpha \in \bar{\mathcal A}} \bar{D}_\alpha.
	\]
	There is a natural action of the Galois group $\Gamma= \mathrm{Gal} (\bar{F}/F)$ on 
	the index set $\bar{\mathcal A}$, and Galois orbits are in one-to-one correspondence 
	with elements of $\mathcal A$.

\subsection{Picard groups and the anticanonical class}

	\begin{prop}{\cite[Proposition 1.1]{CLT02}}
	\label{prop:effectivecone}
	With the above notation, the following hold.
	\begin{enumerate}
	\item There are natural isomorphisms of Galois modules
	\[
	\Pic(\bar{X}) = \bigoplus_{\alpha \in \bar{\mathcal A}} \mathbb Z \bar{D}_\alpha, 
	\quad \mathrm{Eff}^1(\bar{X}) 
	= \bigoplus_{\alpha \in \bar{\mathcal A}} \mathbb R_{\geq 0} \bar{D}_\alpha,
	\]
	where $\mathrm{Eff}^1(\bar{X})$ is the cone of effective divisors on $\bar{X}$.
\medskip

	\item By taking $\Gamma$-invariant parts, we have
	\[
	\Pic(X) = \bigoplus_{\alpha \in \mathcal A} \mathbb Z D_\alpha, 
	\quad \mathrm{Eff}^1(X) = \bigoplus_{\alpha \in \mathcal A} \mathbb R_{\geq 0} D_\alpha,
	\]
	where $\mathrm{Eff}^1(X)$ is the cone of $\Gamma$-invariant effective divisors on $X$. 	
	\end{enumerate}
	\end{prop}

	Let $f$ be a non-zero linear form on $G = \mathbb G_a^n$, defined over $F$.
	Considering $f$ as an element of the function field $F(X)$, we can write 
	$\mathrm{div}(f)$ uniquely as
	\[
	\mathrm{div}(f) = E(f) - \sum_{\alpha \in \mathcal A} d_\alpha(f) D_\alpha,
	\]
	where $E(f)$ is the hyperplane along which $f$ vanishes in $G$, and the $d_\alpha(f)$ 
	are integers.
	
	\begin{prop}{\cite[Lemma 1.4]{CLT02}, \cite[Before Lemma 3.4.1]{CLT12}}
	\label{prop:d_alpha(f)}
 	We have $d_\alpha(f) \geq 0$ for all $\alpha \in A$, and the set of integral vectors
	\[
	\{(d_\alpha(f))_{\alpha \in \mathcal{A}} \mid \text{$f$ is a non-zero linear form on $G$}\}
	\]
	is \emph{finite}.
	\end{prop}

	Finally, the anticanonical divisor turns out to be linearly equivalent to an integral linear 
	combination of boundary components: we have 
	$-K_X \sim \sum_{\alpha \in \mathcal A} \rho_\alpha D_\alpha$ for certain integers 
	$\rho_\alpha$, and by~\cite[Lemma 2.4]{CLT02}, we know that $\rho_\alpha \geq 2$ 
	for all $\alpha$. 

	\begin{rmk} 
	With the above notation, if $(\epsilon_\alpha)_{\alpha \in \mathcal{A}}$ is any vector 
	of weights chosen from the allowed set $\coeffset 
	= \left\{\left. 1 - \frac1m\,\right|\, m \in \mathbb{Z}_{\geq 1}\right\} \cup \{1\}$, 
	the \defi{orbifold anticanonical divisor} $-(K_X + D_\epsilon)$ of the Campana orbifold 
	$(X,D_\epsilon)$ is automatically big. This follows from the fact that the cone of big 
	divisors is the interior of the pseudo-effective cone, together with 
	Proposition~\ref{prop:effectivecone}.
	\end{rmk}

\subsection{Harmonic analysis on vector groups} 

	In this section, we recall some of the basic elements of harmonic analysis on adelic 
	vector groups as developed in \cite{Tate}. Let $G = \mathbb G_a^n$.

	For any non-archimedean place $v$ such that the completion $F_v$ is a finite extension 
	of $\mathbb Q_p$, we define the local additive unitary character by
	\[
	\psi_v(x) := \exp (2\pi i\cdot \mathrm{Tr}_{F_v/\mathbb Q_p}(x)).
	\]
	When $v$ is an archimedean place, we define the local additive character by
	\[
	\psi_v(x) := \exp (-2\pi i\cdot  \mathrm{Tr}_{F_v/\mathbb R}(x)).
	\]
	The Euler product $\psi := \prod_v \psi_v$ is an automorphic character of $\mathbb A_F$.

	\begin{lemma}[{\cite[Lemma 10.3]{CLT02}, \cite[Lemma 2.3.1]{CLT12}}]
	\label{lem:10.3_CLT02}
	Let $v\in\Omega_F^{<\infty}$ and let us fix integers $d \geq 0$ and $i \geq 1$. Let $j$ 
	be an integer and $c = \log_{q_v}\#(\mathcal O_v/(d\mathfrak D))$. If $j=0$ we have
	\[
	\frac{1}{\mu(\mathcal O_v)}\int_{\mathcal O_v^\times} \psi_v(\pi_v^{-id+j}x_v^d)\, 
	\mathrm d x_v
	=\begin{cases}
	(1-q_v^{-1}) & \text{ if } d=0, \\ 
	-q_v^{-1} & \text{ if } i=d=1, 
	\\
	0 & \text{ otherwise.} 
	\end{cases}
	\]
	If $j\neq 0$ the integral above vanishes whenever $id - j \geq c +2$. 
	\end{lemma}

	To each adelic point $\bold a \in G(\mathbb A_F)$, we associate the linear functional 
	$f_{\bold a}: G(\mathbb A_F) \to \mathbb A_F$ that sends an element $\bold x$ to the inner product $\bold a \cdot \bold x$, which is the sum of the coordinatewise products in the adelic ring.
	The composition $\psi_{\bold a} = \psi \circ f_{\bold a}$ 
	defines a Pontryagin duality
	\[
	G(\mathbb A_F) \rightarrow G(\mathbb A_F)^\vee, 
	\quad G(F) \rightarrow (G(\mathbb A_F)/G(F))^\vee.
	\]
	(Note that $G(F)$ is discrete and cocompact in $G(\mathbb A_F)$.)

	Given an integrable function $\Phi$ on $G(\mathbb A_F)$, we define its 
	\defi{Fourier transform} by
	\[
	\widehat{\Phi}(\bold a) 
	= \int_{G(\mathbb A_F)} \Phi(\bold x)\psi_{\bold a}(\bold x)\, \mathrm d \bold x.
	\]
	
	\begin{thm}{\textnormal{(\cite[Theorem 4.2.1]{Tate}, Poisson summation)}}
	\label{thm: Poisson}
	Let $\Phi$ be a continuous function on $G(\mathbb A_F)$. Assume that the series
	\[
	\sum_{\bold x \in G(F)} \Phi(\bold x + \bold b)
	\]
	converges absolutely and uniformly when $\bold b$ belongs to a fundamental domain 
	for the quotient $G(\mathbb A_F)/G(F)$, and that the infinite sum
	\[
	\sum_{\bold a \in G(F)} \widehat{\Phi}(\bold a)
	\]
	converges absolutely.
	Then we have
	\begin{equation*}
	\sum_{\bold x \in G(F)} \Phi(\bold x) = \sum_{\bold a \in G(F)} \widehat{\Phi}(\bold a).
	\end{equation*}
	\end{thm}

\section{Height zeta functions} \label{sec:heightzeta}

	In this section, we will establish some basic properties of height zeta functions.
	Let $G = \mathbb G^n_a$ and let $X$ be a smooth equivariant compactification of 
	$G$ defined over a number field $F$. 
	We assume that the boundary $D=X\setminus G$ is a strict normal 
	crossings divisor on $X$. 
	Let $S \subseteq \Omega_F$ be a finite set containing all archimedean places, such that 
	there exists a good integral model $(\mathcal X, \mathcal D)$ of $(X,D)$ over 
	$\Spec \, \mathcal O_{F, S}$ as in \S\ref{subsec:campana}.
	
\subsection{Height functions} 
\label{subsec:height_functions}
	
	We first recall some of the basic properties of height functions, referring to 
	\cite[\S 2]{CLT10} for more details. Let us consider the decomposition of the boundary 
	into irreducible components:
	\[
	D = \bigcup_{\alpha \in \mathcal A} D_\alpha.
	\]
	
	For each $\alpha \in \mathcal{A}$, we fix a smooth adelic metrization on the line bundle 
	$\mathcal{O}(D_\alpha)$,
	and let $\mathsf f_\alpha$ be a section corresponding 
	to $D_\alpha$. For each place $v$, we define the \defi{local height pairing} by
	\[
	\mathsf H_v\colon G(F_v) \times \Pic(X)_{\mathbb C} \rightarrow \mathbb C^\times, 
	\quad \left(\bold x, \sum_{\alpha \in \calA} s_\alpha D_\alpha\right) 
	\mapsto \prod_{\alpha \in \mathcal A} \|\mathsf f_\alpha(\bold x)\|_v^{-s_\alpha}.
	\]
	This pairing varies linearly on the factor $\Pic(X)_{\mathbb C}$ and continuously on 
	the factor $G(F_v)$. We define the \defi{global height pairing} $\mathsf H$ as the product 
	of the local height pairings
	\[
	\mathsf H = \prod_{v \in \Omega_F} \mathsf H_v \colon G(\mathbb A_F) 
	\times \Pic(X)_{\mathbb C} \rightarrow \mathbb C^\times.
	\]
	Again, this pairing varies continuously on the first factor and linearly on the second factor. 
	The following lemma plays a crucial r\^ole in the analysis of height zeta functions in general.
	
	\begin{lemma}{\cite[Proposition 4.2]{CLT02}}
	\label{lemma: K-invariance of height}
	For each non-archimedean place $v \in \Omega_F$, there exists a compact open subgroup 
	$K_v \subset G(\mathcal O_v)$ such that  $\mathsf H_v$ is $K_v$-invariant, that is, 
	such that for any $\bold s \in \Pic(X)_{\mathbb C}$, any $g_v \in G(F_v) \subset X(F_v)$ 
	and any $k_v \in K_v$, we have
	\[
	\mathsf H_v(g_v + k_v, \bold s) = \mathsf H_v (g_v, \bold s).
	\]
	Moreover, if 
	\begin{enumerate}[ref=(\arabic*)]
	\item 
	\label{lemma: K-invariance of height item1}
	the metric $\|\cdot\|_v$ is induced by our integral model $(\mathcal X, \mathcal D)$,
	\item  
	\label{lemma: K-invariance of height item2}
	our $\mathcal{O}_v$-model $(\mathcal{X}\otimes_{\mathcal O_{F,S}}\mathcal O_v, \mathcal D\otimes_{\mathcal O_{F,S}}\mathcal O_v)$ 
	is a smooth, projective, and relative strict normal crossings pair 
	over $\mathcal O_v$~\cite[\S2]{IllusieTemkin},
	and it comes  equipped with an action of the $\mathcal{O}_v$-group scheme 
	$\mathbb{G}^n_{a,\mathcal{O}_v}$ extending the given action of $G$ on $X$, and if	
	\item 
	\label{lemma: K-invariance of height item3}
	the unique linearisation on $\mathcal O(D_\alpha)$ extends to 
	$\mathcal O(\mathcal D_\alpha)$ for every $\alpha \in \mathcal A$,  
	\end{enumerate}
	then we can choose $K_v = G(\mathcal O_v)$. 

	In particular, for all but finitely many places $v \in \Omega_F$,  we may simply take 
	$K_v = G(\mathcal O_v)$. 
	\end{lemma}

\subsection{Intersection multiplicities} 
\label{subsec:intersectionmultiplicities} 

	With the notation introduced above, 
	let $\mathcal{D} = \sum_{\alpha \in \mathcal{A}} \mathcal{D}_\alpha$, 
	where $\mathcal{D}_\alpha$ denotes the closure of $D_\alpha$ in $\mathcal X$ for all $\alpha$.
	Moreover, let 
	$\epsilon = (\epsilon_\alpha)_{\alpha \in \mathcal{A}}$ be a weight vector as in 
	\S \ref{subsec:orbifolds}. Our object of study is 
	$$G(F)_\epsilon = G(F) \cap (\mathcal{X},\mathcal{D}_\epsilon)(\mathcal{O}_{F,S}),$$ 
	the set of $F$-rational points in $G$ which extend to Campana $\mathcal{O}_{F,S}$-points 
	on $(\mathcal{X},\mathcal{D}_\varepsilon)$. For any $v \notin S$, the functions 
	$n_v(\mathcal{D}_\alpha, \cdot)$ defined in \S~\ref{subsec:campana} extend naturally 
	from $G(F)$ to $G(F_v)$. Hence we may define the analogous sets 
	$$G(F_v)_\epsilon = G(F_v) \cap (\mathcal{X},\mathcal{D}_\epsilon)(\mathcal{O}_v).$$ 
	For $v \not\in S$, we denote by $\delta_{\epsilon,v}$ the indicator function detecting 
	whether or not a given point in $G(F_v)$ belongs to the subset $G(F_v)_\epsilon$. 
	For $v \in S$, we simply set $\delta_{\epsilon,v} = 1$. 
	Let $\delta_\epsilon = \prod_{v \in \Omega_F} \delta_{\epsilon, v}$. 

	For $v \notin S$, we have the reduction map
	\[
	\eta_v \colon G(F_v)\subset \mathcal X(\mathcal O_v) \rightarrow \mathcal X(k_v).
	\]
	Given $\bold{x} \in G(F_v)$ and $\alpha \in \mathcal A$, we have 
	$n_v(\mathcal D_\alpha, \bold x) > 0$ if and only if 
	$\eta_v(\bold x) \in \mathcal D_\alpha(k_v)$.
	Let 
	\[
	D_\alpha \otimes_F F_v = \bigcup_{\beta \in \mathcal A_v(\alpha)} D_{v, \beta}
	\]
	be the decomposition of $D_\alpha \otimes_F F_v$ into irreducible components, and let 
	$\mathcal D_{v, \beta}$ be the Zariski closure of $D_{v, \beta}$ in $\mathcal X$. 
	
	Suppose that our integral model has good reduction at $v$ 
	in the sense of Lemma~\ref{lemma: K-invariance of height}, conditions (2) and (3). 
	Since $D_{v,\beta}$ is smooth, if $y \in \mathcal D_{v, \beta}(k_v)$, then Hensel's Lemma 
	implies that $D_{v,\beta}$ has an $F_v$-point, and therefore it is geometrically irreducible 
	over $F_v$. Using a standard argument in Arakelov geometry 
	(see, e.g., \cite[Theorem 2.13]{Sal98} and its proof), we see that there exist analytic 
	local coordinates $(z_1,\cdots,z_n)$ on $\eta_v^{-1}(y)$ mapping to 
	$\mathbb A_{F_{v}}^n$ such that the following conditions are satisfied:
	\begin{itemize}
	\item these local coordinates induce an analytic isomorphism 
	$\eta_v^{-1}(y) \cong \mathfrak m^n_{v}$;
\smallskip
	\item $\eta_v^{-1}(y) \cap D_{v, \beta}(F_v)$ is defined by $z_1 = 0$.
	\end{itemize}

	With this notation, we see that for any $\bold x \in \eta_v^{-1}(y)$, we have 
	$n_v(\mathcal D_{v, \beta}, \bold x) =v(z_1(\bold x))$. Hence, the function 
	$n_v(\mathcal D_\epsilon, \cdot)\colon G(F_v) \rightarrow \mathbb Z_{\geq 0}$ is locally 
	constant for every $v \not\in S$. 
	Moreover since condition (2) in Lemma~\ref{lemma: K-invariance of height}
	is satisfied, the group action of $G(\mathcal O_v)$ preserves $v(z_1(\bold x))$ 
	so that $n_v(\mathcal D_{v, \beta}, \bold x)$ is invariant under the action of $G(\mathcal O_v)$.

	Even if our integral model has bad reduction at $v$, 
	then one can define 
	\[
	H_{\mathcal D_{v, \beta}}(\bold x) = q_v^{n_v(\mathcal D_{v, \beta}, \bold x)},
	\]
	and one may interpret this as a local height function of $\mathcal D_{v, \beta}$ associated 
	to this particular model $\mathcal X_v \to \Spec \, \mathcal O_v$.
	Thus from Lemma~\ref{lemma: K-invariance of height} we deduce the following result:
	
	\begin{lemma} 
	\label{lemma: K-invariance of delta}
	For each non-archimedean place $v \in \Omega_F$, there exists a compact open 
	subgroup $K_v \subset G(\mathcal O_v)$ such that the indicator function 
	$\delta_{\epsilon, v}$ is $K_v$-invariant. If we moreover assume that $v$ satisfies 
	conditions \ref{lemma: K-invariance of height item2} and \ref{lemma: K-invariance of height item3} 
	in Lemma~\ref{lemma: K-invariance of height}, then we can take 
	$K_v = G(\mathcal O_v)$.
	\qed	
	\end{lemma}

	For each non-archimedean place $v$, we denote by $K_v$ a maximal compact open 
	subgroup of $G(\mathcal O_v)$ satisfying the conclusions of 
	Lemma~\ref{lemma: K-invariance of height} and Lemma \ref{lemma: K-invariance of delta}, 
	and we denote $$\bold K = \prod_{v \in \Omega_F^{<\infty}}K_v.$$
	Our discussion shows that both $\mathsf H(\cdot, \bold s)$ and $\delta_\epsilon$ are 
	$\bold K$-invariant. 

\subsection{Height zeta functions} 
\label{subsec:heightzeta}
	To understand the asymptotic 
	formula for the counting function 
	of Campana points of bounded height
	we introduce the 
	\defi{height zeta function}:
	\[
	\mathsf Z_\epsilon(\bold s) 
	= \sum_{\bold x \in G(F)_\epsilon} \mathsf H(\bold x, \bold s)^{-1}
	=  \sum_{\bold x \in G(F)} \mathsf H(\bold x, \bold s)^{-1}\delta_\epsilon(\bold x).
	\]
	The proof of \cite[Proposition 4.5]{CLT02} shows that $\mathsf Z_\epsilon(\bold s)$ 
	is holomorphic when $\Re(\bold s) \gg 0$. The existence of a meromorphic continuation 
	of this zeta function, together with a standard Tauberian theorem, yields a proof of the 
	desired asymptotic formula.
	We therefore consider the Fourier transform
	\[
	\widehat{\mathsf H}_\epsilon(\bold a, \bold s) 
	= \int_{G(\mathbb A_F)} \mathsf H(\bold x, \bold s)^{-1}\,
	\delta_\epsilon (\bold x)\,\psi_{\bold a}(\bold x) \, \mathrm d \bold x,
	\]
	in hopes of using the Poisson summation formula (Theorem~\ref{thm: Poisson})
	\[
	\sum_{\bold x \in G(F)} \mathsf H(\bold x, \bold s)^{-1}\,\delta_\epsilon(\bold x)
	= \sum_{\bold a \in G(F)} \widehat{\mathsf H}_\epsilon(\bold a, \bold s)
	\]
	to obtain the desired meromorphic continuation of $\mathsf Z_\epsilon (\bold s)$. 
	The first two of the three conditions in Theorem~\ref{thm: Poisson} follow from the proof 
	of \cite[Lemma 5.2]{CLT02} assuming that $\Re(\bold s)$ is sufficiently large. 
	To verify the third condition, we recall the following result.

	\begin{prop}[{\cite[Proposition 5.3]{CLT02}}] 
	\label{prop:Lambda_X}
	With the notation introduced above, for all characters $\psi_{\bold a}$ that are non-trivial 
	on $\bold K$ and for all $\bold s$ such that $\mathsf H(\cdot, \bold s)^{-1}$ is integrable, 
	we have $\widehat{\mathsf H}_\epsilon(\bold a, \bold s) = 0$.
	\end{prop}
	
	Let $\Lambda_X \subset G(F)$ be the set of $\bold a$ such that $\psi_{\bold a}$ is 
	trivial on $\bold K$. Then $\Lambda_X$ is a sub-$\mathcal O_F$-module of $G(F)$ of 
	full rank $n$. 
	Indeed, $\Lambda_X$ is a sub-$\mathcal O_F$-module 
	commensurable  with $G(\mathcal O_F)$. 
	To verify the  	third condition 
	in Theorem~\ref{thm: Poisson} we will prove in \S \ref{sec:proofklt} that the sum \[
	\sum_{\bold a \in \Lambda_X} \widehat{\mathsf H}_\epsilon(\bold a, \bold s),
	\]
	is absolutely convergent whenever $\Re(\bold s) \gg 0$.
	Once this is established, we obtain 
	\begin{equation} \label{formula:Z}
	\mathsf Z_\epsilon(\bold s) 
	= \sum_{\bold a \in \Lambda_X} \widehat{\mathsf H}_	\epsilon(\bold a, \bold s),
	\end{equation}
	for $\Re(\bold s) \gg 0$. 

\section*{Interlude I: Dimension $1$}

Let us first make our analysis explicit for $\mathbb{P}^1$ over $\mathbb{Q}$, considered as the natural equivariant compactification of $G = \mathbb{G}_a = \mathbb{A}^1$, with boundary $D = (1:0)$. We fix the standard integral models for $\mathbb P^1$ as well as $D$. Given $\epsilon \in \coeffset$, we consider the problem of counting Campana $\mathbb Z$-points on $({\mathbb P}_{\mathbb Z}^1, \mathcal D_\epsilon)$. Note that if $\epsilon < 1$, then $x \in G(\mathbb{Q}) = \mathbb{Q}$ is a Campana $\mathbb Z$-point if and only if the denominator of $x$ is \textsf{$m$-full}, where $m = 1/(1 - \epsilon)$; this means that any prime dividing the denominator of $x$ occurs with exponent at least $m$ in the prime factorization. If, on the other hand, $\epsilon = 1$, then $x$ is a Campana $\mathbb{Z}$-point if and only if $x \in \mathbb{Z}$. Since the latter case is trivial, we will assume from now on that $\epsilon < 1$.

We fix a finite set of places $S$. Going back to the notation introduced in \S \ref{sec:heightzeta}, we see that we can take $\bold K =\prod_{p \text{ prime }} G(\mathbb Z_p)$ in this case, so that $\Lambda_X = \mathbb Z$. This yields
\[
\mathsf Z_\epsilon (s) = \sum_{n \in \mathbb Z} \widehat{\mathsf H}_\epsilon (n, s).
\]
We would like to compute $\widehat{\mathsf H}_\epsilon(n, s)$ explicitly.
Using Fubini's theorem
we have
\[
\widehat{\mathsf H}_\epsilon (n, s) = \int_{\mathbb A_{F}} \mathsf H (x)^{-s}\,\delta_\epsilon(x)\, \psi( nx)\, \mathrm dx = \prod_{v \in \Omega_{\mathbb Q}} \int_{\mathbb Q_v} \mathsf H_v(x_v)^{-s}\,\delta_{\epsilon, v}(x_v)\, \psi_v(n x_v)\, \mathrm dx_v.
\]
Note that the inner function of each Euler factor is trivial on $\mathbb Z_p$ for almost all places $p$.

We fix metrizations as follows: 
\begin{center}
\begin{tabular}{rcll}
$\mathsf H_v(x_v)$ & $=$ & $\max \{1, |x_v|_v\}$ & if $v$ is non-archimedean, \\  
$\mathsf H_\infty (x_v)$ & $=$ & $\sqrt{1+|x_v|_v^2}$ &  if $v$ is archimedean.
\end{tabular} 
\end{center}

\subsubsection*{The trivial character}
Here we compute $\widehat{\mathsf H}_\epsilon(0, s)$. For any prime $p \notin S$ we have
$$
\widehat{\mathsf H}_{\epsilon,p}(0, s) = \int_{\mathbb Q_p}  \max \{1, |x_p|_p\}^{-s}\,\delta_{\epsilon, p}(x_p)\, \mathrm dx_p
=1+\left(1-\frac{1}{p}\right)\frac{p^{-(s-1)m}}{1-p^{-(s-1)}},
$$ where $m = 1/(1 - \epsilon)$.
On the other hand, if $p \in S$ then
\[
\widehat{\mathsf H}_{\epsilon,p}(0, s) = 1+\left(1-\frac{1}{p}\right)\frac{1}{1-p^{-(s-1)}}.
\]
Furthermore, we have
\[
\widehat{\mathsf H}_{\epsilon,\infty}(0, s)= \frac{\Gamma((s-1)/2)}{\Gamma(s/2)}.
\]
It follows that the rightmost pole of $\widehat{\mathsf H}_\epsilon(0, s)$ is at $s = 1+1/m = 2 - \epsilon$, and that it has order~$1$.

\subsubsection*{Non-trivial characters} 
Let $n$ be a non-zero integer. Our aim is to understand
\[
\widehat{\mathsf H}_\epsilon(n, s) = \prod_{v\in \Omega_{\mathbb Q}} \widehat{\mathsf H}_{\epsilon,v}(n, s), 
\]
where the local factors are given by
\[
\int_{\mathbb Q_v} \mathsf H_v(x_v)^{-s}\,\delta_{\epsilon, v}(x_v)\, \psi_v(nx_v)\, \mathrm d x_v.
\]
Suppose first that $p \notin S$ and $p \nmid n$. The local factor then reduces to
\[
\int_{\mathbb Q_p} \mathsf H_p(x_p)^{-s}\,\delta_{\epsilon, p}(x_p)\,\psi_p(x_p)\, \mathrm d x_p,
\] which equals $$1 + \sum_{i = m}^\infty \left(1-\frac{1}{p}\right)p^{-i(s-1)}\int_{\mathbb Z_p^\times} \psi_p(p^{-i}x_p) \, \mathrm dx_p\\
=
\begin{cases}
1 & \text{ if } \epsilon \neq 0,\\
1- \left(1-\frac{1}{p}\right) p^{-s} & \text{ if } \epsilon = 0.
\end{cases}
$$ Let us now assume that $p \not\in S$ and $p \mid n$, and let us denote the $p$-adic valuation of $n$ by $k$. In this case, the local factor becomes

\begin{align*}
\widehat{\mathsf H}_{\epsilon, p}(n, s) &= 1 + \sum_{i = m}^\infty \left(1-\frac{1}{p}\right)p^{-i(s-1)}\int_{\mathbb Z_p^\times} \psi_p(p^{-i + k}x_p)\,\mathrm dx_p\\
&=
\begin{cases}
1 & \text{ if } m \geq k+2,\\
1- \sum_{i = m}^{k+1} \left(1-\frac{1}{p}\right)p^{-i(s-1)}\int_{\mathbb Z_p^\times} \psi_p(p^{-i+k}x_p) \, \mathrm dx_p& \text{ if } m \leq k+1.
\end{cases}
\end{align*}
When $p \in S$, we recover the formula above for $\epsilon = 0$. 

Using these explicit formulae, we obtain:

\begin{lemma*} Let $p$ be prime. The function $s \mapsto \widehat{\mathsf H}_{\epsilon,p}(n, s)$ is holomorphic everywhere. Moreover, the product
$\prod_{p \ \textup{prime}}\widehat{\mathsf H}_{\epsilon,p}(n, s)$ is holomorphic for $\Re(s) > 1 - \epsilon$, and there exists positive constants $\ell$ and $C$ such that 
$$
\left|\prod_{p\ \textup{prime}}\widehat{\mathsf H}_{\epsilon,p}(n, s)\right| < C\left(1+|s| + |n|\right)^\ell$$ for any $s$ such that $\Re(s)> 1 - \epsilon$. 
\qed 
\end{lemma*} 

Finally we analyze the archimedean place:
\begin{lemma*}
The function $s \mapsto \widehat{\mathsf H}_{\epsilon,\infty} (n, s)$ is holomorphic everywhere.
Moreover, for any integer $N$, there exists positive constants $\ell$ and $C$ such that
\[
\left|\widehat{\mathsf H}_{\epsilon,\infty} (n, s)\right| < C\frac{1+|s|^\ell}{(1+|n|)^N}
\]
 for all $s$. 
 \qed 
 \end{lemma*}

\subsubsection*{Conclusion} Putting all the information together, we obtain that $\mathsf Z_\epsilon(s)$ has a unique pole located at $s = 1+ 1/m = 2 - \epsilon$, contributed by the trivial character. Applying a Tauberian theorem (see, e.g., \cite[II.7,~Theorem~15]{Tenenbaum}), for the line bundle $L = \calO(1)$ metrized as above, we obtain
\[
\mathsf N(G(\Q)_\epsilon, \mathcal L, T) \sim cT^{1 + 1/m}.
\]
for some $c > 0$.

\section{Height integrals I: the trivial character}
\label{sec: height integral}

	In this section, we resume our general analysis and study the height integral 
	\[
	\widehat{\mathsf H}_\epsilon(0, \bold s) 
	= \prod_{v\in \Omega_F} \int_{G(F_v)} \mathsf H_v(\bold x_v, \bold s)^{-1}\, 
	\delta_{\epsilon, v} (\bold x_v) \, \mathrm d \bold x_v
	=: \prod_v \widehat{\mathsf H}_{\epsilon, v}(0, \bold s).
	\]
	Note that the inner function of each Euler factor is trivial on $G(\mathcal O_v)$ for almost all places $v$. 
	We begin by setting up some necessary notation. 
	Each $c \in \mathbb{R}$ gives rise to a tube domain
	\[
	\mathsf{T}_{> c} 
	= \left\{\bold s \in \Pic(X)_{\mathbb C}: 
	\Re(s_\alpha) > \rho_\alpha- \epsilon_\alpha +c\ 
	\text{ for all }  \alpha \in \mathcal A\right\},
	\] 
	where $(\rho_\alpha)_{\alpha \in \mathcal A}$ is the integer vector given by 
	$$-K_X \sim \sum_{\alpha \in \mathcal{A}} \rho_\alpha D_\alpha;$$ recall that 
	$\rho_\alpha \geq 2$ for all $\alpha \in \mathcal{A}$.

	We write
	\[
	D \otimes_F F_v = \bigcup_{\beta \in \mathcal A_v} D_{v, \beta}
	\]
	where the $D_{v, \beta}$ are irreducible components, and we write
	\[
	D_\alpha\otimes_F F_v = \bigcup_{\beta \in \mathcal A_v(\alpha)} D_{v, \beta}.
	\]
	for an analogous decomposition of $D_\alpha\otimes_F F_v$ into irreducible components.

	Given $\beta \in \mathcal A_v$, let us denote by $F_{v,\beta}$ the field of definition for 
	one of the geometric irreducible components of $D_{v,\beta}$, that is, the algebraic closure 
	of $F_v$ inside the function field of $D_{v,\beta}$, and by $f_{v,\beta}$ the extension 
	degree $[F_{v,\beta}:F_v]$. 

	Finally, for any subset $B \subseteq \calA_v$, we define
	\[
	D_{v,B} := \bigcap_{\beta \in B} D_{v,\beta}, \qquad
	D_{v,B}^\circ  := D_{v,B} \setminus \bigcup_{B\subsetneq B' \subset \calA_v}
	\left(\bigcap_{\beta \in B'} D_{v,\beta}\right),
	\]
	with the convention that $D_{v,\emptyset}=X_{F_v}$ and 
	$D_{v,\emptyset}^\circ = G_{F_v}$. 
	The collection $(D_{v,B}^\circ)_{B \subseteq \mathcal{A}_v}$ yields a stratification of 
	the $F_v$-variety $X \otimes_F F_v$ into finitely many locally closed subsets. 
	If $v \not\in S$, then we denote by $\mathcal{D}_{v,B}$ the Zariski closure of $D_{v,B}$ 
	in $\mathcal{X} \otimes_{\mathcal{O}_{F,S}} \mathcal{O}_v$. 
	We define $\mathcal{D}_{v,B}^\circ$ as above. 

\subsection{Places away from $S$} 
\label{subsubsec:placesofgoodreduction}

	We will now study the basic properties of
	\[
	\widehat{\mathsf H}_{\epsilon, v}(0, \bold s) 
	= \int_{G(F_v)} \mathsf H_v(\bold x_v, \bold s)^{-1}\,
	\delta_{\epsilon, v}(\bold x_v) \, \mathrm d \bold x_v
	\]
	in the case that $v\notin S$. 
	
	\subsubsection{Places of good reduction}
	Here we assume that our model 
	$$(\mathcal X_v =\mathcal X\otimes_{\mathcal O_{F,S}}\mathcal O_v, \mathcal D\otimes_{\mathcal O_{F,S}}\mathcal O_v)$$ 
	has good reduction over $\mathcal O_v$ in the sense of 
	Lemma~\ref{lemma: K-invariance of height}, conditions (1) and (2). 
	In this setting we have the following formula 
	which resembles Denef's formula in \cite[Proposition 4.5]{CLT10}:
	\begin{thm}
	We have
	\begin{equation}
	\label{eq:localFTtrivialchar2}
	\frac{1}{\mu_v(\mathcal O_v)^n}\widehat{\mathsf H}_{\epsilon, v}(0,\bold s) 
	= \sum_{B\subset \mathcal A_v}  
	\frac{\#\mathcal{D}_{v, B}^\circ(k_v)}{q_v^{n- \# B}} 
	\prod_{\beta \in B} \left(1-\frac{1}{q_v}\right)
	\frac{q_v^{-m_{\alpha(\beta)}(s_{\alpha(\beta)} - \rho_{\alpha(\beta)} + 1)}}
	{1-q_v^{-(s_{\alpha(\beta)} - \rho_{\alpha(\beta)} + 1)}} .
	\end{equation}
	\end{thm}
	
	\begin{proof}
	
	To avoid clutter, we first assume that 
	$\mu_v(\mathcal O_v) =1$. Set $\boldsymbol{\rho} = (\rho_\alpha)_{\alpha \in\calA}$. 
	Let $\omega$ be a gauge form on $G$, i.e., a nowhere vanishing differential form of top 
	degree. Considering $\omega$ as a rational section of $\mathcal{O}(K_X)$ equipped with 
	the adelic metrization fixed in the previous section, we have the equality
	\[
	\|\omega\|_v = {\mathsf H}_v(\bold x_v,\boldsymbol{\rho}). 
	\] 
	Writing $$\mathrm d\tau = \frac{\mathrm d\bold x_v}{\|\omega\|_v}$$ for the 
	corresponding Tamagawa measure, we see that 
	\begin{align*}
	\widehat{\mathsf H}_{\epsilon, v} (0,\bold s) 
	&= \int_{G(F_v)} \mathsf H_v(\bold x_v, \bold s)^{-1} \, 
	\mathsf H_v(\bold x_v, \boldsymbol{\rho}) \, 
	\delta_{\epsilon, v}(\bold x_v) \, \frac{\mathrm d \bold x_v}{\|\omega\|_v} \\
	&= \int_{G(F_v)} \mathsf H_v(\bold x_v, \bold s - \boldsymbol{\rho})^{-1} \, 
	\delta_{\epsilon, v}(\bold x_v) \, \mathrm d\tau
	\end{align*}
	Breaking up this integral over the fibres of the reduction map 
	$\eta_v\colon G(F_v) \to \calX(k_v)$ we obtain
	\[
	\widehat{\mathsf H}_{\epsilon, v}(0,\bold s)  
	= \sum_{B \subset \calA_v} \sum_{y \in \mathcal{D}_{v,B}^\circ(k_v)} 
	\int_{\eta_v^{-1}(y)} \mathsf H_v(\bold x_v, \bold s - \boldsymbol{\rho})^{-1}\, 
	\delta_{\epsilon, v}(\bold x_v) \, \mathrm d\tau.
	\]
	We now compute the inner integral
	\begin{equation}
	\label{eq:innersumtrivialchar}
	\int_{\eta_v^{-1}(y)} \mathsf H_v(\bold x_v, \bold s - \boldsymbol{\rho})^{-1}\, 
	\delta_{\epsilon, v}(\bold x_v) \, \mathrm d\tau.
	\end{equation}

	If $B = \emptyset$, then there is a measure preserving analytic isomorphism 
	$\eta_v^{-1}(y) \cong \mathfrak{m}_v^n$. Since any $\bold x_v \in \eta_v^{-1}(y)$ 
	is integral with respect to $\mathcal{D}$, we have 
	$$\mathsf H_v(\bold x_v, \bold s - \boldsymbol{\rho}) 
	= \delta_{\epsilon, v}(\bold x_v) = 1$$ for all such $\bold x_v$, so that 
	\eqref{eq:innersumtrivialchar} simply evaluates to $1/q_v^n$.

	If $B \neq \emptyset$, then every $\beta \in B$ lies above a unique 
	$\alpha(\beta) \in \mathcal{A}$. If $\mathcal{D}_{v, B}^\circ(k_v) \neq \emptyset$, 
	then $\mathcal{D}_{v,\beta}(k_v) \neq \emptyset$ for all $\beta \in B$. 
	Using Hensel's lemma, we deduce that $D_{v, \beta}$ has an $F_v$-rational point, 
	and hence is geometrically irreducible; in particular, $F_{v,\beta} = F_v$ for all $\beta\in B$. 
	Writing $B = \{\beta_1,\cdots,\beta_\ell\}$ and $\alpha_i = \alpha(\beta_i)$ for simplicity, 
	we see as in \S \ref{subsec:intersectionmultiplicities} that there exist analytic local 
	coordinates $(z_1,\cdots,z_n)$ on $\eta_v^{-1}(y)$ inducing a measure-preserving 
	analytic isomorphism $\eta_v^{-1}(y) \cong \mathfrak m_v^{n}$, such that 
	$D_{v,\beta_i}(F_v) \cap \eta_v^{-1}(y)$ is given by $z_i = 0$, for $i = 1,\cdots,\ell$. 

	The integral \eqref{eq:innersumtrivialchar} can now be rewritten as
	\[  
	\int_{\eta_v^{-1}(y)} \mathsf H_v(\bold x_v, \bold s - \boldsymbol{\rho})^{-1}\, 
	\delta_{\epsilon, v}(\bold x_v) \, \mathrm d\tau 
	= \int_{\mathfrak m_v^{n}} \prod_{i = 1}^\ell 
	\left(\left|z_{i}\right|_v^{s_{\alpha_i} - \rho_{\alpha_i}}\, \delta_{\epsilon,v} (z_{i}) \, 
	\mathrm dz_{i}\right)\,\prod_{i > \ell} \mathrm dz_i
	\] 
	where 
 	\[
    \delta_{\epsilon,v}(z_{i}) =1 \iff \epsilon_{\alpha_i} \neq 1 \text{ and }   
    \mathrm{val}_v(z_{i}) \geq m_{i} := \frac{1}{1 - \epsilon_{\alpha_i}}  
    \]
	by definition of $ \delta_{\epsilon,v}$ (see \S\ref{subsec:intersectionmultiplicities}).

	Therefore, if $\Re(s_{\alpha_i})-\rho_{\alpha_i}+1>0$ for all $i\in\{1,\dots,\ell\}$, 
	we obtain
	\begingroup
	\allowdisplaybreaks
    \begin{eqnarray*}
   	\int_{\eta_v^{-1}(y)} \mathsf H_v(\bold x_v, \bold s - \boldsymbol{\rho})^{-1}\,  
   	\delta_{\epsilon, v}(\bold x_v) \, \mathrm d\tau 
   	& = &
    \frac{1}{q_v^{n-\ell}} \prod_{i = 1}^\ell 
    \sum_{j = m_{i}}^\infty q_v^{-j (s_{\alpha_i} - \rho_{\alpha_i})}\cdot 
    \Vol(\pi_{v}^j\calO_{v}^\times) \\
    &= & 
    \frac{1}{q_v^{n-\ell}}\prod_{i = 1}^\ell 
    \sum_{j = m_{i}}^\infty q_v^{-j(s_{\alpha_i} - \rho_{\alpha_i})}\cdot 
    q_v^{- j}\left(1 - \frac{1}{q_v}\right) \\
    &=& 
    \frac{1}{q_v^{n-\ell}} \prod_{i = 1}^\ell
    \left(1-\frac{1}{q_v}\right)
    \frac{q_v^{-m_{i}(s_{\alpha_i} - \rho_{\alpha_i} + 1)}}
    {1-q_v^{-(s_{\alpha_i} - \rho_{\alpha_i} + 1)}} ,
    \end{eqnarray*}  
    \endgroup
    where $\pi_v$ denotes a choice of generator for $\mathfrak{m}_v$.

	Summing the contributions coming from different subsets of $\mathcal{A}_v$, we 
	obtain the equality
	\begin{equation}
	\label{eq:localFTtrivialchar}
	\widehat{\mathsf H}_{\epsilon, v}(0,\bold s) = 
	\sum_{B\subset \mathcal A_v}  
	\frac{\#\mathcal{D}_{v, B}^\circ(k_v)}{q_v^{n- \# B}} 
	\prod_{\beta \in B} \left(1-\frac{1}{q_v}\right)
	\frac{q_v^{-m_{\alpha(\beta)}(s_{\alpha(\beta)} - \rho_{\alpha(\beta)} + 1)}}
	{1-q_v^{-(s_{\alpha(\beta)} - \rho_{\alpha(\beta)} + 1)}} .
	\end{equation}
	Here we interpret the term 
	$q_v^{-m_{\alpha(\beta)}(s_{\alpha(\beta)}-\rho_{\alpha(\beta)} + 1)}$ to be zero 
	whenever $\epsilon_{\alpha(\beta)} = 1$.

	When $\mu_v(\mathcal O_v) \neq 1$, the same arguments show our statement.
	\end{proof}
	
	\subsubsection{Places of bad reduction}
	
	Here we still assume that $v\not\in S$, but now our model has bad reduction at $v$,
	i.e., at least one of the assumptions (1) and (2) of 
	Lemma~\ref{lemma: K-invariance of height} is not satisfied.
	We have the following proposition:
	
	\begin{prop}
	\label{prop:badoutsideS}
	The function
	\[
	\widehat{\mathsf H}_{\epsilon, v}(0, \bold s) 
	= \int_{G(F_v)} \mathsf H_v(\bold x_v, \bold s)^{-1}\,
	\delta_{\epsilon, v}(\bold x_v) \, \mathrm d \bold x_v
	\]
	is holomorphic in $\bold s$ whenever $\Re (s_\alpha) > \rho_\alpha-1$ 
	for all $\alpha \in \mathcal A$ such that $\epsilon_{\alpha}<1$.
	\end{prop}
	
	\begin{proof}
	We observe that an application of \cite[Lemma 4.1]{CLT10} with $\Phi=\delta_{\epsilon,v}$ gives 
	holomorphy of $\widehat{\mathsf H}_{\epsilon, v}(0, \bold s)$ whenever $\Re (s_\alpha) > \rho_\alpha-1$ 
	for all $\alpha \in \mathcal A$. 
	Indeed, let $\omega$ be a $G$-invariant top form on $G$. Then we have
	\begin{align*}
	\widehat{\mathsf H}_{\epsilon, v}(0, \bold s) 
	&= \int_{G(F_v)} \mathsf H_v(\bold x_v, \bold s)^{-1}\,
	\delta_{\epsilon, v}(\bold x_v) \, \mathrm d \bold x_v\\
	&= \int_{G(F_v)} \mathsf H_v(\bold x_v, \bold s)^{-1}\,
	\delta_{\epsilon, v}(\bold x_v) \|\omega\|_v\, \frac{\mathrm d \bold x_v}{\|\omega\|_v}\\
	&= \int_{X(F_v)} \mathsf H_v(\bold x_v, \bold s -  \boldsymbol{\rho})^{-1}\,
	\delta_{\epsilon, v}(\bold x_v) \, \mathrm d \tau ,
	\end{align*}
	where $\boldsymbol{\rho} = (\rho_\alpha)_{\alpha \in \mathcal A}$ and 
	$\tau$ is the local Tamagawa measure.
	Next, recall that
	 $$\mathsf H_v(\bold x_v, \bold s - \boldsymbol{\rho})^{-1} 
	 = \prod_{\alpha\in\mathcal A}\|\mathsf f_\alpha(\bold x_v)\|^{s_\alpha - \rho_\alpha}_v,$$ so in the notation of~\cite[Lemma 4.1]{CLT10}, we have 
	 \[
	 \widehat{\mathsf H}_{\epsilon, v}(0, \bold s) = \mathscr{I}\left(\delta_{\epsilon,v};(s_\alpha - \rho_\alpha + 1)_{\alpha\in \mathcal A}\right),
	 \]
	 which is holomophic whenever $\Re (s_\alpha - \rho_\alpha + 1) = \Re (s_\alpha) - \rho_\alpha + 1 > 0$ for $\alpha \in \mathcal A$. Finally, 
	observe that 
	for all $\alpha\in\mathcal A$ such that $\epsilon_\alpha = 1$, the set $D_\alpha(F_v)$ is disjoint from the support 
	of $\delta_{\epsilon, v}$, hence $\|\mathsf f_\alpha\|_v$ is a nowhere vanishing bounded function 
	on $X(F_v)_\epsilon$. Thus the integral that defines $\widehat{\mathsf H}_{\epsilon, v}(0, \bold s)$ 
	is absolutely convergent also for all $\mathbf s$ that satisfy $\Re (s_\alpha) > \rho_\alpha-1$ only for
	$\alpha \in \mathcal A$ such that $\epsilon_\alpha<1$.
	\end{proof}
	
\subsection{Places contained in $S$} 

	Assume now that $v \in S$. In this case, $\delta_{\epsilon, v} \equiv 1$ by definition. 
	Therefore the local height integral for Campana points coincides with the usual local 
	height integral, so that we do not need to do anything new:

	\begin{prop}
	\label{prop:badreduction_trivial}
	The height integral $\widehat{\mathsf H}_v(0, \bold s)$ is holomorphic when 
	$\Re(s_\alpha) > \rho_\alpha -1$ for all $\alpha\in\mathcal A$. 
	If $L = \sum_{\alpha \in \mathcal A} \lambda_\alpha D_\alpha$ is a big divisor on $X$, 
	and if
	\[
	\widetilde{a} := \tilde{a}((X, D_{\mathrm{red}}), L) \quad \text{and} \quad 
	b := b(F_v, (X, D_{\mathrm{red}}), L)
	\]
	(as in \S\ref{subsec:Clemens}), then the function
	\[
	s \mapsto (\zeta_{F_v}(s-\widetilde{a}))^{-b}\cdot \widehat{\mathsf H}_v(0, sL)
	\]
	admits a holomorphic continuation to the domain $\Re (s) > \widetilde{a} -\delta$ for some 
	$\delta > 0$. Moreover, the function $s \mapsto \widehat{\mathsf H}_v(0, sL)$ has a 
	pole at $s = \widetilde{a}$ of order $b$. 
	\end{prop}
	
	\begin{proof}
	One may apply \cite[Lemma 4.1, Proposition 4.3]{CLT10}, taking $\Phi \equiv 1$ on $X(F_v)$. 
	Note that in \cite[Proposition 4.3]{CLT10}, the main term of the local height integral is formed 
	by the contributions of faces of maximal dimension in the analytic Clemens complex; 
	however, these contribute to the pole at $\widetilde a$ all with the same order $b$.
	Also note that there is a typo in \cite[Proposition 4.3]{CLT10}: 
	each product of local zeta functions should be taken over $\alpha \in A$, 
	not $\alpha \in \mathcal A$.
	This means that $D_\alpha$ contains an $F_v$-point so one has $F_\alpha = F_v$ for all $\alpha\in A$.
	\end{proof}

\subsection{Euler products} 
\label{subsec:Eulerproducts}

	Given $\alpha \in \mathcal{A}$, we denote by $F_\alpha$ the field of definition for one 
	of the geometric irreducible components of $D_\alpha$; in other words, $F_\alpha$ is 
	the algebraic closure of $F$ in the function field of $D_\alpha$. 

	\begin{prop}
	\label{prop:EulerOneFactor} 
	Let $v$ be a place of $F$ not contained in $S$ and of good reduction for $(X,D_\epsilon)$. 
	Let $\alpha \in \mathcal{A}$. 
	Write $$D_\alpha \otimes_F F_v = \bigcup_{\beta \in \mathcal A_v(\alpha)} D_{v, \beta}$$ 
	for the decomposition of $D_\alpha \otimes_F F_v$ into irreducible components.
\medskip
	\begin{itemize}
	\item[(1)] For $\delta > 0$ sufficiently small, the function 
	$$\bold{s} \mapsto 
	\prod_{\substack{\alpha \in \mathcal A}} \prod_{\beta \in \mathcal A_v(\alpha)} 
	\zeta_{F_{v, \beta}}(m_\alpha(s_\alpha - \rho_\alpha +1))^{-1}\, 
	\widehat{\mathsf H}_{\epsilon,v}(0, \bold s)$$ 
	is holomorphic on $\mathsf T_{>-\delta}$. (If $\epsilon_\alpha = 1$, we interpret 
	$\zeta_{F_{v, \beta}}(m_\alpha(s_\alpha - \rho_\alpha +1))^{-1}$ to be $1$.)	
\medskip
	\item[(2)] For $\delta > 0$ sufficiently small, there exists $\delta' > 0$ such that 
	\[
	\prod_{\substack{ \alpha \in \mathcal A}} \prod_{\beta \in \mathcal A_v(\alpha)} 
	\zeta_{F_{v, \beta}}(m_\alpha(s_\alpha - \rho_\alpha +1))^{-1}\, 
	\widehat{\mathsf H}_{\epsilon, v}(0, \bold s) = 1 + O(q_v^{-(1+\delta')}),
	\]
	for any $\bold s \in \mathsf T_{>-\delta}$.
	\end{itemize}
	\end{prop}

	\begin{proof}
	We may safely assume that $\mu_v(\mathcal O_v) = 1$. We analyze the right hand side 
	of~\eqref{eq:localFTtrivialchar}, separating the analysis into three cases.
	\begin{itemize}
	\item If $B = \emptyset$, then $\# \mathcal{D}_{v,B}^\circ(k_v) = \# G(k_v) = q_v^n$. 
	Therefore the term corresponding to $B$ in the right hand side of expression 
	\eqref{eq:localFTtrivialchar} for $\widehat{\mathsf H}_{\epsilon, v}(0,\bold s)$ is 
	simply equal to $1$.
\medskip
	\item If $B = \{\beta\}$, define $\alpha(\beta) \in \mathcal{A}$ as in 
	\S\ref{subsubsec:placesofgoodreduction}. If $\mathcal{D}_{v, B}^\circ(k_v) = \emptyset$ 
	or $\epsilon_{\alpha(\beta)}=1$, then $B$ does not contribute to the right hand side 
	of \eqref{eq:localFTtrivialchar}. If, on the other hand, 
	$\mathcal{D}_{v, B}^\circ(k_v) \neq \emptyset$, then 
	$\mathcal{D}_{v,B} \otimes_{\mathcal{O}_v} k_v$ is a geometrically irreducible 
	$k_v$-variety of dimension $n - 1$, so that 
	$$\#\mathcal{D}_{v, B}^\circ(k_v) = q_v^{n-1} + O(q_v^{n-1-\delta_1})$$ 
	for some $\delta_1 > 0$, which may be chosen independently of $\beta$. 
	Therefore by choosing $\delta > 0$ sufficiently small and 
	$\bold{s} \in \mathsf T_{>-\delta}$, the term corresponding to $B = \{\beta\}$ 
	contributes	to the sum in the right hand side of~\eqref{eq:localFTtrivialchar} by
    \[
    q_v^{-m_{\alpha(\beta)}(s_{\alpha(\beta)} - \rho_{\alpha(\beta)} + 1)}
    (1 + O(q_v^{-\delta_2})),
    \]
    for some $\delta_2 > 0$. Since $\delta>0$, we have
    $$\left|q_v^{-m_{\alpha(\beta)}(s_{\alpha(\beta)} - \rho_{\alpha(\beta)} + 1)}\right| 
    \leq q_v^{-(1-m_{\alpha(\beta)} \delta)}$$ 
    whenever $\bold{s} \in \mathsf T_{>-\delta}$. 
    It follows that if we choose $\delta$ sufficiently small and 
    $\bold{s} \in \mathsf T_{>-\delta}$, then the contribution of the term corresponding 
    to $B = \{\beta\}$ can be rewritten as 
    \[
    q_v^{-m_{\alpha(\beta)}(s_{\alpha(\beta)} - \rho_{\alpha(\beta)} + 1)} 
    + O(q_v^{-(1 + \delta')})
    \] for some $\delta' > 0$.
\medskip
	\item Finally, if $\#B \geq 2$, then $\#\mathcal{D}_{v,B}^\circ(k_v) = O(q_v^{n - \#B})$. 
	Moreover, the product in the term in the right hand side of~\eqref{eq:localFTtrivialchar} 
	corresponding to $B$ is $O(q_v^{-(1 + \delta')})$, with $\delta'$ as above, assuming 
	that we have chosen $\bold{s} \in \mathsf T_{>-\delta}$ for $\delta > 0$ sufficiently 
	small. Indeed, each of the factors in the product is bounded from above by 
    $q_v^{-(1-m \delta)}$ for some $m > 0$, as $\bold{s} \in \mathsf T_{>-\delta}$. 
    There are at least two such factors, so the result is bounded from above by 
    $q_v^{-2(1-m \delta)}$ for some $m > 0$, and hence certainly by 
    $q_v^{-(1 + \delta')}$ if $\delta$ is chosen small enough.
	\end{itemize}

\medskip

	We conclude that for $\delta > 0$ small enough and $\bold{s} \in \mathsf T_{>-\delta}$, 
	we have
	\[
	\widehat{\mathsf H}_{\epsilon, v}(0,\bold s) = 1 + \sum_{\substack{\alpha \in \calA}} 
	\sum_{\substack{\beta \in \calA_v(\alpha) \\ f_{v,\beta} = 1}} 
	q_v^{-m_\alpha(s_\alpha - \rho_\alpha + 1)} + O(q_v^{-(1 + \delta')}),
	\] 
	where $f_{v,\beta} = [F_{v,\beta} : F_v]$, and therefore 
	\[
	\widehat{\mathsf H}_{\epsilon, v}(0,\bold s) 
	\prod_{\substack{\alpha \in \calA}} \prod_{\beta \in \calA_v(\alpha)}
	\left( 1 - q_v^{-f_{v,\beta}m_\alpha(s_\alpha - \rho_\alpha + 1)}\right) 
	= 1 + O(q_v^{-(1 + \delta')}).
	\]
	This implies the proposition. 
	\end{proof}

	\begin{cor}
	\label{cor:eulerproduct}
	The function
	\[ 
	\bold{s} \mapsto
	\left(\prod_{\substack{\alpha \in \mathcal A}} 
	\zeta_{F_\alpha}(m_\alpha(s_\alpha -\rho_\alpha +1))^{-1}\right) 
	\prod_{v \notin S}\widehat{\mathsf H}_{\epsilon, v} (0, \bold s)
	\]
	is holomorphic on $\mathsf T_{>-\delta'}$ for sufficiently small $\delta' > 0$.
	\end{cor}
	
	\begin{proof}
	This follows immediately from Proposition~\ref{prop:EulerOneFactor}, 
	and Proposition~\ref{prop:badoutsideS}
	taking into account the fact that
	\[
	F_\alpha \otimes_F F_v \isom \prod_{\beta \in \calA_v(\alpha)} F_{v,\beta}
	\]
	for all $\alpha \in \mathcal{A}$. 
	\end{proof}

\section{Height integrals II: nontrivial characters}
\label{sec: height integral II}

	In this section, we study the height integral
	\[
	\widehat{\mathsf H}_\epsilon(\bold a, \bold s) 
	= \prod_{v\in \Omega_F} \int_{G(F_v)} \mathsf H_v(\bold x_v, \bold s)^{-1}\, 
	\delta_{\epsilon, v} (\bold x_v) \psi_{\bold a, v}(\bold x_v)\, \mathrm d \bold x_v
	=: \prod_{v\in \Omega_F} \widehat{\mathsf H}_{\epsilon, v}(\bold a, \bold s).
	\]
	Note that the inner function of each Euler factor is trivial on $G(\mathcal O_v)$ for almost all places $v$. 
	We introduce some notation.
	For each $\bold a \in G(F)$ with $\bold{a}\neq0$, we denote the linear functional 
	$\bold x \mapsto \bold a \cdot \bold x$ by $f_{\bold a}$, where 
	$\bold a \cdot \bold x$ is the standard inner product.
	Recall from \S \ref{subsec:geometry} that 
	\[
	\mathrm{div}(f_{\bold a}) 
	= E(f_{\bold a}) - \sum_{\alpha \in \mathcal A} d_\alpha(f_{\bold a}) D_\alpha
	\]
	with $d_\alpha(f_{\bold a}) \geq 0$. We define
	\begin{eqnarray*}
	\mathcal A^0(\bold a) & = & 
	\{\alpha \in \mathcal A\mid d_\alpha(f_{\bold a}) = 0\}, 
	\\ 
	\mathcal A^{\geq 1}(\bold a)& = &
	\{\alpha \in \mathcal A\mid d_\alpha(f_{\bold a}) \geq 1\}.
	\end{eqnarray*}
	For any place $v\in \Omega_F$, we define
	\[
	H_v(\bold a) = \max \{|a_1|_v, \dots, |a_n|_v\}
	\]
	and for any non-archimedean place $v$, we take
	\[
	j_v(\bold a) = \min \{v(a_1), \dots, v(a_n)\}
	\]
	so that $H_v(\bold a) = q_v^{-j_v(\bold a)}$. We also define
	\[
	H_{\mathrm{fin}}(\bold a) = \prod_{v \in\Omega_F^{<\infty}} H_v(\bold a), \quad 
	H_\infty(\bold a) = \prod_{v \in\Omega_F^{\infty}}H_v(\bold a).
	\]
	Note that we have
	\begin{equation}
	\label{eq:finite_infinite_heigt_comparison}
	H_\infty(\bold a) \gg H_{\mathrm{fin}}(\bold a)^{-1}.
	\end{equation}

\subsection{Places away from $S$}
\label{subsec:height_intII_goodreduction} 
        
        In this section we assume that $v \not\in S$ and we analyze
        \[
	\widehat{\mathsf H}_\epsilon(\bold a, \bold s) 
	= \prod_{v\in \Omega_F} \int_{G(F_v)} \mathsf H_v(\bold x_v, \bold s)^{-1}\, 
	\delta_{\epsilon, v} (\bold x_v) \psi_{\bold a, v}(\bold x_v)\, \mathrm d \bold x_v
	=: \prod_{v\in \Omega_F} \widehat{\mathsf H}_{\epsilon, v}(\bold a, \bold s).
	\]
	Since 
	$\widehat{\mathsf H}_\epsilon(\bold a, \bold s) = 0$ whenever 
	$\bold a \notin \Lambda_X$ by Proposition \ref{prop:Lambda_X}, we may safely 
	assume that $\bold a \in \Lambda_X$.
	We separate the analysis into the cases of good reduction and bad reduction.

    \subsubsection{Places of good reduction}
	We further assume that
	our model $(\mathcal X, \mathcal D)$ has good reduction at $v$
	in the sense of Lemma~\ref{lemma: K-invariance of height}, conditions (1) and (2).
	We will  distinguish two cases, depending 
	on whether $j_v(\bold a) = 0$ or $j_v(\bold a) \neq 0$; we start with the former case.

	To analyze the integral
	\[
	\widehat{\mathsf H}_{\epsilon, v}(\bold a, \bold s) 
	= \int_{G(F_v)} \mathsf H_v(\bold x_v, \bold s)^{-1} \, \delta_{\epsilon, v} (\bold x_v)\,  
	\psi_{\bold a, v}(\bold x_v)\, \mathrm d \bold x_v
	\]
	in the domain $\mathsf T_{>-\delta}$, we begin by stratifying $G(F_v)$ by the fibers 
	of the reduction map:
	\[
	\frac{1}{\mu_v(\mathcal O_v)^n}\widehat{\mathsf H}_{\epsilon, v}(\bold a,\bold s)  
	= \sum_{B \subset \calA_v} \sum_{y \in \mathcal D_{v,B}^\circ(k_v)} 
	\int_{\eta_v^{-1}(y)} \mathsf H_v(\bold x_v, \bold s - \boldsymbol{\rho})^{-1}\,  	
	\delta_{\epsilon, v}(\bold x_v) \, \psi_{\bold a, v}(\bold x_v) \, \mathrm d\tau.
	\]
	\begin{itemize}
	\item If $B= \emptyset$, then the inner sum is $1$, since 
	$\eta_v^{-1}(\mathcal D_{v,\emptyset}^\circ(k_v))=G(\mathcal O_v)$ and 
	$\bold a \in\Lambda_X$.
\medskip  
	\item If $B =\{\beta\}$, we define $\alpha(\beta)$ as in 
	\S\ref{subsubsec:placesofgoodreduction}. Without loss of generality, we may assume 
	that $D_{v, \beta}$ is geometrically irreducible and that $\epsilon_{\alpha(\beta)}\neq1$.  	
	We distinguish two cases: either $\alpha(\beta) \in~\mathcal A^0(\bold a)$, or 
	$\alpha(\beta) \in~\mathcal A^{\geq 1}(\bold a)$.

	If $\alpha(\beta) \in \mathcal A^0(\bold a)$, then the character $\psi_{\bold a,v}$ 
	becomes trivial on $\eta_v^{-1}(\mathcal D_{v,B}^\circ(k_v))$. Arguing as in the 
	proof of Proposition \ref{prop:EulerOneFactor}, the inner summation contributes
	\[
    q_v^{-m_{\alpha(\beta)}(s_{\alpha(\beta)} - \rho_{\alpha(\beta)} + 1)}
    (1 + O(q_v^{-\delta_1})),
	\]
	for some $\delta_1 >0$, assuming that  
	$\delta >0$ is sufficiently small.

	If, on the other hand, $\alpha(\beta) \in \mathcal A^{\geq 1}(\bold a)$, we set 
	$d:=d_{\alpha(\beta)}(f_{\bold a})$.
	If $y \notin E(f_{\bold a})(k_v)$, we can use Lemma \ref{lem:10.3_CLT02}
	to compute  
	\begingroup
	\allowdisplaybreaks
	\begin{align*}
   	\int_{\eta_v^{-1}(y)} \mathsf H_v(\bold x_v, \bold s - \boldsymbol{\rho}&)^{-1}\, 
   	\delta_{\epsilon, v} (\bold x_v) \, \psi_{\bold a, v}(\bold x_v)\, \mathrm d \tau \\
   	& = \ \ \frac{1}{q_v^{n-1}}\int_{\mathfrak m_v} 
   	|x|_v^{s_{\alpha(\beta)} -\rho_{\alpha(\beta)}} \, 
   	\boldsymbol{1}_{\mathfrak m_v^{m_{\alpha(\beta)}}}(x)\, 
   	\psi_v\left(\frac{1}{x^d}\right)\, \mathrm d x \\
   	& =  \ \ \frac{1}{q_v^{n-1}} \sum_{i = m_{\alpha(\beta)}}^{+\infty} 
   	q_v^{-i(s_{\alpha(\beta)} - \rho_{\alpha(\beta)} + 1)} 
   	\int_{\mathfrak \calO_v^\times}\psi_v\left(\frac{\pi_v^{- id}}{x^d}\right)\, \mathrm d x \\
   	& =  \ \ \frac{1}{q_v^{n-1}} \sum_{i = m_{\alpha(\beta)}}^{+\infty} 
   	q_v^{-i(s_{\alpha(\beta)} - \rho_{\alpha(\beta)} + 1)} 
   	\int_{\mathfrak \calO_v^\times}\psi_v\left(\pi_v^{ - id}x^d\right)\, \mathrm d x \\
   	& = \ \ \begin {cases} 
   		-\frac{1}{q_v^{n}} q_v^{-(s_{\alpha(\beta)} - \rho_{\alpha(\beta)} + 1)} 
   		& \text{ if } d=m_{\alpha(\beta)}=1\\
    		0 & \text{ otherwise}
    	\end{cases}\\
  	&  = \ \  O(q_v^{-(n+\delta_2)})
	\end{align*}
	\endgroup
	for some $\delta_2>0$, for sufficiently small $\delta>0$.

	If $y \in E(f_{\bold a})(k_v)$ and $\delta>0$ is sufficiently small, then we have
	\begin{align*}
	\left|  \int_{\eta_v^{-1}(y)} \right.
	\mathsf H_v(\bold x_v, \bold s - \boldsymbol{\rho} & )^{-1} \, 
	\delta_{\epsilon, v} (\bold x_v) \, \psi_{\bold a, v}(\bold x_v)\, \mathrm d \tau \left. 
	\vphantom{\int_{\eta_v^{-1}(y)}}\right|  \\
	&\leq \ \ \int_{\eta_v^{-1}(y)} 
	\mathsf H_v(\bold x_v, \Re (\bold s) - \boldsymbol{\rho})^{-1} \, 
	\delta_{\epsilon, v} (\bold x_v) \, \mathrm d \tau \\ 
	&= \ \ O(q_v^{-(n-1 + \delta_3)})
	\end{align*}
	for some $\delta_3>0$.
	
	Thus, using the Lang-Weil estimates
	\[ 
	\#(\mathcal D_{v,B}^\circ\setminus E(f_{\bold a}))(k_v)=O(q_v^{n-1}), 
	\quad \#E(f_{\bold a})(k_v)=O(q_v^{n-2}),
	\]
	we obtain
	\[
    	\sum_{y \in \mathcal D_{v,B}^\circ(k_v)} 
    	\int_{\eta_v^{-1}(y)} \mathsf H_v(\bold x_v, \bold s - \boldsymbol{\rho})^{-1} \,  
    	\delta_{\epsilon, v}(\bold x_v)\psi_{\bold a, v}(\bold x_v) \, \mathrm d\tau 
    	= O(q_v^{-(1+\delta_4)})
	\]
	for some $\delta_4 > 0$. 	
\medskip
	\item If $\#B \geq 2$, then arguing as in the proof of 
	Proposition~\ref{prop:EulerOneFactor}, one can show that
	\[
    	\sum_{y \in \mathcal D_{v,B}^\circ(k_v)} 
    	\int_{\eta_v^{-1}(y)} \mathsf H_v(\bold x_v, \bold s - \boldsymbol{\rho})^{-1} \,  
    	\delta_{\epsilon, v}(\bold x_v) \, \psi_{\bold a, v}(\bold x_v) \, \mathrm d\tau 
    	= O(q_v^{-(1+\delta_5)})
	\]
	for some $\delta_5 > 0$ assuming that $\delta>0$ is sufficiently small.
	\end{itemize}
	Combining the estimates above, we obtain the following analogue of 
	Proposition~\ref{prop:EulerOneFactor}.

	\begin{prop}
	\label{prop:goodplaces_nontrivial}
	There exist real numbers $\delta, \delta' > 0$, independent of $\bold a$, 
	such that the function 
	\[
	\bold s \mapsto 
	\left(\prod_{\alpha\in \mathcal A^0(\bold a)} \prod_{\beta \in \mathcal A_v(\alpha)} 
	\zeta_{F_{v,\beta}}(m_\alpha(s_\alpha - \rho_\alpha +1))^{-1} \right)
	\widehat{\mathsf H}_{\epsilon, v} (\bold a,\bold s) 
	\]
	is holomorphic on $\mathsf T_{>-\delta}$, and such that
	\[
	\left(\prod_{\alpha\in \mathcal A^0(\bold a)} \prod_{\beta \in \mathcal A_v(\alpha)} 
	\zeta_{F_{v,\beta}}(m_\alpha(s_\alpha - \rho_\alpha +1))^{-1} \right) 
	\widehat{\mathsf H}_{\epsilon, v} (\bold a,\bold s) 
	= 1 + O(q_v^{-(1+\delta')})
	\]
	for all $s\in \mathsf T_{>-\delta}$. Here we interpret 
	$\zeta_{F_{v, \beta}}(m_\alpha(s_\alpha - \rho_\alpha +1))^{-1}$ to be $1$ 
	whenever $\epsilon_\alpha = 1$. 
	\qed
	\end{prop}

	This finishes the analysis in the case $j_v(\bold a) = 0$. From now on, we assume 
	that $j_v(\bold a) \neq 0$.

	\begin{prop}
	\label{prop:badreduction_outsideS}
	There exists a real number $\delta > 0$, independent of $\bold a$, such that the function
	\[
	\bold s \mapsto 
	\left(\prod_{\alpha\in \mathcal A^0(\bold a)} \prod_{\beta \in \mathcal A_v(\alpha)} 
	\zeta_{F_{v,\beta}}(m_\alpha(s_\alpha - \rho_\alpha +1))^{-1} \right)
	\widehat{\mathsf H}_{\epsilon, v} (\bold a,\bold s),
	\]
	is holomorphic on the domain $\mathsf T_{> -\delta}$.

	Moreover, there exists a real number $\kappa > 0$, independent of $\bold a$, such that
	\[
 	\left |\left(\prod_{\alpha\in \mathcal A^0(\bold a)} \prod_{\beta \in \mathcal A_v(\alpha)} 
 	\zeta_{F_{v,\beta}}(m_\alpha(s_\alpha - \rho_\alpha +1))^{-1} \right)
 	\widehat{\mathsf H}_{\epsilon, v} (\bold a,\bold s) \right| 
 	\ll (1+H_v(\bold a)^{-1} )^\kappa.
	\]
	Here we interpret $\zeta_{F_{v, \beta}}(m_\alpha(s_\alpha - \rho_\alpha +1))^{-1}$ 
	to be $1$ whenever $\epsilon_\alpha = 1$.
	\end{prop}

	\begin{proof}
	As before we use the stratification of $G(F_v)$ by the fibers of the reduction map:
	\[
	\frac{1}{\mu_v(\mathcal O_v)^n}\widehat{\mathsf H}_{\epsilon, v}(\bold a,\bold s)  
	= \sum_{B \subset \calA_v} \sum_{y \in \mathcal D_{v,B}^\circ(k_v)} 
	\int_{\eta_v^{-1}(y)} \mathsf H_v(\bold x_v, \bold s - \boldsymbol{\rho})^{-1} \,  
	\delta_{\epsilon, v}(\bold x_v) \, \psi_{\bold a, v}(\bold x_v) \, \mathrm d\tau.
	\]
	\begin{itemize}
	\item If $B= \emptyset$, the inner summation is holomorphic everywhere and 
	equal to some constant as in \S\ref{subsec:height_intII_goodreduction}.
\medskip
	\item If $B =\{\beta\}$, we define $\alpha(\beta)$ as in 
	\S\ref{subsubsec:placesofgoodreduction}. Without loss of generality, we may assume 
	that $D_{v, \beta}$ is geometrically irreducible and that $\epsilon_{\alpha(\beta)}\neq1$.  
	We again distinguish two cases: either $\alpha(\beta) \in \mathcal A^0(\bold a)$ or 
	$\alpha(\beta)\in \mathcal A^{\geq1}(\bold a)$.
	
	If $\alpha(\beta) \in \mathcal A^0(\bold a)$, the character $\psi_{\bold a,v}$ becomes 
	trivial on $\eta_v^{-1}(\mathcal D_{v,B}^\circ(k_v))$. Hence, arguing as in the proof of 
	Proposition~\ref{prop:EulerOneFactor}, for a sufficiently small $\delta>0$, the inner 
	summation is holomorphic and bounded by
	\[
    	q_v^{-m_{\alpha(\beta)}(s_{\alpha(\beta)} - \rho_{\alpha(\beta)} + 1)}
    	(c + O(q_v^{-\delta_1})),
	\]
	for some constant $c$ and $\delta_1>0$. 
	
 	If on the other hand $\alpha(\beta) \in\mathcal A^{\geq1}(\bold a)$, we denote 
 	$d:=d_{\alpha(\beta)}(f_\bold a)$. If $y \notin E(f_{\bold a})(k_v)$, then we use 
 	Lemma \ref{lem:10.3_CLT02} to compute
 	\begingroup
 	\allowdisplaybreaks
	\begin{align*}
    	\int_{\eta_v^{-1}(y)} \mathsf H_v(\bold x_v, \bold s - \boldsymbol{\rho} & )^{-1} \, 
    	\delta_{\epsilon, v} (\bold x_v) \, \psi_{\bold a, v}(\bold x_v)\, \mathrm d \tau \\
    	&=\ \ \frac{1}{q_v^{n-1}}\int_{\mathfrak m_v} 
    	|x|_v^{s_{\alpha(\beta)} -\rho_{\alpha(\beta)}} \, 
    	\boldsymbol{1}_{\mathfrak m_v^{m_{\alpha(\beta)}}}(x) \, 
    	\psi_v\left(\frac{\pi_v^{j_v(\bold a)}}{x^d}\right)\, \mathrm d x \\
    	&=\ \ \frac{1}{q_v^{n-1}} \sum_{i = m_{\alpha(\beta)}}^{+\infty} 
    	q_v^{-i(s_{\alpha(\beta)} - \rho_{\alpha(\beta)} + 1)} \int_{\mathfrak \calO_v^\times}
    	\psi_v\left(\frac{\pi_v^{- i d + j_v(\bold a)}}{x^d}\right)\, \mathrm d x \\
    	&= \ \ \frac{1}{q_v^{n-1}} \sum_{i = m_{\alpha(\beta)}}^{+\infty} 
    	q_v^{-i(s_{\alpha(\beta)} - \rho_{\alpha(\beta)} + 1)} \int_{\mathfrak \calO_v^\times}	
    	\psi_v\left(\pi_v^{ - i d + j_v(\bold a)}x^d\right)\, \mathrm d x \\
    	&=\ \ O\left( \frac{|j_v(\bold a)|}{q_v^{n-1}} \right)
	\end{align*}
	\endgroup
	We note that the implied constant can be taken independent of $\bold a$; indeed, there 
	are only finitely many possibilities for $d_\alpha(f_{\bold a})$ by 
	Proposition~\ref{prop:d_alpha(f)}. Finally, if $y \in E(f_{\bold a})(k_v)$, then for 
	$\delta>0$ sufficiently small we have 
	\begin{align*}
	\left| \int_{\eta_v^{-1}(y)} \right. 
	\mathsf H_v(\bold x_v, \bold s - \boldsymbol{\rho} & )^{-1} \, 
	\delta_{\epsilon, v} (\bold x_v) \, \psi_{\bold a, v}(\bold x_v)\, \mathrm d \tau \bigg| 
	\phantom{\int_{\eta_v^{-1}(y)}}  \\ 
   	&\leq \ \ \int_{\eta_v^{-1}(y)} 
   	\mathsf H_v(\bold x_v, \Re (\bold s) - \boldsymbol{\rho})^{-1} \, 
   	\delta_{\epsilon, v} (\bold x_v) \, \mathrm d \tau \\ 
   	&=\ \ O(q_v^{-(n-1 + \delta')})
	\end{align*}
	for some $\delta'>0$.
	Thus, using the Lang-Weil estimates as in \S \ref{subsec:height_intII_goodreduction}, 
	we obtain
	\[
    	\sum_{y \in \mathcal D_{v,A}^\circ(k_v)} \int_{\eta_v^{-1}(y)} 
    	\mathsf H_v(\bold x_v, \bold s - \boldsymbol{\rho})^{-1} \, 
    	\delta_{\epsilon, v}(\bold x_v) \, \psi_{\bold a, v}(\bold x_v) \, \mathrm d\tau 
    	= O(|j_v(\bold a)|).
	\]
	\item If $\#B \geq 2$, then as in the proof of Proposition~\ref{prop:EulerOneFactor} we have
	\[
    	\sum_{y \in \mathcal D_{v,B}^\circ(k_v)} 
    	\int_{\eta_v^{-1}(y)} \mathsf H_v(\bold x_v, \bold s - \boldsymbol{\rho})^{-1} \,  
    	\delta_{\epsilon, v}(\bold x_v) \, \psi_{\bold a, v}(\bold x_v) \, \mathrm d\tau 
    	= O(q_v^{-(1+\delta')}).
	\]
	\end{itemize}
	We conclude as in the proof of Proposition \ref{prop:EulerOneFactor}.
	\end{proof}
	
	\subsubsection{Places of bad reduction} 
	
	We still assume that $v \not\in S$ but our model has bad reduction at $v$,
	i.e., at least one of the assumptions (1) and (2) of 
	Lemma~\ref{lemma: K-invariance of height} is not satisfied.
	
	\begin{prop}
	\label{prop:badoutsideSnontrivial}
	The function
	\[
	\widehat{\mathsf H}_{\epsilon, v}(\bold a, \bold s) 
	= \int_{G(F_v)} \mathsf H_v(\bold x_v, \bold s)^{-1}\, 
	\delta_{\epsilon, v} (\bold x_v) \psi_{\bold a, v}(\bold x_v)\, \mathrm d \bold x_v
	\]
	is holomorphic in $\bold s$ whenever $\Re (s_\alpha )> \rho_\alpha -1$ 
	for all $\alpha \in \mathcal A^0(\bold a)$ such that $\epsilon_\alpha < 1$.
	Moreover, for any $\delta > 0$ there exists constants $\kappa, \delta' > 0$ and $C(\delta)> 0$ such that 
	\[
	|\widehat{\mathsf H}_{\epsilon, v}(\bold a, \bold s) | < C(\delta)(1+ |\bold s|)^{\kappa}(1+ H_\infty(\bold a))^{\delta'}
	\]
	whenever $\Re (s_\alpha) > \rho_\alpha -1 + \delta$ 
	for all $\alpha \in \mathcal A^0(\bold a)$ such that $\epsilon_\alpha < 1$. 
	\end{prop}
	
	\begin{proof}
	One may argue as in \cite[Corollary 3.4.4 and Lemma 3.5.2]{CLT12}.
	\end{proof}
	
\subsection{Places contained in $S$} 

	We now treat the remaining places.
	
	\begin{prop}
	\label{prop:placesinS}
	The following hold whenever $v \in S$.
	\begin{enumerate}
	\item
	Let $\delta > 0$ be any positive real number. Then the function
	\[
	\bold s\mapsto \widehat{\mathsf H}_{\epsilon, v} (\bold a,\bold s),
	\]
	is holomorphic in the domain given by $\Re(s_\alpha) > \rho_\alpha -1 +\delta$ 
	for each $\alpha \in \mathcal A$. Moreover, there exists a real number $M_N > 0$, 
	that does not depend on $\bold a$, such that
	\[
	\left| \prod_{v \in S}\widehat{\mathsf H}_{\epsilon, v} (\bold a,\bold s) \right| 
	\ll \frac{(1+ |\bold s|)^{M_N}}{(1+H_\infty (\bold a))^N}.
	\]
	in the above domain.
\medskip
	\item
	Let $L = \sum_{\alpha\in\mathcal A} \lambda_\alpha D_\alpha$ be a big divisor, and let 
	\[
	a := \tilde{a}((X, D_{\mathrm{red}}), L)\quad \text{and}\quad 
	b := b(F_v, (X, D_{\mathrm{red}}), L, f_{\bold a})
	\]
	be the respective $a$- and $b$-invariants of $X$ defined in \S\ref{subsec:Clemens}. 
	Then the function 
	\[
	s \mapsto \left(\zeta_{F_v}(s-a)\right)^{-b}\widehat{\mathsf H}_v(\bold a, sL)
	\]
 	admits a holomorphic continuation to $\Re (s) > a -\delta$ for some $\delta > 0$. 
	Furthermore 
	\[
	\left| \prod_{v\in S} (\zeta_{F_v}(s-a))^{-b}
	\widehat{\mathsf H}_{\epsilon, v} (\bold a,sL) 	\right| 
	\ll_N \frac{(1+ |s|)^{M_N}}{(1+H_\infty (\bold a))^N}.
	\]
	in the above domain.
	\end{enumerate}
	\end{prop}
	
	\begin{proof}
	The first statement is simply \cite[Proposition 8.1]{CLT02}. The second one
 	follows from \cite[Proposition 3.4.4 and Lemma 3.5.2]{CLT12} as well as the discussion 
 	in \cite[\S 3.3.3]{CLT12}. Note that \cite[Proposition 3.4.4]{CLT12} is stated for a 
 	birational modification $Y_{\bold a}$ of $X$, but this does not matter because of 
 	Lemma~\ref{lemm: birationalinvariance_f}.
	\end{proof}

\subsection{Euler products}

	Finally we analyze the product
	\[
	\widehat{\mathsf H}_\epsilon(\bold a, \bold s) 
	= \prod_{v \in \Omega_F} \widehat{\mathsf H}_{\epsilon, v}(\bold a, \bold s).
	\]
	We introduce some notation. For every $\alpha\in\mathcal A$ we set
	\[
	\zeta_{F_\alpha, S^c}(s)
	=\prod_{v \notin S}\prod_{\beta\in\mathcal A_v(\alpha)}\zeta_{F_{v,\beta}}(s).
	\]

	\begin{prop}
	\label{prop:eulerproduct_estimates}
	Assume that $\lfloor D_\epsilon \rfloor = 0$. There is a real number $\delta > 0$, 
	independent of $\bold a$, such that the function 
	\[
	\bold s \mapsto \left(\prod_{\alpha \in \mathcal A^0(\bold a)} 
	\zeta_{F_\alpha, S^c}(m_\alpha(s_\alpha - \rho_\alpha +1)) \right)^{-1} 
	\widehat{\mathsf H}_\epsilon(\bold a, \bold s)
	\]
	is holomorphic on $\mathsf T_{> -\delta}$. 

	Moreover, for any integer $N>0$, there exists a real number $M_N> 0$ such that
	\[
	\left| \left(\prod_{\alpha \in \mathcal A^0(\bold a)} 
	\zeta_{F_\alpha, S^c}(m_\alpha(s_\alpha - \rho_\alpha +1)) \right)^{-1} 
	\widehat{\mathsf H}_\epsilon(\bold a, \bold s) \right|
	\ll \frac{(1 + \|\bold s\|)^{M_N}}{(1+ H_\infty(\bold a))^N}.
	\]
	\end{prop}

	\begin{proof}
	This follows from Propositions~\ref{prop:goodplaces_nontrivial}, 	
	\ref{prop:badreduction_outsideS}, 
	\ref{prop:badoutsideSnontrivial}, 
	and \ref{prop:placesinS}, together with the estimate 
	\eqref{eq:finite_infinite_heigt_comparison}. The implied constant can be chosen 
	independently of $\bold a$, since $\bold a$ belongs to the $\mathcal O_F$-module 
	$\Lambda_X$. 
	\end{proof}
	
\section{Proof of the main result for klt Campana points}
\label{sec:proofklt}

	In this section we prove our main result, Theorem \ref{thm:maintheorem1}. We work in 
	the setting introduced in \S \ref{subsec:results}, recalled here for the reader's convenience. 

	By $X$ we mean a smooth, projective and equivariant compactification of 
	$G = \mathbb G_a^n$, defined over a number field $F$. We assume that the boundary 
	divisor $D = X \setminus G$ is a strict normal crossings divisor on $X$, with irreducible 
	components $(D_\alpha)_{\alpha \in \calA}$, so that 
	$D = \sum_{\alpha\in \calA}D_\alpha$. We denote by $F_\alpha$ 
	the field of definition for one of the geometric irreducible components of $D_\alpha$; in 
	other words, $F_\alpha$ is the algebraic closure of $F$ in the function field of $D_\alpha$. 

	Let $S \subseteq \Omega_F$ be a finite set containing $\Omega_F^\infty$, such that 
	there exists a good integral model $(\mathcal X, \mathcal D)$ of $(X,D)$ over 
	$\Spec \, \mathcal O_{F, S}$ as in \S\ref{subsec:campana}, and let 
	$\mathcal D = \sum_{\alpha \in \mathcal A} \mathcal D_\alpha$. Having fixed 
	$\epsilon_\alpha \in \coeffset$ for each $\alpha \in \calA$, we let 
	$D_\epsilon = \sum_{\alpha \in \mathcal A} \epsilon_\alpha D_\alpha$ and 
	$\calD_\epsilon = \sum_{\alpha \in \mathcal A} \epsilon_\alpha\mathcal D_\alpha$. 
	In this section we assume that the pair $(X, D_\epsilon)$ is Kawamata log terminal 
	(klt for short), that is, $\epsilon_\alpha <1$ for all $\alpha \in \mathcal A$.

	Let $\mathcal{L}$ denote a big line bundle $L$ on $X$, equipped 
	with a smooth adelic metrization. Our goal is to 
	understand the asymptotic behavior of the counting function
	\[
	\mathsf N(G(F)_\epsilon, \mathcal L, T),
	\]
	which records the number of points of $\mathcal L$-height at most $T$ in 
	$G(F)_\epsilon = G(F) \cap (\calX,\calD_\epsilon)(\calO_{F,S})$. To do this, we apply 
	a Tauberian theorem to the height zeta function
	\[
	\mathsf Z_\epsilon(\bold s) 
	= \sum_{\bold x \in G(F)} \mathsf H(\bold x, \bold s)^{-1} \delta_\epsilon(\bold x)
	\]
	introduced in \S\ref{subsec:heightzeta}. This function is a holomorphic function when 
	$\Re(\bold s) \gg 0$; our first goal is to establish a meromorphic continuation of 
	this function. Subsequently, knowledge of the location of the rightmost pole of 
	$\mathsf Z_\epsilon(sL)$ along $\Re(s)$, its order, and its residue will  serve as inputs 
	to the Tauberian theorem that establishes the asymptotic formula we seek.

	Recall that  for any real number $c$ we defined
	\[
	\mathsf T_{> c} = \{\bold s \in \Pic(X)_{\mathbb C}: 
	\Re(s_\alpha) > \rho_\alpha-\epsilon_\alpha +c, \textup{ for all } \alpha \in \mathcal A\},
	\]
	where the $\rho_\alpha$ are integers satisfying 
	${-K}_X \sim \sum_{\alpha \in \calA} \rho_\alpha D_\alpha$.

	\begin{prop}
	\label{prop:meromorphic}
	The function
	\[
	\bold s\mapsto
	\left(\prod_{\alpha \in \mathcal A} 
	\zeta_{F_\alpha}(m_\alpha(s_\alpha -\rho_\alpha +1))^{-1}\right) 
	\mathsf Z_\epsilon(\bold s) 
	\]
	is holomorphic in the region $\mathsf T_{\geq 0}$.
	\end{prop}

	\begin{proof}
	We begin by verifying that the Poisson summation formula
	\begin{equation}
	\label{eq:psf}
	\mathsf Z_\epsilon(\bold s) 
	= \sum_{\bold a \in \Lambda_X} \widehat{\mathsf H}_\epsilon(\bold a, \bold s)
	\end{equation}
	holds for $\Re(\bold s)\gg0$. The discussion in \S\ref{subsec:heightzeta} shows that 
	all that remains to be done is checking that the right hand side converges absolutely. 
	This follows from Proposition~\ref{prop:eulerproduct_estimates}, as 
	$$\sum_{\bold a \in \Lambda_X}\frac 1 {(1+H_{\infty}(\bold a))^N}$$ 
	converges for sufficiently large $N$. The result now follows from an application of 
	Propositions~\ref{prop:badreduction_trivial} and \ref{prop:eulerproduct_estimates} and 
	Corollary \ref{cor:eulerproduct} to the summands of the right hand side of~\eqref{eq:psf}.
	\end{proof}

	\begin{remark}
	\label{rem:klt}
	It is important to note that the local height integrals studied in 
	\S\S\ref{sec: height integral}--\ref{sec: height integral II} have poles along 
	$s_\alpha = \rho_\alpha-1$; however, it follows from Proposition~\ref{prop:meromorphic} 
	that the rightmost pole of $\mathsf Z_\epsilon(\bold s)$ occurs along some 
	$s_\alpha = \rho_\alpha - \epsilon_\alpha > \rho_\alpha-1$, because of the klt condition.
	\end{remark}

	With a meromorphic continuation of $\mathsf Z_\epsilon(\bold s)$ in hand, we turn to the 
	case where $\bold s = sL$. We may write 
	$L = \sum_{\alpha \in \mathcal A} \lambda_\alpha D_\alpha$, where 
	$\lambda_\alpha > 0$ for all $\alpha \in \mathcal A$, because $L$ is big. 
	Then $s_\alpha = s\lambda_\alpha$. Proposition~\ref{prop:meromorphic} suggests that 
	the rightmost pole along $\Re(s)$ of the zeta function $\mathsf Z_\epsilon(sL)$ is 
	\[
	a=a((X, D_\epsilon), L) = \max_{\alpha \in \mathcal A} 
	\left\{  \frac{\rho_\alpha -\epsilon_\alpha}{\lambda_\alpha}\right\}.
	\]
	Setting
	\[
	\mathcal A_\epsilon(L) 
	= \left\{\alpha \in \mathcal A : \frac{\rho_\alpha -\epsilon_\alpha}{\lambda_\alpha} 
	= a((X, D_\epsilon), L) \right\},
	\] 
	the order of this pole should be
	\[
	b=b(F, (X, D_\epsilon), L) := \# \mathcal A_\epsilon(L);
	\]
	see Remark~\ref{rem:klt}. We shall establish these statements, separating our analysis 
	into two cases, according to the Iitaka dimension of the adjoint divisor
	\[
	a L + K_X + D_\epsilon.
	\]

\subsection{Rigid case}
\label{subsec:rigid}

	In this subsection we assume that the adjoint divisor $a L + K_X + D_\epsilon$ has Iitaka 
	dimension 
	(see \cite[\S2.1]{Laz04} for the definition)
	equal to zero; we say that $a L + K_X + D_\epsilon$ is \defi{rigid}. Recall that 
	$\Lambda_X \subset G(F)$ is the set of $\bold a$ such that the character $\psi_{\bold a}$ 
	is trivial on the compact open $\bold K$ defined in \S\ref{subsec:intersectionmultiplicities}.

	By the Poisson summation formula, we have
	\[
	\mathsf Z_\epsilon(sL) = \sum_{\bold a \in \Lambda_X} 
	\widehat{\mathsf H}_\epsilon(\bold a, sL).
	\]
	We study the poles of $\mathsf Z_\epsilon(sL)$ by looking at the individual terms of the 
	right hand side. When $\bold a = 0$, it follows from Corollary~\ref{cor:eulerproduct} that 
	$\widehat{\mathsf H}_\epsilon(0, sL)$ has a pole at $s=a$ of order $b$, provided we 
	show that the corresponding residue is not zero (we verify this last claim presently). 
	On the other hand, Proposition~\ref{prop:eulerproduct_estimates} shows that if $\bold a\neq0$ the term 
	$\widehat{\mathsf H}_\epsilon(\bold a, sL)$ has a pole of the highest order equal to that 
	of $\widehat{\mathsf H}_\epsilon(0, sL)$ if and only if
	\[
	\mathcal A^0(\bold a) \supset \mathcal A_\epsilon(L).
	\]
	This condition means that whenever $(\rho_\alpha - \epsilon_\alpha)/\lambda_\alpha = a$, 
	we must have $d_\alpha(f_{\bold a}) = 0$. Since 
	\[
	E(f_{\bold a}) \sim \sum_{\alpha \in \calA} d_\alpha(f_{\bold a}) D_\alpha
	\quad\text{and}\quad
	aL + K_X + D_\epsilon 
	= \sum_{\alpha\in\calA} (a\lambda_\alpha - \rho_\alpha + \epsilon_\alpha) D_\alpha,
	\]
	it follows that $E(f_{\bold a})$ is equivalent to a boundary divisor whose support is 
	contained in that of the adjoint divisor $aL + K_X + D_\epsilon$. This is not possible.
	Indeed, $aL + K_X + D_\epsilon$ is rigid,
	and any positive linear combination of components of a rigid effective divisor has 
	a unique effective divisor in its $\mathbb Q$-linear equivalence class. However, 
	we showed that the effective divisor $E(f_{\bold a})$, 
	which is not a boundary divisor, is linearly equivalent 
	to an effective boundary divisor $aL + K_X + D$.
	This is a contradiction. Hence, if $\mathbf a\neq0$, 
	the summand $\widehat{\mathsf H}_\epsilon(\bold a, sL)$ does not contribute to 
	the residue of the pole of $\mathsf Z_\epsilon(sL)$ at $s=a$.	

	Our analysis shows that the main term of $\mathsf Z_\epsilon(sL)$ is furnished by 
	$\widehat{\mathsf H}_\epsilon(0, sL)$, provided  
	\[
	c := \lim_{s \rightarrow a} (s-a)^b\,\widehat{\mathsf H}_\epsilon(0, sL).
	\]
	is non-zero, i.e., only the trivial character can contribute to the leading pole of $\mathsf Z_\epsilon(sL)$. 
	Recall that
	\[
	\widehat{\mathsf H}_\epsilon(0, sL) 
	= \int_{G(\mathbb A_F)} \mathsf H(\bold x, sL)^{-1} \, 
	\delta_\epsilon(\bold x) \, \mathrm d \bold x 
	= \int_{G(\mathbb A_F)_\epsilon} \mathsf H(\bold x, sL + K_X)^{-1}\, \mathrm d \tau
	\]
	where $\tau = \prod_v \tau_{v}$ is the Tamagawa measure on $G$. Let $X^\circ 
	= X \setminus \left(\bigcup_{\alpha \not\in \mathcal A_\epsilon(L)} D_\alpha\right)$. 
	Setting $\Gamma = \Gal(\bar F/F)$ and $\Gamma_{F_\alpha} = \Gal(\bar F/F_\alpha)$, 
	we construct the virtual Artin representation
	\[
	P(\bar{X}^\circ) = \mathrm{Pic}(\bar{X})_{\mathbb C} 
	- \sum_{\alpha \not\in \mathcal A_\epsilon(L)} 
	\mathrm{Ind}_{\Gamma_{F_\alpha}}^{\Gamma} \mathbb C.
	\]
	We denote the corresponding virtual Artin $L$-function by
	\[
	L^S(P(\bar{X}^\circ), s) = \prod_{v \not\in S} L_v(P(\bar{X}^\circ), s).
	\]
	This function has a pole of order $\#\mathcal A_\epsilon(L)$ at $s = 1$ by 
	\cite[Corollary 5.47]{IK04}. For $v \in S$ we define $L_v(P(\bar{X}^\circ), s) = 1$.
	Using this we define the Tamagawa measure 
	\begin{equation}
	\label{eq:TamagawaDefn}
	\tau_{X^\circ} = L^S_*(P(\bar{X}^\circ), 1) 
	\prod_{v\in\Omega_F} L_v(P(\bar{X}^\circ), 1)^{-1}\tau_{X^\circ, v},
	\end{equation}
	where $L^S_*(P(\bar{X}^\circ), 1)$ is the leading constant of $L^S(P(\bar{X}^\circ), s)$.
	We also define
	\[
	\tau_{X^\circ, D_\epsilon,v} = \mathsf H_v(\bold x, D_\epsilon) \tau_{X^\circ,v}
	\quad\textrm{and}\quad
	\tau_{X^\circ, D_\epsilon} = \mathsf H(\bold x, D_\epsilon) \tau_{X^\circ}.
	\]
	\begin{lemma}
	\label{lem:rigid_klt_constant}
	With notation as above, we have
	\[
	c = \prod_{\alpha \in \mathcal A_\epsilon(L)} \frac{1}{m_\alpha \lambda_\alpha} 
	\int_{X^\circ(\mathbb A_F)_\epsilon} \mathsf H(\bold x, aL + K_X + D_\epsilon)^{-1}\, 
	\mathrm d \tau_{X^\circ, D_\epsilon} > 0,
	\]
	where $X^\circ(\mathbb A_F)_\epsilon$ is defined  in \S \ref{subsec:leading_constant}.
	\end{lemma}

	\begin{proof}
	First, we note that
	\begingroup
	\allowdisplaybreaks
	\begin{align*}
	c&= \lim_{s \rightarrow a} (s-a)^b\ \widehat{\mathsf H}_\epsilon(0, sL) \\
	&= \lim_{s \rightarrow a} (s-a)^b\prod_{\alpha \in \mathcal A_\epsilon(L)} 
	\zeta_{F_\alpha, S^c}(m_\alpha(\lambda_\alpha s - \rho_\alpha +1)) \\ 
	& \quad\quad \times\int_{G(\mathbb A_F)_\epsilon} 
	\left(\prod_{\alpha \in \mathcal A_\epsilon(L)} 
	\zeta_{F_\alpha, S^c}(m_\alpha(\lambda_\alpha s - \rho_\alpha +1))\right)^{-1}\,
	\mathsf H(\bold x, sL + K_X)^{-1}\,\mathrm d \tau
	\end{align*}
	\endgroup
	For each $\alpha \in \mathcal A_\epsilon(L)$, we have 
	$a = (\rho_\alpha - \epsilon_\alpha)/\lambda_\alpha$, where 
	$\epsilon_\alpha = 1 - 1/m_\alpha$. Each of the $b$-many Dedekind zeta factors 
	$\zeta_{F_\alpha, S^c}(m_\alpha(\lambda_\alpha s - \rho_\alpha +1))$ has a 
	simple pole at $s = a$, so that the limit 
	\[
	\lim_{s \rightarrow a}\, (s-a) \,
	\zeta_{F_\alpha, S^c}(m_\alpha(\lambda_\alpha s - \rho_\alpha +1))
	\]
	is equal to the residue at $s = a$ for the Dedekind zeta factor corresponding to $\alpha$, 
	which we denote by $\zeta_{F_\alpha, S^c}^*(1)/m_\alpha \lambda_\alpha$, where 
	$\zeta_{F_\alpha, S^c}^*(1)$ is the residue of $\zeta_{F_\alpha, S^c}(s)$ at $s = 1$, 
	the normalization $1/m_\alpha \lambda_\alpha$ being a consequence of the chain rule. 
	With the notation
	\[
	\zeta_{F_\alpha, S^c, v}(s)=
	\begin{cases}
	\prod_{\beta\in\mathcal A_v(\alpha)}\zeta_{F_{v,\beta}}(s) & \text { if } v\notin S,\\
	1 & \text{ otherwise}, \end{cases}
	\]
	we rewrite the integral
	\[
	\int_{G(\mathbb A_F)_\epsilon} \left(\prod_{\alpha \in \mathcal A_\epsilon(L)} 
	\zeta_{F_\alpha, S^c}(m_\alpha(\lambda_\alpha s - \rho_\alpha +1))\right)^{-1}\,
	\mathsf H(\bold x, sL + K_X)^{-1}\,\mathrm d \tau
	\]
	as a product of local integrals
	\[
	\prod_{v\in\Omega_F} \int_{G(F_v)_\epsilon}
	\left(\prod_{\alpha \in \mathcal A_\epsilon(L)} 
	\zeta_{F_\alpha, S^c, v}(m_\alpha(\lambda_\alpha s - \rho_\alpha +1)\right)^{-1}\,
	\mathsf H_v(\bold x, aL + K_X)^{-1}\, \mathrm d \tau_{X^\circ,v}
	\]
	each of which is regular at $s = a$ (note that $\tau_v$ and $\tau_{X^0,v}$ coincide on 
	$G$). We obtain
	\begin{align*}
	c &= \prod_{\alpha \in \mathcal A_\epsilon(L)} 
	\frac{1}{m_\alpha \lambda_\alpha}\zeta_{F_\alpha, S^c}^*(1)
	\prod_{v\in\Omega_F} \int_{G(F_v)_\epsilon}
	\left(\prod_{\alpha \in \mathcal A_\epsilon(L)} \zeta_{F_\alpha, S^c, v}(1)\right)^{-1}\,
	\mathsf H_v(\bold x, aL + K_X)^{-1}\, \mathrm d \tau_{X^\circ,v}
	\end{align*}
	Using the equality
	\[
	\prod_{v\in\Omega_F} L_v(P(\bar{X}^\circ), 1)\,
	\left(\prod_{\alpha \in \mathcal A_\epsilon(L)} \zeta_{F_\alpha, S^c, v}(1)\right)^{-1} 
	= L_*^S(P(\bar{X}^\circ), 1)\,
	\left(\prod_{\alpha \in \mathcal A_\epsilon(L)} \zeta_{F_\alpha, S^c}^*(1)\right)^{-1}
	\]
	we may simplify the above expression for $c$ to
	\begin{align}
	\label{eqn:leadingconstant}
	\prod_{\alpha \in \mathcal A_\epsilon(L)} 
	\frac{1}{m_\alpha \lambda_\alpha} L_*^S(P(\bar{X}^\circ), 1)\,
	\prod_v \int_{G(F_v)_\epsilon}
	\mathsf H_v(\bold x, aL + K_X + D_\epsilon)^{-1}\, L_v(P(\bar{X}^\circ), 1)^{-1}\,
	\mathrm d \tau_{X^\circ, D_\epsilon,v}.
	\end{align}
	Finally,~\eqref{eq:TamagawaDefn} allows us to conclude that
	\[
 	c = \prod_{\alpha \in \mathcal A_\epsilon(L)} \frac{1}{m_\alpha \lambda_\alpha} 
	\int_{X^\circ(\mathbb A_F)_\epsilon} \mathsf H(\bold x, aL + K_X + D_\epsilon)^{-1} \,
	\mathrm d \tau_{X^\circ, D_\epsilon} > 0,
	\]
	Let us discuss the positivity of this constant. 
	Recall that this integration is expressed as the Euler product  (\ref{eqn:leadingconstant}).
	The integral at each place is positive as the inner function is positive over some open subset.
	Then a partial Euler product is also positive 
	because of Proposition~\ref{prop:EulerOneFactor} (2). 
	Thus our assertion follows.
	\end{proof}

	Applying a Tauberian theorem (see, e.g., \cite[II.7,~Theorem~15]{Tenenbaum}), we obtain:
	
	\begin{thm}
	\label{theo:main_kltrigid}
	Let $\mathcal{X}$, $\mathcal L$, $\mathcal{D}$ and $\epsilon$ be as above. 
	Assume that $(X, D_\epsilon)$ is klt and set
	\begin{eqnarray*}
	a & = & a((X, D_\epsilon), L), \\
	b & = & b(F, (X, D_\epsilon), L), \\
	c & = & c(F,S,(\mathcal{X},\mathcal{D}_\epsilon),\mathcal{L}) \\ 
	& = &  \prod_{\alpha \in \mathcal A_\epsilon(L)} \frac{1}{m_\alpha \lambda_\alpha} 
	\int_{X^\circ(\mathbb A_F)_\epsilon} 
	\mathsf H(\bold x, aL + K_X + D_\epsilon)^{-1} \, \mathrm d \tau_{X^\circ, D_\epsilon}. 
	\end{eqnarray*}
	If $aL + K_X + D_\epsilon$ is rigid, then
	\[
	\mathsf N(G(F)_\epsilon, \mathcal L, T) 
	\sim \frac{c}{a(b-1)!}T^a(\log T)^{b-1} \text{ as $T \rightarrow \infty$}. 
	\tag*{\qed}
	\] 
	\end{thm}

\subsection{Non-rigid case}

	The analysis in this subsection is modeled on \cite{Tsc03}. With notation as above, 
	we now assume that the divisor $E := aL + K_X +D_\epsilon$ is not rigid, i.e., 
	that its Iitaka dimension is positive. Then some multiple $mE$ defines the Iitaka fibration 
	$\phi_m \colon X \dashrightarrow Y_m$. 
	(See \cite[\S 2.2]{Laz04} for its definition.)
	Since $mE$ admits a $G$-linearization, 
	$Y_m$ admits a natural $G$-action, and $\phi_m$ is $G$-equivariant. 
	For the sake of simplicity, we assume that $\phi_m$ is a \emph{morphism}. 
	The variety $Y_m$ contains an open orbit of the $G$-action, so it has the structure 
	of an equivariant compactification of the quotient vector space $G/G_L$, where 
	$G_L \subset G$ is a linear subspace of $G$. 
	
	As in \S \ref{subsec:rigid}, the term $\widehat{\mathsf H}_\epsilon(\bold a, sL)$ has 
	a pole of highest  order equal to that of $\widehat{\mathsf H}_\epsilon(0, sL)$ if and 
	only if $\mathcal A^0(\bold a) \supset \mathcal A_\epsilon(L)$. This condition is 
	equivalent to having $f_{\bold a} = 0$ on $G_L$. Therefore the rightmost pole
	of $\mathsf Z_\epsilon(sL)$ is furnished by the sum
	\begin{align*}
 	\sum_{\{f_{\bold a} = 0\} \supset G_L} \widehat{\mathsf H}_\epsilon(\bold a, sL) 
 	&=  \sum_{\{f_{\bold a} = 0\} \supset G_L} 
 	\int_{G(\mathbb A_F)} \mathsf H^{-1}(\bold x, sL)\, 
 	\delta_\epsilon(\bold x)\,\psi(\bold a \cdot \bold x)\, \mathrm d \bold x\\
	&= \sum_{\bold y \in (G/G_L)(F)}
	\int_{G_L(\mathbb A_F)} \mathsf H^{-1}(\bold x + \bold y, sL)\, 
	\delta_\epsilon(\bold x + \bold y)\, \mathrm d \bold x,
	\end{align*}
	where the last equality follows from the Poisson summation formula.
	Note that the equality holds for any $s$ with $\Re(s) > a$ by the monotone and 
	dominated convergence theorems.

	Let $X_{\bold y}$ be the fiber of $\phi_m$ above $\bold y$. 
	It is a smooth equivariant compactification of $G_L$, with boundary divisor 
	$D|_{X_{\bold y}}$. Let $\mathcal X_{\bold y}$ be the closure of $X_{\bold y}$ inside 
	$\mathcal X$. The restriction $(aL + K_X + D_\epsilon)|_{X_{\bold y}}$ is rigid, since 
	$\phi_m$ is an Iitaka fibration. Applying the analysis of \S\ref{subsec:rigid}, we conclude 
	that the inner integral has a pole at $s = a((X_{\bold y}, D_\epsilon|_{X_{\bold y}}), L)$ 
	of order $b(F, (X_{\bold y}, D_\epsilon|_{X_{\bold y}}), L)$. 
	Now \cite[Lemma 5.2]{HTT15} yields	
	\[
	a((X, D_\epsilon), L) = a((X_{\bold y}, D_\epsilon|_{X_{\bold y}}), L), \quad\text{and}\quad
	b(F, (X, D_\epsilon), L) = b(F, (X_{\bold y}, D_\epsilon|_{X_{\bold y}}), L).
	\]
	We claim that
	\begin{align*}
	\lim_{s \rightarrow a}(s-a)^b\, \mathsf Z_\epsilon(sL) 
	= \sum_{\bold y \in (G/G_L)(F)} 
	c(F, S, (\mathcal X_{\bold y}, \mathcal D_\epsilon|_{\mathcal X_{\bold y}}), 
	\mathcal L|_{X_{\bold y}}).
	\end{align*} 
	All we need to do is justify the interchange of limits: the right hand side converges by 
	Fatou's lemma, and the claim then follows from the Poisson summation formula 
	(Theorem~\ref{thm: Poisson}).

	As before, applying a Tauberian theorem (\cite[II.7,~Theorem~15]{Tenenbaum}), we obtain:

	\begin{thm}
	\label{theo:main_kltnonrigid}
	Let $X$, $\mathcal L$, $D$ and $\epsilon$ be as above. Assume that $(X, D_\epsilon)$ 
	is klt, and that $m$ is an integer such that the Iitaka fibration 
	$\phi_m \colon X \dashrightarrow Y_m$ defined by $mE$ is a \emph{morphism}. Set
	\begin{eqnarray*}
	a & = & a((X, D_\epsilon), L), \\ b & = & b(F, (X, D_\epsilon), L), \\
	c & = & \sum_{\bold y \in (G/G_L)(F)} 
	c(F, S, (\mathcal X_{\bold y}, \mathcal D_\epsilon|_{\mathcal X_{\bold y}}), 
	\mathcal L|_{X_{\bold y}}).
	\end{eqnarray*}
	Then
	\[
	\mathsf N(G(F)_\epsilon, \mathcal L, T) 
	\sim \frac{c}{a(b-1)!}T^a(\log T)^{b-1} \text{ as $T \rightarrow \infty$}. 
	\tag*{\qed}
	\]
	\end{thm}

\section*{Interlude II: Examples}

	As mentioned in the introduction, Theorem \ref{theo:main_kltrigid} 
	for klt Campana points of bounded log-anticanonical height (i.e. $L=-(K_X+D_\epsilon)$)
	applies 
	to all smooth compactifications of vector groups with strict normal crossing boundary, 
	as $aL+K_X+D_\epsilon$ is always rigid in that case. 
	We recall that there are numerous such compactifications, 
	as blowing up points 
	that are invariant for the action of the vector group on a compactification 
	always produces new examples.

	For the convenience of the reader, we describe two explicit examples 
	to which Theorem \ref{theo:main_kltrigid} 
	applies with $L\neq -(K_X+D_\epsilon)$. 
	Both can be described as blow-ups of a projective space. 
	We describe the set of Campana points in terms of the 
	projective coordinates to show what type of explicit counting problems 
	can be solved using Theorem \ref{theo:main_kltrigid}.

\subsection*{Blow-ups of $\mathbb P^n$}
	
	Let $f\in\mathbb Z[x_0,\dots,x_n]$ be a homogeneous polynomial of degree $d$ 
	such that the subscheme $\{x_0=f=0\}$ of $\mathbb P^n_{\mathbb Z}$ is regular over $\mathbb Z$.
	Let $\varphi\colon\mathcal X \to \mathbb P^n_{\mathbb Z}$ be the blow-up with center 
	$\{x_0=f=0\}$. Let $\mathcal D_1$ be the exceptional divisor and $\mathcal D_2$ 
	the strict transform of $\{x_0=0\}$. 
	We set $\mathcal X^\circ = \mathcal X \setminus (\mathcal D_1 \cup \mathcal D_2)$. 
	
	Fix positive integers $m_1$ and $m_2$, and let 
	$\epsilon_i=1-1/m_i$ for $i\in\{1,2\}$.
	Then $(\mathcal X, \mathcal D_\epsilon)$ is a good integral model of a klt Campana 
	orbifold in the sense of \S\ref{subsec:campana}. 
	By definition of blow-up, the restriction of the morphism $\varphi$ to $\mathcal X^\circ$ is injective.
	Thus, $\varphi$ induces a bijection between 
	$(\mathcal X, \mathcal D_\epsilon)(\mathbb Z)\cap \mathcal X^\circ(\mathbb Q)$
	and the set $A$ of $(n+1)$-tuples $(\tilde x_0,\dots,\tilde x_n)\in\mathbb Z^n$  
	such that 
	\begin{gather*}
	\gcd(\tilde x_0,\dots,\tilde x_n)=1, 
	\quad \tilde x_0>0, 
	\quad \gcd(\tilde x_0, f(\tilde x_0,\dots,\tilde x_n) ) \text{ is $m_1$-full}, \\
	\tilde x_0/\gcd(\tilde x_0, f(\tilde x_0,\dots,\tilde x_n)) \text{ is $m_2$-full}.
	\end{gather*}
	Indeed, given a point $\tilde x\in\PP^n(\mathbb Q)\smallsetminus\{x_0=0\}$, 
	the first two conditions fix a representative for the projective coordinates of $\tilde x$, 
	and given a linear form $\ell \in\mathbb Z[x_0,\dots,x_n]$ such that $\ell(\tilde x)=1$, 
	we can describe explicitly the morphism $\varphi$ over the neighborhood 
	$U_{\ell}:=\PP^n_{\mathbb Z}\smallsetminus\{\ell=0\}$ of $\tilde x$. In particular, 
	$\varphi^{-1}(U_{\ell})=\{y_0 f\ell^{-d}=y_1x_0\ell^{-1}\}\subseteq U_{\ell}\times\PP^1_{\mathbb Z}$,    
	with  coordinates $(y_0:y_1)$ on $\PP^1_\mathbb Z$,
	and the preimage of $\tilde x$ is the point 
	$(\tilde x,(\tilde x_0/\gcd(\tilde x_0,f(\tilde x)): f(\tilde x)/\gcd(\tilde x_0,f(\tilde x)))) 
	\in U_\ell\times\PP^1_{\mathbb Z}$. 
	In a neighborhood of $\varphi^{-1}(\tilde x)$, the equations defining $\mathcal D_1$ as a subscheme of 
	$U_\ell\times\PP^1_{\mathbb Z}$ are $x_0=f=0$, 
	the equations defining $\mathcal D_2$ are $x_0=y_0=0$. 
	So $\varphi^{-1}(\tilde x)\in (\mathcal X, \mathcal D_\epsilon)(\mathbb Z)$ if and only if 
	$\gcd(\tilde x_0,f(\tilde x))$ is $m_1$-full and 
	$\gcd(\tilde x_0,\tilde x_0/\gcd(\tilde x_0,f(\tilde x)))$ is $m_2$-full.

	An application of Theorem \ref{theo:main_kltrigid} with 
	$L=\pi^*\OO_{\PP^n}(1)$ shows that 
	\[
	\#\{(x_0,\dots,x_n)\in A : \max\{|x_0|,\dots,|x_n|\}\leq T\} 
	\sim c T^{n+1/m_2} \text{ as $T \rightarrow \infty$,}
	\]
	for some $c>0$.

\subsection*{A singular del Pezzo surface}

	Let $X$ be the minimal desingularization of a split quartic del Pezzo surface of type $\mathrm D_5$ 
	over $\mathbb Q$. Then $X$ is an equivariant compactification of $\mathbb G_a^2$ 
	by \cite[Lemmas 4 and 6]{DL10}.
	The irreducible components of the boundary on $X$ are the divisors $E_1,\dots,E_6$ 
	from \cite[\S3.4 Type $\mathrm D_5$]{Der14}. 
	We fix coordinates $(x_0:x_1:x_2)$ on $\mathbb P^2_{\mathbb Q}$ and we denote 
	by $\varphi: X\to\mathbb P^2$ the morphism from \cite[\S3.4 Type $\mathrm D_5$]{Der14} 
	that contracts $E_1,E_2,E_4,E_5,E_6$ to the point $(0:0:1)$ and maps $E_3$ onto $\{x_0=0\}$.
	The morphism $\varphi$ is a sequence of five successive blow ups at $\mathbb Q$-points. 
	Performing the same sequence of blow ups over $\mathbb Z$ 
	as in \cite[Proposition 3.9]{FP16} yields a smooth
	projective $\mathbb Z$-model $\mathcal X$ for $X$. 
	For every $i\in\{1,\dots,6\}$, we fix a positive integer $m_i$, we define $\epsilon_i=1-\frac 1{m_i}$, 
	and we denote by $\mathcal E_i$ the closure of $E_i$ in $\mathcal X$.
	Then $(\mathcal X, \sum_{i=1}^6\epsilon_i \mathcal E_i)$ is a good integral model for the 
	klt Campana orbifold $(X, \sum_{i=1}^6 \epsilon_i E_i)$. Let $X^\circ= X\setminus\bigcup_{i=1}^6E_i$.
 
	We use the notation $f(\cdot):=\cdot/\gcd(\cdot, x_1)$ and $g(\cdot):=x_1/\gcd(\cdot, x_1)$, 
	and we denote by $f^{(n)}$ the $n$-th composition of $f$ with itself. We write $h:=f^{(3)}(x_0) x_2^2 + g(f^{(2)}(x_0))g(f(x_0))g(x_0)$.
	Reasoning as in the previous example for each of the five successive blow ups, 
	we see that the set of $\mathbb Z$-Campana points 
	$(\mathcal X, \sum_{i=1}^6\epsilon_i \mathcal E_i)(\mathbb Z)\cap X^\circ(\mathbb Q)$ is in bijection, 
	via $\varphi$, with the set $A$ of triples $(x_0,x_1,x_2)\in\mathbb Z^3$ 
	such that $\gcd(x_0,x_1,x_2)=1$,  $x_0>0, x_1\neq0$ and 
	\begin{gather*}
	\gcd( f^{(2)}(x_0) , g(h) ) \text{ is $m_1$-full}, \\
	x_2^{m_2}\gcd( h , g(f(h))  ) \text{ is $m_2$-full}, \\
	f^{(3)}(x_0) \text{ is $m_3$-full}, \quad
	\gcd( f(x_0) , g(f^{(2)}(x_0)) ) \text{ is $m_4$-full}, \\
	\gcd( x_0 , g(f(x_0)) ) \text{ is $m_5$-full}, \quad
	x_2^{m_6}\gcd( x_1 , f(h) ) \text{ is $m_6$-full}.
	\end{gather*}
	Then an application of Theorem \ref{theo:main_kltrigid} with 
	$L=\varphi^*\OO_{\PP^2}(1)$ shows that 
	\[
	\#\{(x_0,x_1,x_2)\in A : \max\{|x_0|,|x_1|,|x_2|\}\leq T\} 
	\sim c T^{2+1/m_3} \text{ as $T \rightarrow \infty$,}
	\]
	for some $c>0$.
 	
\section{Proof of the main result for dlt Campana points}
\label{sec:proofdlt}

	In this section we sketch the proof of Theorem~\ref{thm:maintheorem2}.
	We use the notation of \S\ref{sec:proofklt}, but this time we assume that 
	$\lfloor D_\epsilon \rfloor \neq 0$, so that $(X, D_\epsilon)$ is not a klt pair. We set
	\begin{eqnarray*}
	\mathcal A^{\mathrm{klt}} & = &  
	\{ \alpha \in \mathcal A \mid \epsilon_\alpha \neq 1 \}, \\
	 \mathcal A^{\mathrm{nklt}} & = &  
	\{ \alpha \in \mathcal A \mid \epsilon_\alpha = 1 \}.
	\end{eqnarray*}
	Let $L = -(K_X + D_\epsilon)$. Arguing as in the proof of 
	Proposition~\ref{prop:meromorphic}, we obtain:
	
	\begin{prop}
	\label{prop:meromorphic_dlt}
	The function
	\[
	s \mapsto
	\left ( \prod_{\alpha\in\mathcal A^{\mathrm{klt}}}
	\zeta_{F_\alpha}(1+m_\alpha(\rho_\alpha-\epsilon_\alpha)(s-1))\right)^{-1} 
	\left(\prod_{v\in S} \zeta_{F_v}(s-1)^{-b(F_v, (X, D_{\mathrm{red}}), L)}\right)
	\mathsf Z_\epsilon(sL) 
	\]
	is holomorphic in the region $\Re (s)\geq 1$.
	\qed
	\end{prop}

	This implies that the zeta function $\mathsf Z_\epsilon(sL)$ possibly has a pole at $s = 1$.

	We define 
	\[
	b(F, S, (X, D_\epsilon), L)
	= \# \mathcal A^{\mathrm{klt}}  + \sum_{v \in S} b(F_v, (X, D_{\mathrm{red}}), L),
	\]
	where the summands on the right are the $b$-invariants defined in \S\ref{subsec:Clemens}.
	Proposition~\ref{prop:badreduction_trivial} and Corollary \ref{cor:eulerproduct} together 
	imply that $\widehat{\mathsf H}_\epsilon(0, sL)$ has a pole at $s = 1$ of order 
	$b(F, S, (X, D_\epsilon), L)$. 

	Arguing as in \cite[Lemma 3.5.4]{CLT12}, we see that the order of the pole of the function 
	$\widehat{\mathsf H}_\epsilon(\bold a, sL)$ at $s = 1$ is strictly less than 
	$b(F, S, (X, D_\epsilon), L)$ when $\bold a \neq 0$. A final application of the  
	Tauberian theorem \cite[II.7,~Theorem~15]{Tenenbaum} then gives the asymptotic 
	formula for the counting function $N(G(F)_\epsilon, \mathcal L, T)$ in the dlt case 
	when $L = -(K_X +D_\epsilon)$:

	\begin{thm}
	\label{theo:main_dlt}
	Let $X$, $D$ and $\epsilon$ be as above. 
	Set
	\[
	L = -(K_X + D_\epsilon), \quad a = 1,  \quad\text{and}\quad b = b(F, S, (X, D_\epsilon), L).
	\]
	Then there exists a constant $c > 0$ that depends on 
	$F,S,(\mathcal{X},\mathcal{D}_\epsilon)$, and  $\mathcal{L}$, such that
	\[
	\mathsf N(G(F)_\epsilon, \mathcal L, T) 
	\sim \frac{c}{a(b-1)!}T^a(\log T)^{b-1} \text{ as $T \rightarrow \infty$}. 
	\tag*{\qed}
	\]
	\end{thm}



\begin{bibdiv}
\begin{biblist}

\bib{Abramovich}{article}{
   	author={Abramovich, D.},
   	title={Birational geometry for number theorists},
   	conference={
      	title={Arithmetic geometry},
   	},
   	book={
      	series={Clay Math. Proc.},
      	volume={8},
      	publisher={Amer. Math. Soc., Providence, RI},
   	},
   	date={2009},
   	pages={335--373},
}

\bib{AKMW}{article}{
    	AUTHOR = {Abramovich, D.}, 
    	author = {Karu, K.}, 
    	author = {Matsuki, K.}, 
    	author = {W\l odarczyk, J.},
    TITLE = {Torification and factorization of birational maps},
   	JOURNAL = {J. Amer. Math. Soc.},
  	FJOURNAL = {Journal of the American Mathematical Society},
    	VOLUME = {15},
    	YEAR = {2002},
    	NUMBER = {3},
    	PAGES = {531--572},
    	ISSN = {0894-0347},
}

\bib{AVA16}{article}{
	author = {Abramovich, D.},
	author = {V\'arilly-Alvarado, A.},
    TITLE = {Campana points, {V}ojta's conjecture, and level structures on
              semistable abelian varieties},
   	JOURNAL = {J. Th\'{e}or. Nombres Bordeaux},
  	FJOURNAL = {Journal de Th\'{e}orie des Nombres de Bordeaux},
    	VOLUME = {30},
    	YEAR = {2018},
    	NUMBER = {2},
    	PAGES = {525--532},
    	ISSN = {1246-7405},
}

\bib{BCHM}{article}{
    	AUTHOR = {Birkar, C.},
    	author= {Cascini, P.},
    	author = {Hacon, C. D.},
    	author = {McKernan, J.},
   	TITLE = {Existence of minimal models for varieties of log general type},
   	JOURNAL = {J. Amer. Math. Soc.},
  	FJOURNAL = {Journal of the American Mathematical Society},
    	VOLUME = {23},
    	YEAR = {2010},
    	NUMBER = {2},
    	PAGES = {405--468},
    	ISSN = {0894-0347},
}

\bib{BG58}{article}{
    AUTHOR = {Bateman, P.~T.},
    author = {Grosswald, E.},
   	TITLE = {On a theorem of {E}rd\H{o}s and {S}zekeres},
   	JOURNAL = {Illinois J. Math.},
  	FJOURNAL = {Illinois Journal of Mathematics},
    	VOLUME = {2},
   	YEAR = {1958},
     PAGES = {88--98},
    	ISSN = {0019-2082},
}

\bib{Bilu}{article}{
	AUTHOR = {Bilu, M.},
	Title= {Motivic {E}uler products and motivic height zeta functions},
	year = {2018},
	note = {PhD thesis, Universit\'e Paris Saclay, to appear in Mem. Amer. Math. Soc.,  arXiv:1802.06836},
}

\bib{Bir19}{article}{
    AUTHOR = {Birkar, C.},
     TITLE = {Anti-pluricanonical systems on {F}ano varieties},
   JOURNAL = {Ann. of Math. (2)},
  FJOURNAL = {Annals of Mathematics. Second Series},
    VOLUME = {190},
      YEAR = {2019},
    NUMBER = {2},
     PAGES = {345--463},
}

\bib{Bir16}{article}{
    AUTHOR = {Birkar, C.},
     TITLE = {Singularities of linear systems and boundedness of Fano varieties},
      YEAR = {2016},
    NOTE = {arXiv:1609.05543}
}

\bib{BL17}{article}{
    	AUTHOR = {Browning, T. D.},
    	author= {Loughran, D.},
   	TITLE = {Varieties with too many rational points},
   	JOURNAL = {Math. Z.},
  	FJOURNAL = {Mathematische Zeitschrift},
    	VOLUME = {285},
   	YEAR = {2017},
    	NUMBER = {3-4},
   	PAGES = {1249--1267},
   	ISSN = {0025-5874},
}

\bib{BM90}{article}{
   author={Batyrev, V. V.},
   author={Manin, Yu. I.},
   title={Sur le nombre des points rationnels de hauteur born\'{e} des vari\'{e}t\'{e}s
   alg\'{e}briques},
   journal={Math. Ann.},
   volume={286},
   date={1990},
   number={1-3},
   pages={27--43},
   issn={0025-5831},
}

\bib{BO12}{article}{
   	AUTHOR = {Benoist, Y.},
   	author = {Oh, H.},
   	TITLE = {Effective equidistribution of {$S$}-integral points on
              symmetric varieties},
   	JOURNAL = {Ann. Inst. Fourier (Grenoble)},
  	FJOURNAL = {Universit\'e de Grenoble. Annales de l'Institut Fourier},
    	VOLUME = {62},
     YEAR = {2012},
    	NUMBER = {5},
     PAGES = {1889--1942},
    	ISSN = {0373-0956},
}

\bib{BT96a}{article}{
    AUTHOR = {Batyrev, V.},
    author= {Tschinkel, Yu.},
   	TITLE = {Height zeta functions of toric varieties},
    	NOTE = {Algebraic geometry, 5},
   	JOURNAL = {J. Math. Sci.},
  	FJOURNAL = {Journal of Mathematical Sciences},
    	VOLUME = {82},
   	YEAR = {1996},
    	NUMBER = {1},
   	PAGES = {3220--3239},
    	ISSN = {1072-3374},
}
	
\bib{BT96b}{article}{
    AUTHOR = {Batyrev, V.},
    author= {Tschinkel, Yu.},
    	TITLE = {Rational points on some {F}ano cubic bundles},
   	JOURNAL = {C. R. Acad. Sci. Paris S\'er. I Math.},
  	FJOURNAL = {Comptes Rendus de l'Acad\'emie des Sciences. S\'erie I.
              Math\'ematique},
    	VOLUME = {323},
    	YEAR = {1996},
    	NUMBER = {1},
   	PAGES = {41--46},
    	ISSN = {0764-4442},
}

\bib{BT98}{article}{
    	AUTHOR = {Batyrev, V.},
    	author = {Tschinkel, Yu.},
    	TITLE = {Manin's conjecture for toric varieties},
   	JOURNAL = {J. Algebraic Geom.},
  	FJOURNAL = {Journal of Algebraic Geometry},
    	VOLUME = {7},
   	YEAR = {1998},
    	NUMBER = {1},
   	PAGES = {15--53},
    	ISSN = {1056-3911},
}

\bib{TamagawaBT}{article}{,
    	AUTHOR = {Batyrev, V.},
    	author = {Tschinkel, Yu.},
   	TITLE = {Tamagawa numbers of polarized algebraic varieties},
    	NOTE = {Nombre et r\'epartition de points de hauteur born\'ee (Paris,
              1996)},
   	JOURNAL = {Ast\'erisque},
  	FJOURNAL = {Ast\'erisque},
    	NUMBER = {251},
    	YEAR = {1998},
     PAGES = {299--340},
    	ISSN = {0303-1179},
}

\bib{BVV12}{article}{
    	AUTHOR = {Browning, T. D.},
    	author = {Van Valckenborgh, K.},
    	TITLE = {Sums of three squareful numbers},
   	JOURNAL = {Exp. Math.},
  	FJOURNAL = {Experimental Mathematics},
    	VOLUME = {21},
    	YEAR = {2012},
    	NUMBER = {2},
   	PAGES = {204--211},
   	ISSN = {1058-6458},
}

\bib{BY19}{article}{
	author = {Browning, T. D.},
	AUTHOR = {Yamagishi, S.},
	Title= {Arithmetic of higher-dimensional orbifolds and a mixed {W}aring problem},
	year = {2019},
	note = {Preprint, arXiv:1902.07782},
}
	
\bib{Campana04}{article}{,
    	AUTHOR = {Campana, F.},
   	TITLE = {Orbifolds, special varieties and classification theory},
   	JOURNAL = {Ann. Inst. Fourier (Grenoble)},
  	FJOURNAL = {Universit\'e de Grenoble. Annales de l'Institut Fourier},
    	VOLUME = {54},
   	YEAR = {2004},
    	NUMBER = {3},
   	PAGES = {499--630},
}

\bib{MR2163487}{article}{
   	author={Campana, F.},
   	title={Fibres multiples sur les surfaces: aspects geom\'{e}triques,
   	hyperboliques et arithm\'{e}tiques},
   	journal={Manuscripta Math.},
   	volume={117},
   	date={2005},
   	number={4},
   	pages={429--461},
  	 issn={0025-2611},
}

\bib{MR2831280}{article}{
   	author={Campana, F.},
   	title={Orbifoldes g\'{e}om\'{e}triques sp\'{e}ciales et classification bim\'{e}romorphe
   	des vari\'{e}t\'{e}s k\"{a}hl\'{e}riennes compactes},
   	journal={J. Inst. Math. Jussieu},
   	volume={10},
   	date={2011},
   	number={4},
   	pages={809--934},
   	issn={1474-7480},
}

\bib {Campana15}{incollection}{
    	AUTHOR = {Campana, F.},
    	TITLE = {Special manifolds, arithmetic and hyperbolic aspects: a short
              survey},
 	BOOKTITLE = {Rational points, rational curves, and entire holomorphic
              curves on projective varieties},
    	SERIES = {Contemp. Math.},
    	VOLUME = {654},
    	PAGES = {23--52},
 	PUBLISHER = {Amer. Math. Soc., Providence, RI},
   	YEAR = {2015},
}

\bib{Chow}{article}{
	AUTHOR = {Chow, D.},
	Title= {The {D}istribution of {I}ntegral {P}oints on the {W}onderful {C}ompactification 
			by {H}eight},
	year = {2019},
	note = {PhD thesis, University of Illinois at Chicago, arXiv:1903.07232},
}

\bib{CLL16}{article}{,
    	AUTHOR = {Chambert-Loir, A.},
    	author = {Loeser, F.},
   	TITLE = {Motivic height zeta functions},
   	JOURNAL = {Amer. J. Math.},
  	FJOURNAL = {American Journal of Mathematics},
    	VOLUME = {138},
   	YEAR = {2016},
    	NUMBER = {1},
   	PAGES = {1--59},
   	ISSN = {0002-9327},
}

\bib{CLT02}{article}{,
    	AUTHOR = {Chambert-Loir, A.},
    	author = {Tschinkel, Yu.},
   	TITLE = {On the distribution of points of bounded height on equivariant
              compactifications of vector groups},
   	JOURNAL = {Invent. Math.},
  	FJOURNAL = {Inventiones Mathematicae},
    	VOLUME = {148},
    	YEAR = {2002},
    	NUMBER = {2},
   	PAGES = {421--452},
   	ISSN = {0020-9910},
}

\bib{CLT10}{article}{
    	AUTHOR = {Chambert-Loir, A.},
    	AUTHOR = {Tschinkel, Yu.},
   	TITLE = {Igusa integrals and volume asymptotics in analytic and adelic
              geometry},
   	JOURNAL = {Confluentes Math.},
  	FJOURNAL = {Confluentes Mathematici},
    	VOLUME = {2},
    	YEAR = {2010},
    	NUMBER = {3},
    	PAGES = {351--429},
    	ISSN = {1793-7442},
}

\bib{CLT10b}{article}{
	AUTHOR = {Chambert-Loir, A.},
    AUTHOR = {Tschinkel, Yu.},
	Title= {Integral points of bounded height on toric varieties},
	year = {2010},
	note = {Preprint, arXiv:1006.3345},
}

\bib{CLT12}{article}{
    	AUTHOR = {Chambert-Loir, A.},
    	author = {Tschinkel, Yu.},
    	TITLE = {Integral points of bounded height on partial equivariant
              compactifications of vector groups},
   	JOURNAL = {Duke Math. J.},
  	FJOURNAL = {Duke Mathematical Journal},
    	VOLUME = {161},
   	YEAR = {2012},
    	NUMBER = {15},
   	PAGES = {2799--2836},
    	ISSN = {0012-7094},
}

\bib{Coc19}{article}{
    	AUTHOR = {Coccia, S.},
   	TITLE = {The {H}ilbert property for integral points of affine smooth
              cubic surfaces},
   	JOURNAL = {J. Number Theory},
  	FJOURNAL = {Journal of Number Theory},
    	VOLUME = {200},
    	YEAR = {2019},
    	PAGES = {353--379},
    	ISSN = {0022-314X},
}

\bib{CT03}{article}{
   	author={Colliot-Th{\'e}l{\`e}ne, J.-L.},
   	title={Points rationnels sur les fibrations},
   	conference={
    		title={Higher dimensional varieties and rational points},
    		address={Budapest},
    		date={2001},
   	},
   	book={
      	series={Bolyai Soc. Math. Stud.},
      	volume={12},
      	publisher={Springer},
      	place={Berlin},
   	},
   	date={2003},
   	pages={171--221},
}

\bib{Der14}{article}{
    	AUTHOR = {Derenthal, U.},
   	TITLE = {Singular del {P}ezzo surfaces whose universal torsors are
              hypersurfaces},
   	JOURNAL = {Proc. Lond. Math. Soc. (3)},
   	FJOURNAL = {Proceedings of the London Mathematical Society. Third Series},
   	VOLUME = {108},
   	YEAR = {2014},
    	NUMBER = {3},
    	PAGES = {638--681},
   	ISSN = {0024-6115},
}

\bib{DL10}{article}{,
    	AUTHOR = {Derenthal, U.},
    	author = {Loughran, D.},
   	TITLE = {Singular del {P}ezzo surfaces that are equivariant
              compactifications},
   	JOURNAL = {Zap. Nauchn. Sem. S.-Peterburg. Otdel. Mat. Inst. Steklov.
              (POMI)},
  	FJOURNAL = {Rossi\u\i skaya Akademiya Nauk. Sankt-Peterburgskoe Otdelenie.
              Matematicheski\u\i \ Institut im. V. A. Steklova. Zapiski
              Nauchnykh Seminarov (POMI)},
   	VOLUME = {377},
   	YEAR = {2010},
    	NUMBER = {Issledovaniya po Teorii Chisel. 10},
    	PAGES = {26--43, 241},
   	ISSN = {0373-2703},
}

\bib{DL15}{article}{,
    	AUTHOR = {Derenthal, U.},
    	author = {Loughran, D.},
    	TITLE = {Equivariant compactifications of two-dimensional algebraic
              groups},
   	JOURNAL = {Proc. Edinb. Math. Soc. (2)},
  	FJOURNAL = {Proceedings of the Edinburgh Mathematical Society. Series II},
    	VOLUME = {58},
   	YEAR = {2015},
    	NUMBER = {1},
   	PAGES = {149--168},
    	ISSN = {0013-0915},
}

\bib{DRS93}{article}{
    	AUTHOR = {Duke, W.},
    	author = {Rudnick, Z.},
   	author = {Sarnak, P.},
    	TITLE = {Density of integer points on affine homogeneous varieties},
   	JOURNAL = {Duke Math. J.},
  	FJOURNAL = {Duke Mathematical Journal},
    	VOLUME = {71},
   	YEAR = {1993},
    	NUMBER = {1},
   	PAGES = {143--179},
   	ISSN = {0012-7094},
}

\bib{erdos}{article}{
	author = {Erd\H{o}s, P.},
	author = {Szekeres, G.},
	title = {{\"U}ber die {A}nzahl der {A}belschen {G}ruppen gegebener {O}rdnung und \"uber ein verwandtes zahlentheoretisches {P}roblem},
	journal = {Acta Univ. Szeged. Sect. Sci. Math.},
	volume = {7},
	year = {1934},
	number = {2},
	pages = {92 -- 102},
}

\bib{EM93}{article}{
    	AUTHOR = {Eskin, A.},
    	author = {McMullen, C.},
   	TITLE = {Mixing, counting, and equidistribution in {L}ie groups},
   	JOURNAL = {Duke Math. J.},
  	FJOURNAL = {Duke Mathematical Journal},
    	VOLUME = {71},
    	YEAR = {1993},
    	NUMBER = {1},
   	PAGES = {181--209},
    	ISSN = {0012-7094},
}

\bib{FMT89}{article}{,
    	AUTHOR = {Franke, J.},
    	Author = {Manin, Yu. I.},
    	AUTHOR = {Tschinkel, Yu.},
     	TITLE = {Rational points of bounded height on {F}ano varieties},
   	JOURNAL = {Invent. Math.},
  	FJOURNAL = {Inventiones Mathematicae},
    	VOLUME = {95},
    	YEAR = {1989},
    	NUMBER = {2},
   	PAGES = {421--435},
    	ISSN = {0020-9910},
}

\bib{FP16}{article}{
   author={Frei, C.},
   author={Pieropan, M.},
   title={O-minimality on twisted universal torsors and Manin's conjecture
   over number fields},
   journal={Ann. Sci. \'{E}c. Norm. Sup\'{e}r. (4)},
   volume={49},
   date={2016},
   number={4},
   pages={757--811},
   issn={0012-9593},
}

\bib{GMO08}{article}{
    	AUTHOR = {Gorodnik, A.},
    	author = {Maucourant, F.}, 
    	author = {Oh, H.},
   	TITLE = {Manin's and {P}eyre's conjectures on rational points and
              adelic mixing},
   	JOURNAL = {Ann. Sci. \'Ec. Norm. Sup\'er. (4)},
  	FJOURNAL = {Annales Scientifiques de l'\'Ecole Normale Sup\'erieure. Quatri\`eme
              S\'erie},
    	VOLUME = {41},
   	YEAR = {2008},
    	NUMBER = {3},
   	PAGES = {383--435},
    	ISSN = {0012-9593},
}

\bib{GO11}{article}{
    	AUTHOR = {Gorodnik, A.},
    	author =  {Oh, H.},
    	TITLE = {Rational points on homogeneous varieties and equidistribution
              of adelic periods},
   	NOTE = {With an appendix by Mikhail Borovoi},
   	JOURNAL = {Geom. Funct. Anal.},
  	FJOURNAL = {Geometric and Functional Analysis},
   	VOLUME = {21},
    	YEAR = {2011},
    	NUMBER = {2},
    	PAGES = {319--392},
    	ISSN = {1016-443X},
}

\bib{GTBT15}{article}{
    	AUTHOR = {Gorodnik, A.},
    	author =  {Takloo-Bighash, R.},
    	author = {Tschinkel, Yu.},
   	TITLE = {Multiple mixing for adele groups and rational points},
   	JOURNAL = {Eur. J. Math.},
  	FJOURNAL = {European Journal of Mathematics},
    	VOLUME = {1},
    	YEAR = {2015},
    	NUMBER = {3},
    	PAGES = {441--461},
    	ISSN = {2199-675X},
}

\bib{HM07}{article}{
    AUTHOR = {Hacon, C. D.}
    author = {Mckernan, J.},
     TITLE = {On {S}hokurov's rational connectedness conjecture},
   JOURNAL = {Duke Math. J.},
  FJOURNAL = {Duke Mathematical Journal},
    VOLUME = {138},
      YEAR = {2007},
    NUMBER = {1},
     PAGES = {119--136},
      ISSN = {0012-7094},
}


\bib{HM18}{article}{
	AUTHOR = {Huang, Z.},
	author = {Montero, P.},
	Title= {Fano threefolds as equivariant compactifications of the vector group},
	journal = {Michigan Math. J.},
	volume = {69},
	year = {2018},
	number = {2},
	pages = {341--368},
	ISSN = {0026-2285},
}

\bib{HT99}{article}{
    	AUTHOR = {Hassett, B.},
    	author={Tschinkel, Yu.},
     TITLE = {Geometry of equivariant compactifications of {${\bf G}_a^n$}},
   	JOURNAL = {Internat. Math. Res. Notices},
  	FJOURNAL = {International Mathematics Research Notices},
    	YEAR = {1999},
    	NUMBER = {22},
    	PAGES = {1211--1230},
    	ISSN = {1073-7928},
}

\bib{HTT15}{article}{
    	AUTHOR = {Hassett, B.},
    	author = {Tanimoto, S.},
    	author = {Tschinkel, Yu.},
   	TITLE = {Balanced line bundles and equivariant compactifications of
              homogeneous spaces},
   	JOURNAL = {Int. Math. Res. Not. IMRN},
  	FJOURNAL = {International Mathematics Research Notices. IMRN},
    	YEAR = {2015},
    	NUMBER = {15},
   	PAGES = {6375--6410},
    	ISSN = {1073-7928},
}
	
\bib{IllusieTemkin}{article}{
   author={Illusie, L.},
   author={Temkin, M.},
   title={Expos\'{e} X. Gabber's modification theorem (log smooth case)},
   note={Travaux de Gabber sur l'uniformisation locale et la cohomologie
   \'{e}tale des sch\'{e}mas quasi-excellents},
   journal={Ast\'{e}risque},
   number={363-364},
   date={2014},
   pages={167--212},
}


\bib{IK04}{book}{
    	AUTHOR = {Iwaniec, H.},
    	author = {Kowalski, E.},
     TITLE = {Analytic number theory},
    	SERIES = {American Mathematical Society Colloquium Publications},
    	VOLUME = {53},
 	PUBLISHER = {American Mathematical Society, Providence, RI},
    	YEAR = {2004},
    	PAGES = {xii+615},
    	ISBN = {0-8218-3633-1},
}

\bib{Kollar}{article}{,
    	AUTHOR = {Koll\'{a}r, J.},
    	TITLE = {Singularities of pairs},
 	BOOKTITLE = {Algebraic geometry---{S}anta {C}ruz 1995},
    	SERIES = {Proc. Sympos. Pure Math.},
    	VOLUME = {62},
    	PAGES = {221--287},
 	PUBLISHER = {Amer. Math. Soc., Providence, RI},
    	YEAR = {1997},
}

\bib{KM98}{book}{,
    AUTHOR = {Koll\'{a}r, J.}
    AUTHOR = {Mori, S.},
     TITLE = {Birational geometry of algebraic varieties},
    SERIES = {Cambridge Tracts in Mathematics},
    VOLUME = {134},
      NOTE = {With the collaboration of C. H. Clemens and A. Corti,
              Translated from the 1998 Japanese original},
 PUBLISHER = {Cambridge University Press, Cambridge},
      YEAR = {1998},
     PAGES = {viii+254},
      ISBN = {0-521-63277-3},
}

\bib{KPS19}{article}{
	author = {Kebekus, S.},
	AUTHOR = {Pereira, J.V.},
	author = {Smeets, A.},
	Title= {Brauer-Manin failure for a simply connected fourfold over a global function field, 
		via orbifold Mordell},
	year = {2019},
	note = {Preprint, arXiv:1905.02795},
}

\bib{Laz04}{book}{,
    AUTHOR = {Lazarsfeld, R.},
     TITLE = {Positivity in algebraic geometry. {I}},
    SERIES = {Ergebnisse der Mathematik und ihrer Grenzgebiete. 3. Folge. A
              Series of Modern Surveys in Mathematics [Results in
              Mathematics and Related Areas. 3rd Series. A Series of Modern
              Surveys in Mathematics]},
    VOLUME = {48},
      NOTE = {Classical setting: line bundles and linear series},
 PUBLISHER = {Springer-Verlag, Berlin},
      YEAR = {2004},
     PAGES = {xviii+387},
      ISBN = {3-540-22533-1},
}


\bib{LeRu13}{article}{
	AUTHOR = {Le Rudulier, C.},
	Title= {Points alg\'ebriques de hauteur born\'ee sur une surface},
	year = {2014},
	note = {Preprint, \url{http://cecile.lerudulier.fr/articles/surfaces.pdf}},
}
	
\bib{LST18}{article}{
	author = {Lehmann, B.},
	AUTHOR = {Sengupta, A.},
	author = {Tanimoto, S.},
	Title= {Geometric consistency of Manin's conjecture},
	year = {2018},
	note = {Preprint,  arXiv:1805.10580},
}

\bib{LT17}{article}{
    	AUTHOR = {Lehmann, B.},
    	author = {Tanimoto, S.},
   	TITLE = {On the geometry of thin exceptional sets in {M}anin's
              conjecture},
   	JOURNAL = {Duke Math. J.},
  	FJOURNAL = {Duke Mathematical Journal},
    	VOLUME = {166},
    	YEAR = {2017},
    	NUMBER = {15},
    	PAGES = {2815--2869},
    	ISSN = {0012-7094},
}
	
\bib{Pey95}{article}{,
    	AUTHOR = {Peyre, E.},
    	TITLE = {Hauteurs et mesures de {T}amagawa sur les vari\'et\'es de {F}ano},
   	JOURNAL = {Duke Math. J.},
  	FJOURNAL = {Duke Mathematical Journal},
    	VOLUME = {79},
    	YEAR = {1995},
    	NUMBER = {1},
    	PAGES = {101--218},
    	ISSN = {0012-7094},
}

\bib{Pey17}{article}{
   	author={Peyre, E.},
   	title={Libert\'{e} et accumulation},
   	journal={Doc. Math.},
   	volume={22},
   	date={2017},
   	pages={1615--1659},
   	issn={1431-0635},
}

\bib{Poonen}{book}{
    AUTHOR = {Poonen, B.},
     TITLE = {Rational points on varieties},
    SERIES = {Graduate Studies in Mathematics},
    VOLUME = {186},
 PUBLISHER = {American Mathematical Society, Providence, RI},
      YEAR = {2017},
     PAGES = {xv+337},
}

\bib{PS20}{article}{
	author = {Pieropan, M.},
	author = {Schindler, D.},
	title = {Hyperbola method on toric varieties},
	year = {2020},
	note = {Preprint, arXiv:2001.09815},
}



\bib{Sal98}{article}{,
    	AUTHOR = {Salberger, P.},
    	TITLE = {Tamagawa measures on universal torsors and points of bounded
              height on {F}ano varieties},
    	NOTE = {Nombre et r\'{e}partition de points de hauteur born\'{e}e (Paris,
              1996)},
   	JOURNAL = {Ast\'{e}risque},
  	FJOURNAL = {Ast\'{e}risque},
    	NUMBER = {251},
    	YEAR = {1998},
    	PAGES = {91--258},
    	ISSN = {0303-1179},
}

\bib{Sen17}{article}{
	AUTHOR = {Sengupta, A.},
	Title= {Manin's {C}onjecture and the {F}ujita invarant of finite covers},
	year = {2017},
	note = {Preprint, arXiv:1712.07780},
}
	
\bib{Ser92}{book}{,
    	AUTHOR = {Serre, J.-P.},
   	TITLE = {Topics in {G}alois theory},
    	SERIES = {Research Notes in Mathematics},
    	VOLUME = {1},
    	NOTE = {Lecture notes prepared by Henri Damon [Henri Darmon],
              With a foreword by Darmon and the author},
 	PUBLISHER = {Jones and Bartlett Publishers, Boston, MA},
    	YEAR = {1992},
    	PAGES = {xvi+117},
    	ISBN = {0-86720-210-6},
}

\bib{ST16}{article}{
    	AUTHOR = {Shalika, J.},
    	author = {Tschinkel, Yu.},
   	TITLE = {Height zeta functions of equivariant compactifications of
              unipotent groups},
   	JOURNAL = {Comm. Pure Appl. Math.},
  	FJOURNAL = {Communications on Pure and Applied Mathematics},
    	VOLUME = {69},
    	YEAR = {2016},
    	NUMBER = {4},
   	PAGES = {693--733},
    	ISSN = {0010-3640},
}

\bib{STBT07}{article}{
    	AUTHOR = {Shalika, J.},
    	author = {Takloo-Bighash, R.},
    	author = {Tschinkel, Yu.},
   	TITLE = {Rational points on compactifications of semi-simple groups},
   	JOURNAL = {J. Amer. Math. Soc.},
  	FJOURNAL = {Journal of the American Mathematical Society},
    	VOLUME = {20},
    	YEAR = {2007},
    	NUMBER = {4},
    	PAGES = {1135--1186},
    	ISSN = {0894-0347},
}

\bib{Smeets}{article}{,
	AUTHOR = {Smeets, A.}, 
	TITLE = {Insufficiency of the \'etale Brauer-Manin obstruction: towards a simply 
		connected example},
	JOURNAL = {Amer. J. Math.},
	FJOURNAL = {American Journal of Mathematics},
	VOLUME = {139},
	YEAR = {2017},
	NUMBER = {2},
	PAGES = {417--431}
}

\bib{Tate}{incollection}{
    	AUTHOR = {Tate, J. T.},
   	TITLE = {Fourier analysis in number fields, and {H}ecke's
              zeta-functions},
 	BOOKTITLE = {Algebraic {N}umber {T}heory ({P}roc. {I}nstructional {C}onf.,
              {B}righton, 1965)},
     PAGES = {305--347},
 	PUBLISHER = {Thompson, Washington, D.C.},
    	YEAR = {1967},
}

\bib{TBT13}{article}{,
   	AUTHOR = {Takloo-Bighash, R.},
   	author =  {Tschinkel, Yu.},
   	TITLE = {Integral points of bounded height on compactifications of
              semi-simple groups},
   	JOURNAL = {Amer. J. Math.},
  	FJOURNAL = {American Journal of Mathematics},
    	VOLUME = {135},
    	YEAR = {2013},
    	NUMBER = {5},
    	PAGES = {1433--1448},
    	ISSN = {0002-9327},
}

\bib{Tenenbaum}{book}{
   	author={Tenenbaum, G.},
   	title={Introduction to analytic and probabilistic number theory},
   	series={Cambridge Studies in Advanced Mathematics},
   	volume={46},
   	publisher={Cambridge University Press, Cambridge},
   	date={1995},
   	pages={xvi+448},
   	isbn={0-521-41261-7},
}

\bib{Tsc03}{incollection}{
    	AUTHOR = {Tschinkel, Yu.},
   	TITLE = {Fujita's program and rational points},
 	BOOKTITLE = {Higher dimensional varieties and rational points ({B}udapest, 2001)},
    	SERIES = {Bolyai Soc. Math. Stud.},
    	VOLUME = {12},
   	PAGES = {283--310},
 	PUBLISHER = {Springer, Berlin},
    	YEAR = {2003},
}
		
\bib{TT12}{incollection}{
    	AUTHOR = {Tanimoto, S.},
    	author=  {Tschinkel, Yu.},
    	TITLE = {Height zeta functions of equivariant compactifications of
              semi-direct products of algebraic groups},
 	BOOKTITLE = {Zeta functions in algebra and geometry},
    	SERIES = {Contemp. Math.},
    	VOLUME = {566},
    	PAGES = {119--157},
 	PUBLISHER = {Amer. Math. Soc., Providence, RI},
    	YEAR = {2012},
}

\bib{TT15}{article}{
    	AUTHOR = {Tanimoto, S.},
    	author = {Tanis, J.},
   	TITLE = {The distribution of {$S$}-integral points on {${\rm
              SL}_2$}-orbit closures of binary forms},
   	JOURNAL = {J. Lond. Math. Soc. (2)},
  	FJOURNAL = {Journal of the London Mathematical Society. Second Series},
    	VOLUME = {92},
    	YEAR = {2015},
    	NUMBER = {3},
    	PAGES = {760--777},
    	ISSN = {0024-6107},
}
\bib{Stacks}{misc}{,
    shorthand    = {Stacks},
    author       = {The {Stacks Project Authors}},
    title        = {\textit{Stacks Project}},
    howpublished = {\url{https://stacks.math.columbia.edu}},
    year         = {2020},
  }

\bib{VanV12}{article}{
    	AUTHOR = {Van Valckenborgh, K.},
    	TITLE = {Squareful numbers in hyperplanes},
   	JOURNAL = {Algebra Number Theory},
  	FJOURNAL = {Algebra \& Number Theory},
    	VOLUME = {6},
    	YEAR = {2012},
    	NUMBER = {5},
   	PAGES = {1019--1041},
   	ISSN = {1937-0652},
}

\bib{Xiao}{article}{
	author = {Xiao, H.},
	title = {Campana points on biequivariant compactifications of the {H}eisenberg group},
	year = {2020},
	note = {Preprint,  arXiv:2004.14763},
}
	
\end{biblist}
\end{bibdiv}
\end{document}